\documentclass[11pt]{amsart}
\usepackage{amsthm,amsfonts,amssymb,amsmath,oldgerm}
\numberwithin{equation}{section}
\usepackage{fullpage}
\usepackage{amsmath}
\usepackage{setspace}
\usepackage{stmaryrd}
\usepackage{nicematrix}
\usepackage{mathrsfs}
\usepackage{fancyhdr}  
\usepackage{graphicx}
\usepackage{enumerate}
\usepackage{bm}
\usepackage{bbm}
\usepackage{dsfont}
\usepackage{multirow}
\usepackage{psfrag}
\usepackage[font=scriptsize]{caption}
\usepackage[font=scriptsize]{subcaption}
\usepackage{listings}
\usepackage{tikz}
\usepackage{empheq}
\usepackage{cases}
\usetikzlibrary{matrix} 
\usepackage[framemethod=TikZ]{mdframed}
\mdfdefinestyle{MyFrame}{backgroundcolor=gray!20!white}
\usepackage[makeroom]{cancel}
\usepackage{cite}

\usepackage[top=20mm,bottom=20mm,left=16mm,right=16mm,a4paper]{geometry}

\usepackage{hyperref,bookmark}

\usepackage[normalem]{ulem}
\definecolor{violet}{rgb}{0.580,0.,0.827}

\usepackage[textwidth= \marginparwidth,textsize=tiny]{todonotes}

\hypersetup{
    colorlinks,          
    citecolor=blue,        
    filecolor=blue,      
    urlcolor=blue,           
    linkcolor=blue,
}

\renewcommand\d{\partial}
\newcommand\dD{\mathrm{d}}

\newcommand\dd{\dD}
\def\eps{\varepsilon }

\newcommand{\beq}{\begin{equation}}
\newcommand{\eeq}{\end{equation}}
\newcommand{\beqa}{\begin{eqnarray}}
\newcommand{\eeqa}{\end{eqnarray}}
\newcommand\br{\begin{remark}}
\newcommand\er{\end{remark}}
\newcommand\bp{\begin{pmatrix}}
\newcommand\ep{\end{pmatrix}}
\newcommand{\be}{\begin{equation}}
\newcommand{\ee}{\end{equation}}
\newcommand\ba{\begin{equation}\begin{aligned}}
\newcommand\ea{\end{aligned}\end{equation}}
\newcommand\ds{\displaystyle}

\newcommand{\beg}{\begin{example}}
\newcommand{\eeg}{\end{exaplem}}
\newcommand{\bpr}{\begin{proposition}}
\newcommand{\epr}{\end{proposition}}
\newcommand{\bt}{\begin{theorem}}
\newcommand{\et}{\end{theorem}}
\newcommand{\bc}{\begin{corollary}}
\newcommand{\ec}{\end{corollary}}
\newcommand{\bl}{\begin{lemma}}
\newcommand{\el}{\end{lemma}}
\newcommand{\bd}{\begin{definition}}
\newcommand{\ed}{\end{definition}}
\newcommand{\brs}{\begin{remarks}}
\newcommand{\ers}{\end{remarks}}
\newcommand{\mycomment}[1]{}

\newtheorem{theorem}{Theorem}[section]
\newtheorem{proposition}[theorem]{Proposition}
\newtheorem{corollary}[theorem]{Corollary}
\newtheorem{lemma}[theorem]{Lemma}
\newtheorem{remark}[theorem]{Remark}
\newtheorem{definition}[theorem]{Definition}

\newtheorem{example}[theorem]{Example}

\newcommand{\N}{{\mathbb N}}

\newcommand{\R}{{\mathbb R}}

\newcommand\bx{{\mathbf x}}
\newcommand{\bz}{\mathbf z}
\newcommand{\by}{\mathbf y}

\newcommand{\bu}{\mathbf u}
\newcommand\bv{{\mathbf v}}
\newcommand\bw{{\mathbf w}}
\newcommand\Div{{\rm div}}

\newcommand\bA{{\mathbf A}}
\newcommand\bB{{\mathbf B}}

\newcommand\bE{{\mathbf E}}
\newcommand\bF{{\mathbf F}}

\newcommand\bK{{\mathbf K}}
\newcommand\bL{{\mathbf L}}

\newcommand\bU{{\mathbf U}}

\newcommand\cE{{\mathcal E}}

\newcommand\cO{{\mathcal O}}

\usetikzlibrary{fit}
\usetikzlibrary{arrows}
\numberwithin{equation}{section}
\numberwithin{figure}{section}
\hypersetup{linkbordercolor=red}
\usepackage{lineno}
\usepackage{xcolor}

\title[Crank-Nicolson schemes for the Vlasov-Poisson
system]{A modified Crank-Nicolson scheme for the Vlasov-Poisson system with a strong external magnetic field}



\author{Francis Filbet}
\address{Universit\'e de, Toulouse , Institut de Math\'ematiques de Toulouse, France}
\email{francis.filbet@math.univ-toulouse.fr}
\thanks{}

\author{L.~Miguel Rodrigues}
\address{Université de Rennes, CNRS, IRMAR - UMR 6625, F-35000 Rennes, France}
\email{luis-miguel.rodrigues@univ-rennes.fr}
\thanks{Research of L.M.R. was partially supported by the ANR Project HEAD ANR-24-CE40-3260 and the Institut Universitaire de France.}

\author{Kim Han Trinh}
\address{Université de Rennes, CNRS, IRMAR - UMR 6625, F-35000 Rennes, France}
\email{kim-han.trinh@univ-rennes.fr}
\thanks{Research of K.H.T. was partially supported by the Institut Universitaire de France.}

\begin{document}


\maketitle

\bigskip

\begin{abstract}
We propose and study  a Particle-In-Cell (PIC) method utilizing
Crank-Nicolson time discretization for the Vlasov-Poisson system with
a strong, inhomogeneous external magnetic field with fixed direction.
Our focus is on particle dynamics in the plane orthogonal to the
magnetic field.
 In this regime, traditional explicit schemes are constrained by
 stability conditions linked to the small Larmor radius and plasma
 frequency \cite{RiCh20}. To avoid this limitation, our approach is based on  numerical
schemes \cite{FiRo16, FiRo17, FiRo23},  providing a consistent PIC
discretization of the guiding-center system taking into account
variations of the magnetic field. We carry out some theoretical proofs and perform several numerical experiments to validate the method demonstrating its robustness and accuracy.
\end{abstract}

\vspace{0.1cm}

\tableofcontents

%
%

\section{Introduction}
\label{sec:1}
This paper focuses on plasma confinement in the presence of a strong,
spatially varying external magnetic field, where charged particles evolve under the influence of both electrostatic forces and intense
magnetic confinement. Such configurations are characteristic of
tokamak plasmas \cite{bellan_2006_fundamentals, miyamoto_2006_plasma},
where the magnetic field plays a crucial role in containing particles
within the core of the device. Kinetic models, which provide a
mesoscopic description of charged particle dynamics, are highly
accurate and essential tools for investigating the behavior of
thermonuclear fusion plasmas.
\\
 We assume that collective effects dominate, and the plasma is modeled
 entirely through \textcolor{red}{kinetic} equations. The primary unknown is the
 particle number density $f\equiv f(t,\bx,\bv)$, which depends on time
 $t \geq 0$, position $\bx\in\Omega\subset \R^d$, and velocity
 $\bv\in\R^d$, with $d \geq 2$.  Its behaviour is given by
the Vlasov equation, 
\begin{equation}
\label{eq:vlasov} 
\frac{\partial f}{\partial t}\,+\,\bv\cdot\nabla_\bx f \,+\,\mathbf{F}(t,\bx,\bv)\cdot\nabla_\bv f \,=\, 0,
\end{equation}
where the force field $F(t,\bx,\bv)$ is coupled with the distribution function $f$ giving a nonlinear system.

Here, we  consider the two-dimensional case where  the magnetic field acts in the vertical direction and
only depends on $\bx=(x_1,x_2)\in \R^2$, that is,  
\[
\bB(\bx) \,\,=\,\, \frac{1}{\eps} \, \left( \begin{array}{l}0\\0\\b(\bx)\end{array}\right)\,, 
\]
where the function  $b$ describes the variations of the amplitude with $b \in
W^{1,\infty}(\mathbb{R}^{2})$  and
\begin{equation}
  \label{eq:1.3}
b(\bx)\geq b_0>0\,.
\end{equation}
The number $\eps>0$ is a small parameter related
to the ratio between the reciprocal Larmor frequency and the advection
time scale (see \cite{DeFi16,HaMe03, HaWa18} and the references therein for more
details on scalings).

We will focus on the long-time behavior of  positive ions in the
orthogonal plane to the external magnetic field. Therefore, the
distribution function  $f_\eps$ is a solution to
the Vlasov equation coupled with the Poisson equation for the
electrical potential $\phi_\eps$ generated by the motion of these
charged particles, that is,
\begin{flalign} \label{eq:VP_system}
    \begin{dcases}
        \eps \dfrac{\partial f_{\eps}}{\partial t} \,+\, \mathbf{v} \cdot \nabla_{\mathbf{x}} f_{\eps} \,+\, \left( \mathbf{E} _{\eps}(t, \mathbf{x}) - b(\mathbf{x}) \,\dfrac{\mathbf{v}^{\perp}}{\eps} \right) \cdot \nabla_{\mathbf{v}} f_{\eps} \,=\, 0\,, \\
        \mathbf{E}_\eps \,=\, - \nabla_{\mathbf{x}} \phi_\eps\,, \qquad - \Delta_{\mathbf{x}} \phi_\eps\,=\, \rho_\eps\,,
    \end{dcases}
  \end{flalign}
  where $\bv^\perp\,=\,(-v_2,v_1)\in\R^2$ and the density  $\rho_\eps$ is given by
  \[
\rho_\eps( t, \mathbf{x}) \,:=\, \displaystyle \int_{\mathbb{R}^{2}} f_{\eps}(t,\bx,\bv) \,\dD\bv\,.
  \]

Here we aim to  construct numerical approximations for the
Vlasov-Poisson system \eqref{eq:VP_system} using particle methods (see
\cite{BiLa85}), which involve  in
approximating the distribution function by a finite number of
macro-particles. The trajectories of these particles are determined from
the characteristic curves associated  to the Vlasov equation
\begin{flalign} \label{eq:ODE_system}
    \begin{dcases}
        \eps \dfrac{\dD \mathbf{x}_{\eps} }{\dD t} = \mathbf{v}_{\eps}\,, \\[0.8em]
        \eps \dfrac{\dD \mathbf{v}_{\eps} }{\dD t} = \mathbf{E}_\eps(t, \mathbf{x}_{\eps}) - b(\mathbf{x}_{\eps}) \dfrac{ \mathbf{v}_{\eps}^{\perp}}{\eps}\,, \\[0.8em]
        \mathbf{x}_{\eps}(0) = \mathbf{x}_\eps^{0}\,, \qquad \mathbf{v}_{\eps}(0) = \mathbf{v}_\eps^{0}\,,
    \end{dcases}
  \end{flalign}
 then  we use the conservation of $f_\eps$ along the characteristic
  curves, that is,
  \begin{flalign}
    \nonumber
    f_{\eps}(t, \mathbf{x}_{\eps}(t), \mathbf{v}_{\eps}(t)) \, =\, f_{\eps} (t^{0}, \mathbf{x}^{0}_{\eps}, \mathbf{v}^{0}_{\eps})\,.
  \end{flalign}

  In particular, we will focus on the construction of numerical
  schemes for the ODE system \eqref{eq:ODE_system}, where the time
  step $\Delta t$ is arbitrary and free from any stability
  constraint. Following the work of Filbet and Rodrigues \cite{FiRo17,
    FiRo20, FiRo23}, the ODE system can be decomposed into fast
  dynamics, driven by the fast variable $\bv_{\eps}$, and slow
  dynamics, governed by the variables $(\bx_{\eps}, e_{\eps})$,
  where $e_{\eps} = \frac{1}{2} | \bv_{\eps} |^{2}$. This
  decomposition allows the design of a class of numerical schemes that
  precisely capture slow-scale variables, while faster scales are
  correctly filtered. More precisely, when the intensity of the magnetic
  field is sufficiently large, {\it i.e.} when $\eps \ll 1$, the scheme
  provides a consistent approximation to the asymptotic model
  \cite{FiRo17}.

\subsection{Formal asymptotic behavior for a given electromagnetic field}
Before describing  a numerical scheme  for the nonlinear
Vlasov-Poisson system \eqref{eq:VP_system}, we first briefly expound on what may be expected from the continuous model  with
a given electric field in the limit $\eps\to 0$. For this
purpose, we consider a  function $\phi \in
W^{3, \infty}$ such that, for all $\mathbf{x} \in \mathbb{R}^{2}$,
$\mathbf{E}(\mathbf{x}) = - \nabla_{\mathbf{x}} \phi(\mathbf{x})$ and observe that
the system \eqref{eq:ODE_system} has an Hamiltonian
structure associated with the total  energy  $\mathcal{E}_{\eps}(t) $,
\begin{flalign}
  \label{total_energy}
 \mathcal{E}_{\eps}(t) = \dfrac{\| \mathbf{v}_{\eps}(t) \|^{2}}{2} + \phi(\mathbf{x}_{\eps}(t)), \qquad t \geq 0\,,
\end{flalign}
which is an invariant of the system.  Therefore,  to study the limit $\eps \rightarrow
0$,  we first define the  kinetic energy as a slow scale variable
\[
e_{\eps}(t) \,:=\, \frac{1}{2} \| \mathbf{v}_{\eps}(t)
\|^{2}\,,
\]
leading to the study  of the augmented system 
\begin{flalign} \label{eq:augmented_ODE_system}
    \begin{dcases}
        \dfrac{\dD \mathbf{x}_{\eps}}{\dD t} \,=\, \frac{\mathbf{v}_{\eps}}{\eps}\,, \\[0.9em]
        \dfrac{\dD e_{\eps}}{\dD t} \;=\, \frac{1}{\eps}\,\mathbf{E}(
        \mathbf{x}_{\eps}) \cdot \mathbf{v}_{\eps}\,,
            \end{dcases}
  \end{flalign}
  still coupled with the equation on $\bv_\eps$
  \begin{equation} \label{eq:v}
\eps \dfrac{\dD \mathbf{v}_{\eps}}{\dD t} \,=\, \mathbf{E}(\mathbf{x}_{\eps})
\,-\, b(\mathbf{x}_{\eps}) \,\dfrac{\mathbf{v}_{\eps}^{\perp}}{\eps}\,,
\end{equation}
which describes the fastest scale. Of course the second equation of
\eqref{eq:augmented_ODE_system} is a consequence of \eqref{eq:v}  but it retains only its slower part. 

Following \cite{FiRo20}, one may prove that $\mathbf{x}_{\eps}(t) \rightarrow \mathbf{y}(t)$ and
$e_{\eps}(t) \rightarrow g(t)$,  as $\eps \rightarrow 0$ where
$(\by,g)$ is solution to the so-called guiding center system,
\begin{flalign} \label{eq:guiding_center_system}
    \begin{dcases}
        \dfrac{\dD\by }{\dD t} \,=\, - \dfrac{\mathbf{E}^{\perp}}{b} (\mathbf{y}) \,+\, g\, \dfrac{\nabla_{\mathbf{y}}^{\perp} b}{b^{2}}(\mathbf{y})\,, \\[0.9em]
        \dfrac{\dD g }{\dD t}  \,=\, g \,\mathbf{E} \cdot \dfrac{\nabla_{\mathbf{y}}^{\perp} b}{b^{2}}(\mathbf{y})\,.
    \end{dcases}
\end{flalign}
For the convenience of the reader we provide in Appendix~\ref{s:app} the main formal computations leading to \eqref{eq:guiding_center_system}.

Furthermore, we may identify the limit for the  total energy $\mathcal{E}_\eps (t)$, as $\eps \rightarrow 0$, this
quantity converging to
\begin{flalign}
   \nonumber
    \mathcal{E}_{gc}(t) \,:=\, g(t) \,+\, \phi(\mathbf{y}(t))\,, 
\end{flalign}
which is indeed an invariant of the guiding center model \eqref{eq:guiding_center_system},
\begin{flalign}
  \label{energy_conserve} 
    \dfrac{\dD \mathcal{E}_{gc} }{\dD t} \,=\, \dfrac{\dD g}{\dD t}  \,+\, \nabla_{\mathbf{y}} \phi(\mathbf{y}) \cdot \dfrac{\dD\by}{\dD t} \,=\, g \,\mathbf{E} \cdot \dfrac{\nabla_{\mathbf{y}}^{\perp} b}{b^{2}} (\mathbf{y}) \,-\, g \,\mathbf{E} \cdot \dfrac{\nabla_{\mathbf{y}}^{\perp} b}{b^{2}} (\mathbf{y}) \,=\, 0\,.
\end{flalign}
Furthermore,  we may define the magnetic moment as
\[ 
\mu_{gc} (t) \,=\,  \dfrac{g}{b(\mathbf{y})}\,, 
\]
which  is an invariant for the guiding center system
\eqref{eq:guiding_center_system} without counterpart for the original
\eqref{eq:ODE_system}, thus called an adiabatic invariant for
\eqref{eq:guiding_center_system}. Indeed, we have 
\begin{flalign} \label{adiabatic_conserve}
	\dfrac{\dD \mu_{gc}}{\dD t} \,=\, \dfrac{\dD }{\dD t} \left(  \dfrac{g}{b(\mathbf{y})} \right) \,=\,  g\, \mathbf{E} \cdot \dfrac{\nabla_{\mathbf{y}}^{\perp} b}{b^{3}}(\mathbf{y}) \,-\, g \dfrac{\nabla_{\mathbf{y}} b}{b^{2}} (\mathbf{y}) \cdot \left( - \dfrac{\mathbf{E}^{\perp}}{b} (\mathbf{y}) + g  \dfrac{\nabla_{\mathbf{y}}^{\perp} b}{b^{2}} (\mathbf{y}) \right) \,=\, 0\,.
\end{flalign}

Reproducing these properties at the discrete level is a target when designing a scheme for  \eqref{eq:ODE_system} preserving  asymptotics when $\eps$ tends to zero.

\subsection{Formal asymptotic limit of the Vlasov-Poisson system}

We come back to the Vlasov-Poisson system \eqref{eq:VP_system}. Here one cannot anymore remain completely at the characteristic level \eqref{eq:ODE_system}. Moreover whereas arguments of the previous subsection could be turned into sound analytic arguments (by slight variations on \cite{FiRo20}), to the best of our knowledge the present situation does not fall directly into the range of the actually available analysis of gyro-kinetic limits. We refer the reader to the introductions of \cite{FiRo20,Vu} and references therein for a representative sample of such analytic techniques.

Nevertheless the known results and the previous subsection strongly suggests for $(f^\eps,\bE^\eps)$ solving the
Vlasov-Poisson system \eqref{eq:VP_system} that in the limit
$\eps\rightarrow 0$, the electric field $\bE^\eps$ and the following
velocity-averaged version of $\bar{F}^\eps$ 
\[
\bar{F}^\eps\,:\,(t,\bx,e)\mapsto\frac{1}{2\pi}\int_0^{2\pi} f^\eps(t,\bx, \sqrt{2e}(\cos(\theta),\sin(\theta)))\,\dd \theta
\]
converge to some $\bE:(t,\by)\mapsto \bE(t,\by)$ and some\footnote{We use distinct notation of variables for limiting functions to be consistent with asymptotic analysis at the characteristic level. This is of course completely immaterial.} $f:(t,\by,g)\mapsto f(t,\by,g)$ solving the following system consisting in a  \textcolor{red}{kinetic} equation supplemented with a Poisson equation, 
\begin{equation}
\left\{
\begin{array}{l}
 \ds \frac{\d f}{\d t}\,+\,\mathbf{U}\cdot\nabla_{\by} f\,+\, u_g \frac{\d f}{\d g}\,=\,0,
\\ \, \\
\ds -\Delta_{\by}\phi\,=\,\rho\,,\quad
\rho=2\pi\,\int_{\R^+} f \,\dD g,
\end{array}
\right.
   \label{eq:gc}
  \end{equation}
where the velocity field is given by
\[
{\bf U}(t,\by,g)\,=\,\bF(t,\by)  \,+\, g \,\frac{\nabla_\by^\perp b}{b^2}(t,\by)\,, \qquad
u_g = -\Div_\by(\bF)(t,\by) \,g\,,
\]  
with $\bE = -\nabla_\by \phi$, $\bF=-\bE^\perp/b$. {We remind the reader that $\bU$ contains two classical components of the guiding center velocity, the $\bE\times\bB$ drift and the grad $\bB$ drift.}

\subsection{Particle methods for the Vlasov-Poisson system}

To make the most of the previous discussions in order to discretize
the Vlasov equation \eqref{eq:VP_system}, particle methods are
particularly well suited since they directly involve an approximation
of the characteristic curves \eqref{eq:ODE_system}. Here, we will
consider the Particle-In-Cell (PIC) method, in which trajectories are
computed {\it via} the characteristic curves \eqref{eq:ODE_system}, while the self-consistent electric field is calculated using the Poisson equation on a grid of the physical space. We refer the reader to \cite{BiLa85, DeDe17} or \cite{FiRo16} for a brief review of particle methods.

To keep the notation as concise as possible, we temporarily omit the
dependence of solutions on $\eps$. The starting point is the
approximation of the solution $f$, which solves \eqref{eq:VP_system},
by a finite sum of smoothed functions --- viewed as macro particles. More
explicitly, in dimension $d$, one computes
\begin{flalign*}
	f_{N}(t, \mathbf{x}, \mathbf{v}) = \sum \limits_{1 \leq k \leq N} \omega_{k} \,\varphi_{\alpha} \left( \mathbf{x} - \mathbf{x}_{k}(t) \right) \otimes\varphi_{\alpha} \left( \mathbf{v} - \mathbf{v}_{k}(t) \right)\,,
\end{flalign*}
where $\varphi_{\alpha} = \alpha^{-d} \varphi(\cdot / \alpha)$ is a
particle shape function with radius proportional to $\alpha$ --- usually
seen as an approximation of the Dirac measure $\delta_{0}$ ---
obtained by rescaling a fixed compactly supported mollifier $\varphi$
whereas the set $(\mathbf{x}_{k}, \mathbf{v}_{k})_{1 \leq k \leq N}$
represents the position in phase space of $N$ macro-particles evolving
along characteristic curves \eqref{eq:ODE_system} from the initial
data $(\mathbf{x}^{0}_{k}, \mathbf{v}^{0}_{k}), 1 \leq k \leq N$. More
explicitly,  $(\mathbf{x}_{k}, \mathbf{v}_{k})_{1 \leq k \leq N}$ is
solution to 
\begin{flalign*}
	\begin{dcases}
		\eps \dfrac{\dD \mathbf{x}_{k}}{\dD t} = \mathbf{v}_{k}, \\
		\eps \textcolor{red}{\dfrac{\dD \mathbf{v}_{k}}{\dD t}} \,=\, \mathbf{E}(t, \mathbf{x}_{k}) - b(t, \mathbf{x}_{k}) \dfrac{\mathbf{v}_{k}^{\perp}}{\eps}\,, \\[0.9em]
		\mathbf{x}_{k}(0)\,=\, \mathbf{x}_{k}^{0}, \quad  \mathbf{v}_{k}(0)\,=\, \mathbf{v}_{k}^{0}\,,
	\end{dcases}
\end{flalign*} 
where the electric field $\mathbf{E}$  is computed by discretizing the
Poisson equation on a mesh of the physical space.

\textcolor{red}{Here, we deliberately choose to use a classical PIC method with $P_1$
reconstruction for the density and electric field in order to focus on
resolving particle trajectories with a time step large compared to the
$\varepsilon$ scale parameter. Our approach can be easily extended to much
more sophisticated PIC methods. For example, forward-backward
Lagrangian methods \cite{MCPFC} reduce fluctuations in the density
reconstruction step.} In recent years, other advanced methods have been developed to
improve the stability properties of the Particle-In-Cell (PIC)
method\footnote{For a discussion of some other classes of methods, we
  refer to the introduction of \cite{FiRo23} and some references
  therein.} in the presence of a large, inhomogeneous external
magnetic field. Among these, the earliest schemes were introduced by
Boris \cite{Boris70, BiLa85} for relativistic plasma simulation. It is
a second-order explicit method, often referred to as an explicit PIC
method, employing a time-centered electromagnetic field and an
averaged phase-space representation $(\mathbf{x}, \mathbf{v})$ for
\eqref{eq:ODE_system}. Later, this scheme was \textcolor{red}{applied} by Parker and
Birdsall \cite{PaBi91} to address the high magnetic field regime,
aiming to accurately capture drift motions of particles in three
dimensions. However, these standard explicit PIC approaches still
suffer from temporal numerical stability constraints imposed by the
Courant-Friedrichs-Lewy (CFL) condition \cite{BiLa85}. As a result,
despite their simplicity and computational efficiency, these schemes
are significantly constrained in high-field regimes.

To overcome this lack of stability, several implicit PIC schemes have
been developed to solve \eqref{eq:ODE_system} and to capture grad $\bB$
drifts in strongly magnetized plasma. We refer to Brackbill, Forslund,
and Vu \cite{BrFo85, VuBr95}, who introduce an effective force into
the velocity equation such that the scheme remains consistent with
\eqref{eq:ODE_system} for small time steps. The scheme is formulated
in a fully implicit manner as a modified version of the Crank-Nicolson
scheme. Alternatively, the magnetized implicit (MI) scheme proposed by
Genoni, Clark, and Welch \cite{GeCl10} employs a two-step
predictor-corrector approximation. However, these schemes overlook the
role of kinetic energy, which significantly contributes to particle
motion when $\varepsilon \rightarrow 0$, as shown in
\eqref{eq:guiding_center_system}. Consequently, these schemes fail to
capture the correct regime when, for a fixed time step, $\varepsilon
\rightarrow 0$.

More recently,  Ricketson, Chac\'on and Chen \cite{RiCh20,
  ChCh23} built upon the Crank-Nicolson scheme with an additional
effective force designed to achieve two objectives. First, they
addressed the challenges arising in the regime $\eps \ll 1$, by
capturing  grad $\bB$ drifts. Second, the additional force is designed
to conserve energy for all $\eps > 0$. However,  this method still
requires to adapt the time step when $\eps$ becomes small. In parallel, building on \cite{BoFi16},
Filbet and Rodrigues proposed a class of semi-implicit methods (IMEX)
where the position is updated explicitly whereas velocity is treated
implicitly \cite{FiRo16,FiRo17, FiRo23}. These schemes are developed to solve
the augmented system \eqref{eq:augmented_ODE_system}, which introduces
additional variable to separate slow scale and fast scale dynamics. As
a result, in the regime  $\eps \rightarrow 0$ and for a fixed time
step,  IMEX schemes accurately describe the dynamics of both position
and kinetic energy. This ensures a consistent approximation  to the
guiding center system \eqref{eq:augmented_ODE_system}.

\textcolor{red}{
Another strategy has been developed to accurately
capture the dynamics of \eqref{eq:ODE_system} by following fast
oscillations. This approach works well  when the magnetic field is
constant or when it varies slowly. For instance, the two-scaled
formulation method, proposed in \cite{CrLe17}, employs two time variables to split fast and
slow scales. Additionally,  a class of Lie-Trotter type splitting
schemes coupled with  exponential integrators has  been developed by
Wang and Zhao \cite{WaZha21} to provide a first order approximation with respect to $\eps$. These  schemes are very successful for uniform or
slowly varying magnetic field, but require a deeply understanding of
the fast oscillations. }

In this article, we propose to delve deeper and extend the strategy
already proposed by Filbet and Rodrigues \cite{FiRo16, FiRo17}
for Crank-Nicolson-type schemes. On the one hand, these schemes are
widely recognized in the computational physics community and are
valued for their effective energy preservation \cite{BrFo85, VuBr95,
  RiCh20}. On the other hand, it is important to note that the schemes
proposed by Filbet and Rodrigues, relying on IMEX methods that are
more dissipative, can sometimes compromise their accuracy for
intermediate values of the parameter $\varepsilon$. The present work
aims to maximize their efficiency and robustness by applying
Crank-Nicolson schemes to the augmented system
\eqref{eq:augmented_ODE_system} to separate slow and fast scale
dynamics. Moreover, this numerical scheme is implemented within the
PIC framework for long-term simulations of the Vlasov-Poisson system
\eqref{eq:VP_system}.

The rest of the paper is organized as follows. In Section
\ref{sec:Review_schemes}, we recall and analyze several numerical
schemes based on the Crank-Nicolson scheme, including those developed
by Brackbill, Forslund, and Vu \cite{BrFo85, VuBr95} as well as
Ricketson and Chacón \cite{RiCh20}. These schemes will be
analyzed in the asymptotic limit $\varepsilon \rightarrow 0$ with a
fixed time step to clarify the importance of decomposing the solution
into fast and slow variables, as in \cite{FiRo16, FiRo17}. In Section
\ref{sec:Modified_CN}, we propose a new numerical scheme built upon
the Crank-Nicolson scheme, following the strategy of Filbet and
Rodrigues in \cite{FiRo17}, and we examine its accuracy in the regime
$\varepsilon \rightarrow 0$. Finally, in Section
\ref{sec:Numerical_simulation}, we present numerical experiments for
the new scheme, both for the computation of single-particle motion and
as a particle pusher within the PIC framework for the Vlasov-Poisson
system.

\bigskip

\noindent {\bf Acknowledgement.} KHT expresses his appreciation of the hospitality of IMT, Universit\'e Toulouse III, during the preparation of the present contribution. FF and LMR are grateful to Luis Chac\'on for stimulating discussions that have motivated the present piece of work.


\section{Review of Crank-Nicolson schemes}
\label{sec:Review_schemes}
In this section, we aim to discuss the Crank-Nicolson method, applied in the framework of Particle-In-Cell methods for the Vlasov-Poisson system. In \cite{SiTa92}, the authors show that the Crank-Nicolson scheme is second-order accurate, unconditionally stable, and energy-conserving for quadratic potentials when considering the system for a single particle motion \eqref{eq:ODE_system}. Later, this scheme has been studied for the Vlasov-Poisson system \cite{VuBr95, ChCh23, RiCh20}.

Here, we will review different modified Crank-Nicolson schemes
proposed in the literature and discuss their conservation properties
and asymptotic behavior when $\varepsilon$ approaches zero, that is,
when the external magnetic field becomes large. Our aim is to
investigate the consistency of the numerical approximation with the
guiding center model \eqref{eq:guiding_center_system} at the discrete
level in the limit as $\varepsilon \to 0$.

Let us start with the original Crank-Nicolson scheme and consider a time step $\Delta t >0$ and  $t^n=n\,\Delta t$, for $n\in\N$, we
define $(\bx^n_\eps, \bv^n_\eps)$ as an approximation of the solution
$(\bx_\eps, \bv_\eps)$ to \eqref{eq:ODE_system}. Applying the Crank-Nicolson scheme, the
sequence $(\bx_\eps^n,\bv_\eps^n)_{n\in\N}$ is given by 
\begin{equation}
  \label{scheme:CNS}
    \begin{dcases}
        \eps\, \dfrac{\mathbf{x}^{n+1}_{\eps} - \mathbf{x}^{n}_{\eps}}{\Delta t} \,=\, \mathbf{v}^{n+1/2}_{\eps}\,, \\[0.9em]
        \eps\, \dfrac{\mathbf{v}^{n+1}_{\eps} - \mathbf{v}^{n}_{\eps}}{\Delta t} \,=\, \mathbf{E}(\mathbf{x}^{n+1/2}_{\eps}) \,-\, b(\mathbf{x}^{n+1/2}_{\eps}) \,\dfrac{ ( \mathbf{v}^{n+1/2}_{\eps} )^{\perp}}{\eps}\,, 
    \end{dcases}
\end{equation}
where 
\begin{equation}
    \nonumber
    \mathbf{v}^{n+1/2}_{\eps} \,=\, \dfrac{\mathbf{v}^{n+1}_{\eps} +
      \mathbf{v}^{n}_{\eps}}{2} \qquad{\rm and}\qquad \mathbf{x}^{n+1/2}_{\eps} \,=\, \dfrac{\mathbf{x}^{n+1}_{\eps} + \mathbf{x}^{n}_{\eps}}{2}\,.  
    \end{equation}
First, it is worth mentioning that  this scheme provides a good
approximation of the energy for a wide range  of values of 
  $\eps$. To be precise, assume that the electric field derives from a
  given smooth potential $\phi$, hence  we have $\mathbf{E} = -
  \nabla_{\mathbf{x}} \phi$ and  the discrete kinetic energy is
  defined as
  $$
  e^{n}_{\eps}
  = \frac{1}{2} \| \mathbf{v}^{n}_{\eps} \|^{2}.
  $$
  The total discrete
  energy is then
\begin{flalign}
    \mathcal{E}^{n}_{\eps} \,=\, e^{n}_{\eps} + \phi(\bx_\eps^{n}), \qquad n \geq 0\,.
\end{flalign}
From \eqref{scheme:CNS}, it follows that the variation of the  total energy is given by
\begin{flalign} \label{eq:te_transform_CN}
    \dfrac{\mathcal{E}^{n+1}_{\eps} -
      \mathcal{E}^{n}_{\eps}}{\Delta t} & \,=\,
    \dfrac{e^{n+1}_{\eps} - e^{n}_{\eps}}{\Delta t} +
    \dfrac{\phi(\bx_{\eps}^{n+1}) -
      \phi(\bx_{\eps}^{n})}{\Delta t}\,,
    \\[0.9em]
    \nonumber
    & \,=\, - \nabla_{\mathbf{x}} \phi(\bx_\eps^{n+1/2}) \cdot \dfrac{\mathbf{x}_{\eps}^{n+1} - \mathbf{x}_{\eps}^{n}}{\Delta t} + \dfrac{\phi(\bx_{\eps}^{n+1}) - \phi(\bx_{\eps}^{n})}{\Delta t}\,.
\end{flalign}
where $\mathbf{E}(\mathbf{x}^{n+1/2}_{\eps}) = - \nabla_{\mathbf{x}}
\phi(\mathbf{x}_{\eps}^{n+1/2})$. Therefore, as is well-known, this scheme  conserves the
discrete energy only for quadratic  potentials and for more general potential $\phi \in W^{3, \infty}$,  a Taylor expansion yields
\begin{flalign} \label{eq:pe_CN}
    \phi(\bx^{n+1}_\eps) - \phi(\bx^{n}_\eps) \,=\, \nabla_{\mathbf{x}} \phi(\bx^{n+1/2}_{\eps}) \cdot (\mathbf{x}_{\eps}^{n+1} - \mathbf{x}_{\eps}^{n}) \,+\, \Delta t^{3} \,\mathcal{O} \left( \left \| \dfrac{\mathbf{x}_{\eps}^{n+1} - \mathbf{x}_{\eps}^{n}}{\Delta t} \right \|^{3} \right)\,.
\end{flalign}
thus
\begin{flalign}
  \label{eq:te_CN}
    \dfrac{\mathcal{E}^{n+1}_{\eps} - \mathcal{E}^{n}_{\eps}}{\Delta t} \,=\, \Delta t^{2}\, \mathcal{O}\left( \left\|\dfrac{\mathbf{x}_{\eps}^{n+1} - \mathbf{x}_{\eps}^{n}}{\Delta t} \right\|^{3} \right)\,.
\end{flalign}
In other words, under our above assumptions and the further reasonable assumption that $\mathbf{x}_{\eps}^{n}$ is bounded, the variations of the total discrete  energy is of order $\Delta t^2$, which endows this scheme with a form of stability for large time simulations and for all  $\eps >0$.

Now, concerning the asymptotic limit of the scheme
\eqref{scheme:CNS} as $\eps$ tends to zero, we have the following result.

\begin{proposition}[Asymptotic behavior $\eps\rightarrow 0$ with
  a fixed $\Delta t$]
  \label{pro:CN}
Let $\phi \in W^{3, \infty}(\mathbb{R}^{2})$, choose a sufficiently small fixed time step $\Delta t$ and a final time $T > 0$.  We set  $N_{T} = \left
   \lfloor T/\Delta t \right \rfloor$. 
Assume that the Crank-Nicolson scheme \eqref{scheme:CNS} defines a numerical approximation
 $(\mathbf{x}^{n}_{\eps}, \mathbf{v}^{n}_{\eps})_{0 \leq n \leq N_{T}}$ satisfying
    \begin{itemize}
        \item[$(i)$] for all $1 \leq n \leq N_{T}$, $\mathbf{x}^{n}_{\eps}$ is uniformly bounded with respect to $\eps > 0$\,;
        \item[$(ii)$] in the limit $\eps \rightarrow 0$, $(\mathbf{x}^{0}_{\eps}, \frac{1}{2} \| \mathbf{v}^{0}_{\eps} \|^{2})$ converges to some $(\mathbf{y}^{0}, g^{0})$\,.
    \end{itemize}
    Then, we have
    \begin{itemize}
      \item for all $1 \leq n \leq N_{T}$, $(\mathbf{x}_{\eps}^{n},
        e^{n}_{\eps})$ converges to $(\mathbf{y}^{n}, g^{n})$, as $\eps \rightarrow 0$ with $e^{n}_{\eps} = \frac{1}{2} \| \mathbf{v}^{n}_{\eps} \|^{2}$  and the limit $(\mathbf{y}^{n}, g^{n})_{1 \leq n \leq N_{T}}$ solves
    \begin{equation} \label{eq:guiding_center_scheme_CN}
        \begin{dcases}
            \dfrac{\mathbf{y}^{n+1} - \mathbf{y}^{n}}{\Delta t} = - \dfrac{\mathbf{E}^{\perp}}{b} (\mathbf{y}^{n+1/2})\,, \\
            \dfrac{g^{n+1} - g^{n}}{\Delta t} = 0\,;
        \end{dcases}
    \end{equation}
   \item  for all $1 \leq n \leq N_{T}$, the total energy $\mathcal{E}^{n}_{\eps} = e^{n}_{\eps} + \phi(\mathbf{x}^{n}_{\eps})$
   converges to $\mathcal{E}^{n}_{gc} := g^{n} + \phi(\mathbf{y}^{n})$ as
   $\eps \rightarrow 0$, which satisfies
    \begin{flalign} \label{eq:te_limit_CN}
        \dfrac{\mathcal{E}^{n+1}_{gc} -  \mathcal{E}^{n}_{gc}}{\Delta t} = \mathcal{O}\left( \Delta t^{2}  \right)\,;
      \end{flalign}
\item defining the discrete  magnetic moment
$\mu^{n}_{\eps}$ as \[
\mu^{n}_{\eps} =
\dfrac{e^{n}_{\eps}}{b(\mathbf{x}_{\eps}^{n})},
\]
$(\mu_\eps^n)_{\eps>0}$ converges to
$\mu^{n}_{gc} := \dfrac{g^{n}}{b(\mathbf{y}^{n})}$ as $\eps \rightarrow
0$ such that
    \begin{flalign} \label{eq:mu_limit_CN}
        \dfrac{\mu^{n+1}_{gc} -  \mu^{n}_{gc}}{\Delta t} \,=\, - g^{0}
        \, \dfrac{\mathbf{E}\cdot\nabla_{\mathbf{y}}^{\perp} b}{b^{3}}
        (\mathbf{y}^{n+1/2}) \,+\,   \mathcal{O}\left( \Delta t^{2} \right)\,,
    \end{flalign}
    where $\mathbf{y}^{n+1/2}$ is defined as $\mathbf{y}^{n+1/2} \,=\, {(\mathbf{y}^{n+1} + \mathbf{y}^{n})}/{2}$.
\end{itemize}
  \end{proposition}
  \begin{proof}
From our first assumption and the first equation of \eqref{scheme:CNS}, we derive that each $(\eps^{-1} \mathbf{v}^{n+1/2}_{\eps})_{\eps > 0}$ is uniformly bounded with respect to $\eps$. By taking the limit $\eps\to 0$ in the triangle inequality $|\|\mathbf{v}^{n+1}_{\eps}\|-\|\mathbf{v}^{n}_{\eps}\||\leq 2\|\mathbf{v}^{n+1/2}_{\eps}\|$, one then deduces the convergence of $e^{n}_{\eps}$ and the second equation of \eqref{eq:guiding_center_scheme_CN}, thus also a uniform bound on $(\mathbf{v}^{n}_{\eps})_{\eps > 0}$. 

Now,  let us extract a subsequence still abusively labeled by $\eps$
                      such that $\mathbf{x}^{n}_{\eps}$
                      converges to some $\mathbf{y}^{n}$ as $\eps$
                      goes to zero.
By using the derived boundedness one may then take first a limit in the second equation  of \eqref{scheme:CNS} to obtain 
\[
\lim_{\eps\rightarrow 0}\dfrac{\mathbf{v}^{n+1/2}_{\eps}}{\eps} \,=\, -\frac{1}{b
              (\mathbf{y}^{n+1/2})}  
              \mathbf{E}^\perp(\mathbf{y}^{n+1/2})\,.
\]
Then, take a limit in the first equation of \eqref{scheme:CNS}  to
conclude the derivation of \eqref{eq:guiding_center_scheme_CN}. The
latter uniquely characterizes the limit of the subsequence, thereby
implying full convergence. At this stage, we use the smallness of $\Delta t$ (independent of $T$ and $\eps$)  to guarantee that the implicit scheme \eqref{eq:guiding_center_scheme_CN} is indeed solvable.

        Let us now turn to  the evolution of the total energy
        $\cE_\eps^n$ and the magnetic moment $\mu_\eps^n$.  For
        any $0 \leq n \leq N_{T}-1$, the convergence of
        $(\mathbf{x}^{n}_{\eps}, e^{n}_{\eps})$ to
        $(\mathbf{y}^{n},g^{n})$ as $\eps$ goes to zero implies the convergence of
        $(\cE_\eps^n,\mu_\eps^n)$ to $(\mathcal{E}_{gc}^{n},\mu_{gc}^{n})$. The total energy of the limiting system $\mathcal{E}_{gc}^{n}$ satisfies
       \[
        		\dfrac{\mathcal{E}_{gc}^{n+1} - \mathcal{E}_{gc}^{n}}{\Delta t} \, =\, \dfrac{g^{n+1} - g^{n}}{\Delta t} \,+\, \dfrac{\phi(\mathbf{y}^{n+1}) - \phi(\mathbf{y}^{n})}{\Delta t}
		\,=\, -\nabla_{\mathbf{y}} \phi(\mathbf{y}^{n+1/2}) \cdot
                \dfrac{\left( \mathbf{y}^{n+1} - \mathbf{y}^{n}
                  \right)}{\Delta t} \,+\, \mathcal{O} \left( \Delta
                  t^{2} \right)\,=\,\mathcal{O} \left( \Delta
                  t^{2} \right)\,,
\]
as deduced from a Taylor expansion and the insertion of $\mathbf{E} (\mathbf{y}^{n+1/2}) = - \nabla_{\mathbf{y}}
        \phi(\mathbf{y}^{n+1/2})$ in the first equation of \eqref{eq:guiding_center_scheme_CN}. 
	 Similarly, the evolution of the discrete  magnetic moment
        $\mu_{gc}^{n}$ obeys for all $0 \leq n \leq N_{T}-1$,
\begin{flalign*} 
	\dfrac{\mu_{gc}^{n+1} - \mu_{gc}^{n}}{\Delta t} & =\,
        \dfrac{1}{\Delta t} \left( \dfrac{g^{n+1}}{b(\mathbf{y}^{n+1})} -
          \dfrac{g^{n}}{b(\mathbf{y}^{n})} \right) \\[0.9em]
        & =\,\,-\,  g^{0}
        \,\dfrac{\nabla_{\mathbf{y}} b}{b^{2}} (\mathbf{y}^{n+1/2}) \cdot
        \dfrac{(\mathbf{y}^{n+1} - \mathbf{y}^{n})}{\Delta t}  \,+\,
        \mathcal{O}\left( \Delta t^{2}  \right)
        \\[0.9em]
	& =\, - g^{0} \,\dfrac{\mathbf{E}\cdot\nabla_{\mathbf{y}}^{\perp} b}{b^{3}} (\mathbf{y}^{n+1/2}) + \mathcal{O}\left( \Delta t^{2} \right)\,.
\end{flalign*}   
  \end{proof}

It is worth mentioning that  Proposition \ref{pro:CN}, clearly indicates that as $\eps$ goes to zero,  the discrete guiding center system
 \eqref{eq:guiding_center_scheme_CN} obtained by passing to the limit
 in the \textit{Crank-Nicolson scheme} is not consistent with the
 continuous system \eqref{eq:guiding_center_system}. Indeed,
 it fails to capture the correct drift $\nabla_{\mathbf{y}}^{\perp}
 b/b^{2}$ for both  position $\mathbf{y}$ and kinetic energy
 $g$. Similarly, the evolution of the magnetic moment $(\mu_{gc}^{n})_{n \in
   \mathbb{N}}$ derived from solution $(\mathbf{y}^{n}, g^{n})_{n \in
   \mathbb{N}}$ is not consistent with the continuous evolution. Even
 if the \textit{Crank-Nicolson scheme} provides a second order in time
 variation of the total energy  uniformly with respect to $\eps > 0$
 as indicated by \eqref{eq:te_CN} and \eqref{eq:te_limit_CN}, the
 variations of the discrete energy and discrete potential energy are
 not consistent.

Thus, several works have been devoted to modifications of the
Crank-Nicolson scheme for \eqref{eq:ODE_system}  to obtain a result of uniform consistency with
respect to $\eps$.  For instance, we mention  the work of  Brackbill,
Forslund and Vu \cite{BrFo85, VuBr95} who first introduced an effective force into the equation
on $\bv$ in order to capture the $\nabla_{\mathbf{x}}^{\perp} b/b^{2}$
term in the limit $\eps \rightarrow 0$. We also refer to the recent
work of Ricketson and Chac\'on \cite{RiCh20}, who proposed an
alternative approach, which is  expected to conserve energy within the
Particle-In-Cell framework.


\subsection{The Brackbill-Forslund-Vu scheme}
\label{section:Brackbill}
The scheme developed by Brackbill, Forslund and Vu in \cite{BrFo85,
  VuBr95} incorporates an effective force into the second equation of the  \textit{Crank-Nicolson scheme}
\eqref{scheme:CNS}. This additional force is designed to capture the
correct drift
$\nabla_{\mathbf{x}}^{\perp} b/b^{2}$ when $\eps \rightarrow 0$. Obviously in the regime $\Delta t \ll \eps$, this force is expected to
be significantly small, actually of order $\cO (\Delta t^2/\eps^4)$, see
\cite{BrFo85} for instance.  More precisely, for a given time step
$\Delta t > 0$, we define  $t^{n} = n \Delta t$, for $n \in
\mathbb{N}$ and  $(\bx^n_\eps, \bv^n_\eps)$, an approximation of the
solution $(\bx_\eps, \bv_\eps)$ to \eqref{eq:ODE_system} at time $t^n$,
through
\begin{equation}
\label{scheme:Brackbill}
    \begin{dcases}
        \eps\, \dfrac{\bx^{n+1}_\eps - \bx^n_\eps}{\Delta t} \, = \, \bv^{n+1/2}_\eps\,, \\[0.9em]
        \eps\, \dfrac{\bv^{n+1}_\eps - \bv^n_\eps}{\Delta t} \, = \,
        \bE (\bx^{n+1/2}_\eps)   \,+ \,  \bF^{n+1/2}_{\rm eff} - b(\bx^{n+1/2}_\eps) \dfrac{(\bv^{n+1/2}_\eps)^{\perp}}{\eps}\,, \\[0.9em]
        \bx^{0}_\eps = \bx(0),\, \quad \bv^{0}_\eps = \bv(0)\,,
    \end{dcases}
\end{equation}
where the effective force $\mathbf{F}_{\rm eff}$ is given by
\begin{flalign} \label{eq:F_eff}
    \bF^{n+1/2}_{\rm eff} \,:=\,- \eta^{n+1/2} \dfrac{\nabla_{\bx} b}{b}(\bx^{n+1/2}_\eps)\,, \quad \eta^{n+1/2} \,=\, \dfrac{1}{2} \left( \dfrac{\| \bv^{n+1}_\eps \|^{2} \, +\, \| \bv^{n}_{\eps} \|^{2}}{2} - \| \bv^{n+1/2}_\eps \|^{2} \right)\,,
\end{flalign}
and  we again use notation 
\[
    \mathbf{v}^{n+1/2}_{\eps} \,=\, \dfrac{\bv^{n+1}_\eps + \bv^{n}_\eps}{2}, \qquad \bx^{n+1/2}_\eps \,=\, \dfrac{\bx^{n+1}_\eps + \bx^{n}_\eps}{2} \,.
\]
We now define the discrete kinetic energy $e^n_\eps =
\frac{1}{2} \| \bv^n_\eps \|^2$ and total energy  as
\[
	\cE^n_\eps \,:=\, e^n_\eps \,+\, \phi(\bx^n_\eps)\,,
\]
in which $\phi$ is a given smooth function $\phi \in W^{3,
  \infty}(\mathbb{R}^{2})$.  Hence, we  obtain the variation of the
discrete kinetic energy by multiplying the second equation of
\eqref{scheme:Brackbill} by $\bv^{n+1/2}_\eps$, which gives  
\begin{flalign} \label{eq:Brackbill_ke}
	\dfrac{e^{n+1}_\eps - e^n_\eps}{\Delta t} \, = \, \bE
        (\bx_\eps^{n+1/2}) \cdot \dfrac{\bv^{n+1/2}_\eps}{\eps} +
        \bF^{n+1/2}_{\rm eff} \cdot \dfrac{\bv^{n+1/2}_\eps}{\eps}\,.
\end{flalign}
Then, applying a Taylor expansion to the potential $\phi$, it yields
that 
\begin{flalign}
  \label{eq:Brackbill_energy}
    \nonumber
    \dfrac{\cE^{n+1}_\eps \, - \, \cE^{n}_\eps}{\Delta t} & \, =\,  \dfrac{e^{n+1}_\eps - e^n_\eps}{\Delta t} \, +\, \dfrac{\phi_\eps^{n+1} - \phi_\eps^n}{\Delta t}\,, \\
    & \, = \, \bF^{n+1/2}_{\rm eff} \cdot \dfrac{\bx^{n+1}_\eps - \bx^n_\eps}{\Delta t} \, +\, \Delta t^{2} \,\cO \left(  \left\| \dfrac{\bx^{n+1}_\eps - \bx^n_\eps}{\Delta t} \right\|^{3} \right)\,.
  \end{flalign}
  Using the definition of the effective force $\bF^{n+1/2}_{\rm eff}$,
  we have
  \[
|\eta^{n+1/2}| \,=\, \frac{1}{8} \| \bv^{n+1}- \bv^n\|^2 \,=\,  \left(\frac{\Delta t}{\eps}\right)^2 \,\cO \left(  \left\| \dfrac{\bx^{n+1}_\eps - \bx^n_\eps}{\Delta t}\right\| \right).  
  \]
Then under our above assumptions and the further reasonable assumption
that $\mathbf{x}_{\eps}^{n}$ is bounded,  the evolution of the discrete
energy obtained by \eqref{scheme:Brackbill}  is much worse than the
one \eqref{eq:te_CN} corresponding to the Crank-Nicolson
scheme. Now, let us study the asymptotic behavior of \eqref{scheme:Brackbill} as $\eps$ tends to zero.

\begin{proposition}[Asymptotic behavior $\eps\rightarrow 0$ with a fixed $\Delta t$]
  \label{pro:Brackbill_scheme}
Let $\phi \in W^{3, \infty}(\mathbb{R}^{2})$, choose a sufficiently small fixed time step $\Delta t$ and a final time $T > 0$.  We set  $N_{T} = \left
   \lfloor T/\Delta t \right \rfloor$. 
Assume that the modified Crank-Nicolson scheme \eqref{scheme:Brackbill} defines a numerical approximation
 $(\mathbf{x}^{n}_{\eps}, \mathbf{v}^{n}_{\eps})_{0 \leq n \leq N_{T}}$ satisfying
    \begin{itemize}
        \item[$(i)$] for all $1 \leq n \leq N_{T}$, $\mathbf{x}^{n}_{\eps}$ is uniformly bounded with respect to $\eps > 0$\,;
        \item[$(ii)$] in the limit $\eps \rightarrow 0$, $(\mathbf{x}^{0}_{\eps}, \frac{1}{2} \| \mathbf{v}^{0}_{\eps} \|^{2})$ converges to some $(\mathbf{y}^{0}, g^{0})$\,.
    \end{itemize}  
Then,  we have
\begin{itemize}
\item    for all $1 \leq n \leq N_{T}$, $(\mathbf{x}_{\eps}^{n},
  e^{n}_{\eps})$ converges to $(\mathbf{y}^{n}, g^{n})$, as $\eps \rightarrow 0$ with $e^{n}_{\eps} = \frac{1}{2} \| \mathbf{v}^{n}_{\eps} \|^{2}$  and the limit $(\mathbf{y}^{n}, g^{n})_{1 \leq n \leq N_{T}}$ solves
    \begin{equation} \label{eq:guiding_center_Brackbill}
        \begin{dcases}
            \dfrac{\mathbf{y}^{n+1} - \mathbf{y}^{n}}{\Delta t} \,=\, - \dfrac{\mathbf{E}^{\perp}}{b}(\mathbf{y}^{n+1/2}) \,+\, g^{0} \dfrac{\nabla_{\mathbf{y}}^{\perp} b}{b^{2}} (\mathbf{y}^{n+1/2}), \\
            \dfrac{g^{n+1} - g^{n}}{\Delta t} = 0.
        \end{dcases}
    \end{equation}
\item  for all $1 \leq n \leq N_{T}$, the total energy
  $\mathcal{E}^{n}_{\eps} = e^{n}_{\eps} + \phi(\mathbf{x}^{n}_{\eps})$
  converges to $\mathcal{E}^{n}_{gc} := g^{n} + \phi(\mathbf{y}^{n})$ as
  $\eps \rightarrow 0$, which satisfies
  \begin{flalign}
    \label{eq:te_limit_Brackbill}
        \dfrac{\mathcal{E}^{n+1}_{gc} -  \mathcal{E}^{n}_{gc}}{\Delta
          t} \,=\, -g^{0} \dfrac{\mathbf{E}\cdot\nabla_{\mathbf{y}}^{\perp}
          b}{b^{2}} (\mathbf{y}^{n+1/2}) + \mathcal{O}\left( \Delta t^{2}
        \right)\,;
      \end{flalign}
\item  defining the discrete  magnetic moment $\mu^{n}_{\eps} =
  \dfrac{e^{n}_{\eps}}{b(\mathbf{x}_{\eps}^{n})}$ converges to
  $\mu^{n}_{gc} := \dfrac{g^{n}}{b(\mathbf{y}^{n})}$ as $\eps \rightarrow
  0$ such that
    \begin{flalign} \label{eq:mu_limit_Brackbill}
        \dfrac{\mu^{n+1}_{gc} -  \mu^{n}_{gc}}{\Delta t} \,=\, - g^{0} \dfrac{\mathbf{E}\cdot\nabla_{\mathbf{y}}^{\perp} b}{b^{3}} (\mathbf{y}^{n+1/2}) \,+\,  \mathcal{O}\left( \Delta t^{2} \right).
    \end{flalign}
    where $\mathbf{y}^{n+1/2}$ is defined as $\mathbf{y}^{n+1/2} \,=\, {(\mathbf{y}^{n+1} \,+\, \mathbf{y}^{n})}/{2}$.
\end{itemize}
  \end{proposition}
  \begin{proof}
The beginning of the proof of Proposition \ref{pro:CN} applies word by word to the present case since it only uses the first equation of the scheme. In this way one arrives at a stage where one knows that $(\eps^{-1} \mathbf{v}^{n+1/2}_{\eps})_{\eps > 0}$ is uniformly bounded with respect to $\eps$ and a subsequence $(\mathbf{x}^{n}_{\eps}, e^{n}_{\eps})$ converges to some $(\mathbf{y}^{n},g^{n})$ as $\eps$ goes to zero, satisfying  the second equation of \eqref{eq:guiding_center_Brackbill}.

Note that this also implies
\[
\lim_{\eps \rightarrow 0} \eta^{n+1/2} = g^{0}
\]
so that when taking the limit $\eps\to0$ in the second equation of \eqref{scheme:Brackbill} one receives
        \begin{equation} 
            \nonumber
            \lim_{\eps \rightarrow 0}  \frac{\mathbf{v}^{n+1/2}_{\eps}}{\eps} \,=\, - \dfrac{ \mathbf{E}^{\perp}}{b}(\mathbf{y}^{n+1/2}) + g^{0} \dfrac{\nabla_{\mathbf{y}}^{\perp} b}{b^{2}}(\mathbf{y}^{n+1/2}).
        \end{equation}
This is sufficient to complete the derivation of \eqref{eq:guiding_center_Brackbill}. As in the proof of Proposition~\ref{pro:CN}, we conclude the full convergence (and not only the convergence of a subsequence) from the fact that \eqref{eq:guiding_center_Brackbill} defines a unique solution.

We then
        proceed as in the proof of Proposition~\ref{pro:CN}  for the
        evolution of  the total energy
        $\cE_\eps^n$ and the magnetic moment $\mu_\eps^n$, and derive
\begin{flalign} 
	\nonumber
	\dfrac{\mathcal{E}_{gc}^{n+1} - \mathcal{E}_{gc}^{n}}{\Delta t} 
	& = \nabla_{\mathbf{y}} \phi(\mathbf{y}^{n+1/2}) \cdot \left( \dfrac{\mathbf{y}^{n+1} - \mathbf{y}^{n}}{\Delta t} \right) + \mathcal{O}\left( \Delta t^{2} \right) 
	= -g^{0} \dfrac{\mathbf{E}\cdot\nabla_{\mathbf{y}}^{\perp} b}{b^{2}} (\mathbf{y}^{n+1/2}) +  \mathcal{O}\left( \Delta t^{2} \right)
\end{flalign}
and 
\begin{flalign} 
	\nonumber
	\dfrac{\mu_{gc}^{n+1} - \mu_{gc}^{n}}{\Delta t} 
	& = -  g^{0} \dfrac{(\mathbf{y}^{n+1} - \mathbf{y}^{n})}{\Delta t}\cdot\dfrac{\nabla_{\mathbf{y}} b}{b^{2}} (\mathbf{y}^{n+1/2}) + \mathcal{O}\left( \Delta t^{2} \right) 
	= - g^{0}\dfrac{ \mathbf{E}\cdot\nabla_{\mathbf{y}}^{\perp} b}{b^{3}} (\mathbf{y}^{n+1/2}) + \mathcal{O}\left( \Delta t^{2} \right).
\end{flalign}   
    \end{proof}

Here it is worth mentioning that  Proposition
\ref{pro:Brackbill_scheme} shows that the scheme with the effective force
\eqref{scheme:Brackbill}  does not give a consistent
    approximation of  slow variables $(\mathbf{x}_{\eps}, e_{\eps})$ in
    the limit $\eps \rightarrow 0$. Indeed, in the limit
    $\eps\rightarrow 0$,  the scheme
    \eqref{scheme:Brackbill} exactly preserves the kinetic energy over
    time while this quantity should vary according to the gradient
    of the magnetic field (grad $\bB$ drift). Therefore, neither the
    discrete guiding center variable $\by^n$ nor the kinetic energy
    $g^n$ are consistent approximation  of   the guiding center system
    \eqref{eq:guiding_center_system}. As a consequence, the evolution of the discrete
    magnetic moment $\mu_{gc}$ is also not consistent with  the
    continuous equation \eqref{adiabatic_conserve}.    Furthermore, the scheme
    \eqref{scheme:Brackbill} fails to preserve the second order
    accuracy with respect to $\Delta t$ of the total  energy
    $\mathcal{E}_{gc}$, in contrast to  the Crank-Nicolson scheme
    \eqref{scheme:CNS}.

    In order to overcome this drawback, an alternative numerical
    scheme, still based on the Crank-Nicolson method  has been
    proposed recently by Ricketson and Chac\'on \cite{RiCh20}.


\subsection{The Ricketson-Chac\'on  scheme}
\label{section:Chacon}

The numerical scheme proposed by L. F. Ricketson and L. Chac\'on
\cite{RiCh20}  still consists in adding a force term to capture the
$\nabla_{\mathbf{x}}^{\perp} b/b^{2}$  drift  in the asymptotic
limit $\eps \rightarrow 0$. More precisely, the force
$\mathbf{F}^{n+1/2}_{\rm eff}$ is chosen to be orthogonal to
the velocity $\mathbf{v}^{n+1/2}_{\eps}$, so that it does not interfere with the evolution
of the discrete kinetic energy.  Explicitly as in \cite{RiCh20}, set
\begin{flalign} \label{eq:F_cons}
    \mathbf{F}^{n+1/2}_{\rm cons} = \left( {\rm Id} - \dfrac{\mathbf{v}_{\eps}^{n+1/2}\otimes\mathbf{v}_{\eps}^{n+1/2}}{\| \mathbf{v}_{\eps}^{n+1/2} \|^{2}} \right) \,\mathbf{G}^{n+1/2}\,,
\end{flalign}
where $\mathbf{G}^{n+1/2}$ is given by 
\begin{flalign} \label{eq:G}
    \mathbf{G}^{n+1/2} = 
    \begin{dcases}
        2 \,\mathbf{F}^{n+1/2}_{\rm eff} \quad \textrm{if }\quad
        \|\mathbf{v}^{n+1/2}_{\eps} - \mathbf{v}^{n+1/2}_{\mathbf{E},\eps} \|\,\geq\,\| \mathbf{v}^{n+1/2}_{\mathbf{E},\eps}\|, \\[0.9em]
     \left( \dfrac{2}{\beta^{n+1/2}_\eps} \hat{\mathbf{v}}^{n+1/2}_{\mathbf{E},\eps}\otimes
          \hat{\mathbf{v}}^{n+1/2}_{\mathbf{E},\eps} \,+\, \dfrac{{\rm Id} \,-\,
            \hat{\mathbf{v}}^{n+1/2}_{\mathbf{E},\eps}\otimes \hat{\mathbf{v}}^{n+1/2}_{\mathbf{E},\eps}}{1 -
            \frac{\beta^{n+1/2}_\eps}{2}} \right)\mathbf{F}^{n+1/2}_{\rm eff} \quad 
            \textrm{ otherwise}
    \end{dcases}
\end{flalign}
with  
\begin{align*}
\bv^{n+1/2}_{\bE,\eps}&\,=\, - \dfrac{\bE^\perp}{b}(\bx^{n+1/2}_{\eps})\,,&
\hat{\mathbf{v}}^{n+1/2}_{\mathbf{E},\eps}&\,=\, \dfrac{\mathbf{v}^{n+1/2}_{\mathbf{E},\eps}}{\|\mathbf{v}^{n+1/2}_{\mathbf{E},\eps}\|}\,,&
\beta^{n+1/2}_\eps&\,=\,\frac{\|\mathbf{v}^{n+1/2}_{\eps} - \mathbf{v}^{n+1/2}_{\mathbf{E},\eps} \|^2}{\| \mathbf{v}^{n+1/2}_{\mathbf{E},\eps}\|^2}\,,
\end{align*}
whereas the effective force  $\mathbf{F}^{n+1/2}_{\rm eff}$ is given in \eqref{eq:F_eff} and ${\rm Id}$ is the identity matrix.  Then, the modified Crank-Nicolson scheme  now becomes \cite{RiCh20} 
\begin{equation}
  \label{scheme:Chacon}
    \begin{dcases}
        \eps \dfrac{\mathbf{x}^{n+1}_{\eps} - \mathbf{x}^{n}_{\eps}}{\Delta t} \,=\, \mathbf{v}^{n+1/2}_{\eps}, \\
        \eps \dfrac{\mathbf{v}^{n+1}_{\eps} - \mathbf{v}^{n}_{\eps}}{\Delta
          t} \,=\, \mathbf{E}(\mathbf{x}^{n+1/2}_{\eps}) \,+\,
        \mathbf{F}^{n+1/2}_{\rm cons} \,-\,  b(\mathbf{x}^{n+1/2}_{\eps}) \,\dfrac{(\mathbf{v}^{n+1/2}_{\eps})^{\perp}}{\eps}, \\
        \mathbf{x}^{0}_{\eps} \,=\, \mathbf{x}(0), \quad \mathbf{v}^{0}_{\eps} = \mathbf{v}(0).
    \end{dcases}
\end{equation}
where again
\begin{equation*}
    \nonumber
    \mathbf{v}^{n+1/2}_{\eps} = \dfrac{\mathbf{v}^{n+1}_{\eps} + \mathbf{v}^{n}_{\eps}}{2}, \qquad\qquad \mathbf{x}^{n+1/2}_{\eps} = \dfrac{\mathbf{x}^{n+1}_{\eps} + \mathbf{x}^{n}_{\eps}}{2}.
  \end{equation*}
As we did previously, we define the  kinetic energy $e^{n}_{\eps} = \frac{1}{2}
\| \mathbf{v}^{n}_{\eps} \|^{2}$ and the discrete  total energy as
$\mathcal{E}^{n}_{\eps} = e^{n}_{\eps} + \phi(\mathbf{x}_{\eps}^{n})$ for
a given potential charge $\phi \in W^{3,
  \infty}(\mathbb{R}^{2})$. Using  that $\mathbf{F}^{n+1/2}_{\rm cons} $
is orthogonal to $\mathbf{v}^{n+1/2}_{\eps}$ thus to $\mathbf{x}^{n+1}_{\eps}-\mathbf{x}^{n}_{\eps}$, we recover the same evolution of the discrete total energy as the one for the Crank-Nicolson scheme:  for all $\eps > 0$,
\begin{flalign} \label{eq:te_Chacon}
\dfrac{\mathcal{E}^{n+1}_{\eps} - \mathcal{E}^{n}_{\eps}}{\Delta t} 
\,=\, \Delta t^{2} \,\mathcal{O}\left( \left\|\frac{\mathbf{x}^{n+1}_{\eps}-\mathbf{x}^{n}_{\eps}}{\Delta t}\right\|^{3} \right)\,.
\end{flalign}

Note that strictly speaking, because of \eqref{eq:F_cons}, when
$\mathbf{v}^{n+1/2}_{\eps}$ is zero an alternative for
$\mathbf{F}^{n+1/2}_{\rm cons}$ should be used.  Now, let us investigate the  asymptotic limit of
the scheme \eqref{scheme:Chacon} when $\eps \rightarrow 0$ with a fixed $\Delta t$.
\begin{proposition}[Asymptotic behavior $\eps\rightarrow 0$ with a fixed $\Delta t$]
\label{pro:Chacon_scheme}
Let $\phi \in W^{3, \infty}(\mathbb{R}^{2})$ such that $\nabla\phi$ is nowhere vanishing, choose a sufficiently small fixed time step $\Delta t$ and a final time $T > 0$.  We set  $N_{T} = \left
   \lfloor T/\Delta t \right \rfloor$. 
Assume that the modified Crank-Nicolson scheme \eqref{scheme:Chacon} defines a numerical approximation
 $(\mathbf{x}^{n}_{\eps}, \mathbf{v}^{n}_{\eps})_{0 \leq n \leq N_{T}}$ satisfying
    \begin{itemize}
        \item[$(i)$] for all $1 \leq n \leq N_{T}$, $\mathbf{x}^{n}_{\eps}$ is uniformly bounded with respect to $\eps > 0$\,;
        \item[$(ii)$] in the limit $\eps \rightarrow 0$, $(\mathbf{x}^{0}_{\eps}, \frac{1}{2} \| \mathbf{v}^{0}_{\eps} \|^{2})$ converges to some $(\mathbf{y}^{0}, g^{0})$\,.
    \end{itemize}  
    Then,  we have
    \begin{itemize}
\item    for all $1 \leq n \leq N_{T}$, $(\mathbf{x}_{\eps}^{n},
  e^{n}_{\eps})$ converges to $(\mathbf{y}^{n}, g^{n})$, as $\eps \rightarrow 0$ with $e^{n}_{\eps} = \frac{1}{2} \| \mathbf{v}^{n}_{\eps} \|^{2}$  and the limit $(\mathbf{y}^{n}, g^{n})_{1 \leq n \leq N_{T}}$ solves
    \begin{equation} 
    \label{eq:guiding_center_Chacon}
         \begin{dcases}
            \dfrac{\mathbf{y}^{n+1} - \mathbf{y}^{n}}{\Delta t} \,=\, -
            \dfrac{\mathbf{E}^{\perp}}{b}(\mathbf{y}^{n+1/2}) \,+\, 2\,
            g^{0} \,\left( \dfrac{\mathbf{E} \cdot \nabla_{\mathbf{y}} b}{b^{2} \,\| \mathbf{E} \|^{2}}\right)\, \mathbf{E}^{\perp}  (\mathbf{y}^{n+1/2}), \\
            \dfrac{g^{n+1} - g^{n}}{\Delta t} \,=\, 0.
        \end{dcases}
      \end{equation}
     \item  for all $1 \leq n \leq N_{T}$, the total energy
  $\mathcal{E}^{n}_{\eps} = e^{n}_{\eps} + \phi(\mathbf{x}^{n}_{\eps})$
  converges to $\mathcal{E}^{n}_{gc} := g^{n} + \phi(\mathbf{y}^{n})$ as
  $\eps \rightarrow 0$, which satisfies
    \begin{flalign} \label{eq:te_limit_Chacon}
        \dfrac{\mathcal{E}^{n+1}_{gc} -  \mathcal{E}^{n}_{gc}}{\Delta
          t} \,=\, \mathcal{O}\left( \Delta t^{2}
        \right),
      \end{flalign}
     \item  defining the discrete  magnetic moment $\mu^{n}_{\eps} =
  \dfrac{e^{n}_{\eps}}{b(\mathbf{x}_{\eps}^{n})}$ converges to
  $\mu^{n}_{gc} := \dfrac{g^{n}}{b(\mathbf{y}^{n})}$ as $\eps \rightarrow
  0$, which satisfies
  \begin{flalign}
    \label{eq:mu_limit_Chacon}
         \dfrac{\mu_{gc}^{n+1} - \mu_{gc}^{n}}{\Delta t} \,=\, - g^{0} \, \dfrac{\mathbf{E}^{\perp} \cdot\nabla_{\mathbf{y}} b}{b^{2}}(\mathbf{y}^{n+1/2}) \,-\, 2 \,\left( g^{0} \right)^{2} \,\dfrac{\left( \mathbf{E}^{\perp} \cdot \nabla_{\mathbf{y}} b \right) \left( \mathbf{E} \cdot \nabla_{\mathbf{y}} b \right)}{b^{4} \| \mathbf{E} \|^{2}} (\mathbf{y}^{n+1/2}) \,+\,  \mathcal{O}\left( \Delta t^{2} \right).
    \end{flalign}
    where $\mathbf{y}^{n+1/2}$ is defined as $\mathbf{y}^{n+1/2} \,=\,{(\mathbf{y}^{n+1} \,+\, \mathbf{y}^{n})}/{2}$.
       \end{itemize}
    \end{proposition}
\begin{proof}
The beginning of the proof of Proposition \ref{pro:Brackbill_scheme}
applies verbatim to the present case, since it only utilizes the first
equation of the scheme and the definition of $\mathbf{F}^{n+1/2}_{\rm
  eff}$. In this way, one arrives at a stage where it is known that $(\eps^{-1} \mathbf{v}^{n+1/2}_{\eps})_{\eps > 0}$ is uniformly bounded with respect to $\eps$ and a subsequence $(\mathbf{x}^{n}_{\eps}, e^{n}_{\eps})$ converges to some $(\mathbf{y}^{n},g^{n})$ as $\eps$ goes to zero, satisfying the second equation of \eqref{eq:guiding_center_Chacon} and 
\[
\lim \limits_{\eps \rightarrow 0} \mathbf{F}_{\rm eff}^{n+1/2} \,=\, -\,g^{0} \dfrac{\nabla_{\mathbf{y}} b}{b} (\mathbf{y}^{n+1/2}).
\]

Note that, since by assumption $\bE$ is nowhere vanishing, this also implies
\[
\lim_{\eps \rightarrow 0} \beta^{n+1/2} = 1
\]
so that 
\begin{flalign}
 	\nonumber
 	\lim \limits_{\eps \rightarrow 0} \mathbf{G}^{n+1/2} \,=\, 2\,\lim \limits_{\eps \rightarrow 0} \mathbf{F}_{\rm eff}^{n+1/2} \,=\, -2\,g^{0} \dfrac{\nabla_{\mathbf{y}} b}{b} (\mathbf{y}^{n+1/2}).
      \end{flalign}
This time the determination of the limit of $\eps^{-1} \mathbf{v}^{n+1/2}_{\eps}$ is much more complicated. By extracting further if necessary we may assume that it converges to some $\bu^{n+1/2}$ and that $\mathbf{v}^{n+1/2}_{\eps}/\|\mathbf{v}^{n+1/2}_{\eps}\|$ converges to some $\hat{\bu}^{n+1/2}$. By using that for any nonzero $\bz$
\[
{\rm Id} -  \frac{\bz\otimes\bz}{\|\bz\|^2}\,=\,
\frac{\bz^\perp\otimes\bz^\perp}{\|\bz\|^2}
\]
and taking a limit in the second equation of \eqref{scheme:Chacon}, we derive
\[
\hat{\bu}^{n+1/2}\left(\|\bu^{n+1/2}\|-2g^0\,\hat{\bu}^{n+1/2}\cdot\frac{\nabla_{\mathbf{y}}^\perp b}{b^2}(\mathbf{y}^{n+1/2})\right)
\,=\,-\frac{\bE^\perp}{b}(\mathbf{y}^{n+1/2})
\]
Since $\bE$ is non vanishing, this implies that $\hat{\bu}^{n+1/2}$ is colinear to $\bE^\perp$ and thus
\[
\bu^{n+1/2}\,=\,-\frac{\bE^\perp}{b}(\mathbf{y}^{n+1/2})
+2g^0\,\bE^\perp(\mathbf{y}^{n+1/2})\,\frac{\bE\cdot\nabla_{\mathbf{y}} b}{b^2\|\bE\|^2}(\mathbf{y}^{n+1/2})\,.
\]
This completes the derivation of \eqref{eq:guiding_center_Chacon}, which then may be used as before to upgrade the convergence to the full convergence.

The rest of the proof for the variations of the total discrete energy and the discrete adiabatic
invariant is then analogous to the ones in the proof of Propositions~\ref{pro:CN} and~\ref{pro:Brackbill_scheme}.
\end{proof}

This latter Proposition shows that the modified scheme
\eqref{scheme:Chacon} does not provide a consistent asymptotic limit
when $\eps$ goes to zero and $\Delta t $ is
fixed. \textcolor{red}{Actually, in \cite{RiCh20}, the authors propose
  an adaptive time step procedure to overcome  this drawback. }

\begin{remark}
  \label{rem:ener}
It is worth
mentioning that the Crank-Nicolson is well suited to design an
approximation preserving the total energy. In particular, we refer to
\cite{SiTa92} where the following  modified electric field 
	\begin{flalign} \label{replacement_E}
		\mathbf{\tilde{E}}(\bx_\eps^{n+1/2}) \,=\, \dfrac{\phi(\mathbf{x}^{n+1}_{\eps}) - \phi(\mathbf{x}^{n}_{\eps})}{(\mathbf{x}^{n+1}_{\eps} - \mathbf{x}^{n}_{\eps} ) \cdot \mathbf{E}(\bx_\eps^{n+1/2})} \, \mathbf{E}(\bx_\eps^{n+1/2}).
              \end{flalign}
               is applied ensuring exact preservation of the total
               energy,  that is $\mathcal{E}^{n}_{\eps} =
               \mathcal{E}^{0}_{\eps}$, for all $n \in
               \mathbb{N}$. However, this exact preservation does not
               help to provide a  consistent approximation as $\eps
               \to 0$ since the discrete kinetic energy is not
               uniformly consistent with respect to $\eps$.
\end{remark}
               
%
%

\section{A modified Crank-Nicolson scheme with an additional variable}
\label{sec:Modified_CN}
We now propose a new numerical scheme based on the Crank-Nicolson
method \eqref{scheme:CNS},  designed to be asymptotically consistent
with the  guiding center model \eqref{eq:guiding_center_system} as
$\eps \to 0$. To this aim, we apply the strategy developed in
\cite{FiRo17} which consists in  solving an augmented system
incorporating  the discrete  kinetic energy $(e_\eps^n)_{n \in \N}$
into the discrete system. Furthermore, as in the previous work, we add
an effective force to capture the drift $ \nabla_{\bx}^{\perp} b/
b^{2}$ in the limit  $\eps \to 0$.  More precisely, we reformulate the
system \eqref{eq:ODE_system} for  $(\bx_\eps, \bv_\eps) $ in an
equivalent manner for the new unknowns  $(\bx_\eps, \bw_\eps,e_\eps)$
as 
\begin{equation} \label{eq:reformulated_ODE_system}
    \begin{dcases}
       \, \eps \, \dfrac{\dD \bx_\eps}{{\dD} t} \,=\, \mathbf{w}_\eps \,, \\
       \, \eps \, \dfrac{\dD e_\eps}{{\dD} t}  \,=\,  \bE(\bx_\eps) \cdot \mathbf{w}_\eps\,, \\
       \, \eps \, \dfrac{\dD \mathbf{w}_\eps}{{\dD} t}  \,=\, \bE(\bx_\eps) \,-\, \chi \left(\mathbf{w}_\eps, e_\eps \right) \, \dfrac{\nabla_{\bx} b}{b} (\bx_\eps) \,-\, b(\bx_\eps) \, \dfrac{\mathbf{w}_\eps^{\perp}}{\eps}, \\
       \, \bx_\eps^0 \,=\, \bx(0)\,, \quad \mathbf{w}_\eps^0 \,=\, \bv(0)\,, \quad e_\eps^0 \,=\, \dfrac{1}{2} \| \bv(0) \|^{2}\,,
    \end{dcases}
\end{equation}
where  $\chi$ is chosen as
\begin{equation}
    \nonumber
    \chi \left( \mathbf{w}, e \right) \,=\, \max \left( e - \dfrac{1}{2} \left \| \mathbf{w}\right \|^{2}, 0 \right), \quad \forall \, ( \mathbf{w}, e) \, \in \,  \R^{2} \times \R^{+} \,.
  \end{equation}
 Then, we discretize this system applying a classical Crank-Nicolson
 scheme  to  $(\bx^n_\eps, \mathbf{w}^n_\eps, e^{n}_\eps)$, 
\begin{equation} \label{scheme:modified_CN}
    \begin{dcases}
        \,\eps \, \dfrac{\bx^{n+1}_\eps - \bx^n_\eps}{\Delta t} \,=\, \mathbf{w}^{n+1/2}_\eps \,, \\
        \,\eps \, \dfrac{e^{n+1}_\eps - e^n_\eps}{\Delta t} \,=\,  \bE (\bx^{n+1/2}_\eps) \cdot \mathbf{w}^{n+1/2}_\eps \,, \\
        \,\eps \, \dfrac{\mathbf{w}^{n+1}_\eps - \mathbf{w}^n_\eps}{\Delta t} \,=\, \bE(\bx^{n+1/2}_\eps) \,-\, \chi(\mathbf{w}_{\eps}^{n+1/2}, e_{\eps}^{n+1/2}) \, \dfrac{\nabla_{\bx} b}{b}(\bx^{n+1/2}_{\eps}) \,-\, b(\bx^{n+1/2}_{\eps}) \, \dfrac{(\mathbf{w}^{n+1/2}_{\eps})^{\perp}}{\eps} \,, \\
       \, \mathbf{x}_{\eps}^{0} \,=\, \mathbf{x}_{\eps}(0), \quad \mathbf{w}_{\eps}^{0} \,=\, \bv_{\eps}(0), \quad e_{\eps}^{0} \,=\, \frac{1}{2} \| \bv_{\eps}(0) \|^{2} \,,
    \end{dcases}
\end{equation}
in which
\begin{equation}
    \nonumber
    \mathbf{w}^{n+1/2}_\eps \,=\, \dfrac{\mathbf{w}^{n+1}_\eps + \mathbf{w}^n_\eps}{2}\,, \quad \bx^{n+1/2}_\eps \,=\, \dfrac{\bx^{n+1}_\eps \,+\, \bx^n_\eps}{2}\,.
  \end{equation}
\textcolor{red}{Computing an approximation at
  time $t^{n+1}$ requires the numerical solution of a nonlinear
  system. Here,  we employ a straightforward fixed-point iteration
  scheme, where the variable  $\bx^{n+1/2}_{\eps}$, depending on $\bx^{n+1}_{\eps}$, is frozen at each
step. This approach reduces the problem to solving a two-dimensional
linear system for the velocity
variable  $\mathbf{w}^{n+1}_\eps$ of each particle. The convergence tolerance is set to $10^{-10}$. Although alternative approaches based on Newton's method could be utilized, the current fixed-point scheme converges rapidly, typically requiring fewer than five iterations.}

At each time step, the velocity $(\bv^{n}_{\eps})_{n \in \N}$ is given by
\begin{equation}
    \nonumber
    \bv^{n}_{\eps} \,=\, \sqrt{2 \,e^{n}_{\eps}} \, \dfrac{\mathbf{w}^{n}_{\eps}}{\| \mathbf{w}^{n}_{\eps} \|}\,.
\end{equation}
As for the original Crank-Nicolson scheme \eqref{eq:te_CN}, the variation of the discrete total
energy obeys
\begin{flalign} \label{eq:te_modified_CN}
\dfrac{\mathcal{E}^{n+1}_{\eps} - \mathcal{E}^{n}_{\eps}}{\Delta t} \,=\, \dfrac{e^{n+1}_{\eps} - e^{n}_{\eps}}{\Delta t} \,+\, \dfrac{\phi_{\eps}^{n+1} - \phi_{\eps}^{n}}{\Delta t} \,=\, \Delta t^{2} \,\mathcal{O}\left(\left\|\frac{\mathbf{x}_{\eps}^{n+1}-\mathbf{x}_{\eps}^{n}}{\Delta t}\right\|^{3} \right).
\end{flalign}
Therefore, under the reasonable assumption that $\mathbf{x}_{\eps}^{n}$ is bounded, the variation of the total discrete  energy is of order  $\Delta t^2$. Now, let us investigate the asymptotic behavior of the scheme \eqref{scheme:modified_CN} as $\eps $ goes to zero.
\begin{proposition}[Consistency in the limit $\eps \rightarrow 0$ for a fixed $\Delta t$] \label{pro:modified_CN}
Let $\phi \in W^{3, \infty}(\mathbb{R}^{2})$, choose an a priori bound $M$ and a final time $T>0$ then a sufficiently small fixed time step $\Delta t$. We set  $N_{T} = \left
   \lfloor T/\Delta t \right \rfloor$.\\
Assume that the Crank-Nicolson scheme \eqref{scheme:modified_CN} defines a numerical approximation
 $(\mathbf{x}^{n}_{\eps}, \mathbf{w}^{n}_{\eps},e^{n}_{\eps})_{0 \leq n \leq N_{T}}$ satisfying
    \begin{itemize}
        \item[$(i)$] for all $1 \leq n \leq N_{T}$, $\mathbf{x}^{n}_{\eps}$ is uniformly bounded with respect to $\eps > 0$\,;
        \item[$(ii)$] in the limit $\eps \rightarrow 0$, $(\mathbf{x}^{0}_{\eps},e^{n}_{\eps})$ converges to some $(\mathbf{y}^{0}, g^{0})$ such that $g^{0}\leq M$\,.
    \end{itemize}  
    Then we have
    \begin{itemize}
    	\item for all $1 \leq n \leq N_{T}$, $(\mathbf{x}_{\eps}^{n},
          e^{n}_{\eps})$ converges to $(\mathbf{y}^{n}, g^{n})$, as $\eps \rightarrow 0$ and the limit $(\mathbf{y}^{n}, g^{n})_{1 \leq n \leq N_{T}}$ solves
    \begin{equation} \label{eq:guiding_center_modified_CN}
         \begin{dcases}
            \dfrac{\mathbf{y}^{n+1} - \mathbf{y}^{n}}{\Delta t} \,=\, -\, \dfrac{\mathbf{E}^{\perp}}{b}(\mathbf{y}^{n+1/2}) \,+\, g^{n+1/2} \, \dfrac{\nabla_{\mathbf{y}}^{\perp} b}{b^{2}}(\mathbf{y}^{n+1/2})\,, \\
            \dfrac{g^{n+1} - g^{n}}{\Delta t} \,=\, g^{n+1/2} \, \dfrac{\mathbf{E} \cdot \nabla_{\mathbf{y}}^{\perp} b}{b^{2}} (\mathbf{y}^{n+1/2}) \,.
        \end{dcases}
    \end{equation}
    \item  for all $1 \leq n \leq N_{T}$, the total energy
  $\mathcal{E}^{n}_{\eps} \,=\, e^{n}_{\eps} \,+\, \phi(\mathbf{x}^{n}_{\eps})$
  converges to $\mathcal{E}^{n}_{gc} \,:=\, g^{n} \,+\, \phi(\mathbf{y}^{n})$ as
  $\eps \rightarrow 0$, which satisfies
    \begin{flalign} \label{eq:te_limit_modified_CN}
        \dfrac{\mathcal{E}^{n+1}_{gc} -  \mathcal{E}^{n}_{gc}}{\Delta t} \,=\, \mathcal{O}\left( \Delta t^{2}  \right)\,,
     \end{flalign}
      \item  defining the discrete  magnetic moment $\mu^{n}_{\eps} \,=\,
  \dfrac{e^{n}_{\eps}}{b(\mathbf{x}_{\eps}^{n})}$ converges to
  $\mu^{n}_{gc} \,:=\, \dfrac{g^{n}}{b(\mathbf{y}^{n})}$ as $\eps \rightarrow
  0$ such that
    \begin{flalign} \label{eq:mu_limit_modified_CN}
         \dfrac{\mu_{gc}^{n+1} - \mu_{gc}^{n}}{\Delta t} \,=\, \mathcal{O}\left(  \Delta t^{2}  \right) \,
    \end{flalign}
    where $(\mathbf{y}^{n+1/2}, g^{n+1/2})$ is defined as  $\mathbf{y}^{n+1/2} \,=\, {(\mathbf{y}^{n+1} +
            \mathbf{y}^{n})}/{2}$ and  $g^{n+1/2} \,=\, {(g^{n+1} + g^{n})}/{2}$.
    \end{itemize}
    \end{proposition}
    
      \begin{proof}
We again follow the lines of the proof of Proposition~\ref{pro:CN}. To begin with, from the first line of \eqref{scheme:modified_CN} we deduce that $(\eps^{-1} \mathbf{w}^{n+1/2}_{\eps})_{\eps > 0}$ is
 uniformly bounded with respect to $\eps$. Combined with the second line of \eqref{scheme:modified_CN} this implies that each $(e^{n+1/2}_{\eps})_{\eps > 0}$ is also uniformly bounded with respect to $\eps$.

Therefore, up to a subsequence $(\mathbf{x}^{n}_{\eps}, e^{n}_{\eps})$ converges to $(\mathbf{y}^{n},g^{n})$ as $\eps$ goes to zero. One readily deduces that
\begin{equation} 
            \nonumber
            \lim_{\eps \rightarrow 0} \, \chi (\mathbf{w}_\eps^{n+1/2}, e_\eps^{n+1/2}) \,=\, g^{n+1/2} \,,
\end{equation}
so that from the third equation of \eqref{scheme:modified_CN} stems
    \begin{equation} 
            \nonumber
            \lim_{\eps \rightarrow 0} \dfrac{ \mathbf{w}^{n+1/2}_{\eps}}{\eps} \,=\, -\, \dfrac{ \mathbf{E}^{\perp}}{b}(\mathbf{y}^{n+1/2}) \,+\, g^{n+1/2} \, \dfrac{\nabla_{\mathbf{y}}^{\perp} b}{b^{2}}(\mathbf{y}^{n+1/2})\,.
      \end{equation}
Inserting the latter in the first and second equations of
\eqref{scheme:modified_CN} completes the derivation of
\eqref{eq:guiding_center_modified_CN}. Then, again, the convergence is
upgraded from the convergence of a subsequence to full convergence by
the uniqueness of solutions to the limiting system,
\eqref{eq:guiding_center_modified_CN}. Note that the system~\eqref{eq:guiding_center_modified_CN} is more nonlinear from previously derived asymptotic systems, which is why the required constraint on $\Delta t$ depends here on a priori bound on $g^0$ and $T$.

The rest of the proof, on the variations of the total discrete energy
and the discrete magnetic moment, is omitted as completely analogous to the corresponding one of Proposition~\ref{pro:CN}.
\end{proof}

Proposition \ref{pro:modified_CN} shows that the modified scheme \eqref{scheme:modified_CN} is consistent
uniformly with respect to $\eps$ and allows to recover a consistent approximation of the guiding center system
\eqref{eq:guiding_center_system}. Moreover, the new scheme also
preserves the second order accuracy for the total energy and the
magnetic moment. 

Let us also stress that a straightforward adaptation of the conservation trick \eqref{replacement_E} from \cite{SiTa92} provides a genuinely energy conserving version of the present scheme. Yet our numerical simulations, not reported here, show no further significant improvement so that we have decided not to delve further into this direction.

\textcolor{red}{It is worth noting that this scheme does not guarantee
  the nonnegativity of the variable $e_\eps^n$ for all $n \in \N$. If such a situation arises, one possible remedy is to update $e_\eps^n$ as
$$
	e_\eps^n  \,=\, \dfrac{1}{2} \| \mathbf{w}^{n}_{\eps} \|^{2} \,.
$$
However, in the upcoming numerical simulations, this scenario never occurs.}

%
%
\section{Numerical simulations}
\label{sec:Numerical_simulation}

In this section, we provide examples of numerical computations to validate and compare
the different time discretization schemes introduced in the previous
sections.  We first consider the motion of a single particle  under the effect of a given
electromagnetic field. It allows us to  illustrate the theoretical
results in the limit $\varepsilon\rightarrow 0$ of the numerical
schemes and their accuracy for
multi-scale problems.

Then we consider the Vlasov-Poisson system with an external non
uniform magnetic field.  We  apply a classical Particle-In-Cell
method with the  time discretization technique based on the
Crank-Nicolson scheme \eqref{scheme:modified_CN} to describe the diocotron
instability in a disk and also the stability of vortices in a D-shape domain.

\subsection{One single particle motion}
We first investigate  the motion of an individual particle in a given
electromagnetic field. 
We consider the electric field $\mathbf{E} = - \nabla_{\mathbf{x}} \phi$
where the potential  $\phi$ is given by
\begin{equation} \label{phi}
	\phi (\bx) \,=\,  \frac{x_{2}^{2}}{2} \,,
      \end{equation}
      while the external magnetic field is
\begin{equation} \label{magnetic_field}
    	b(\bx) \,=\, 1 \,+\, \|\bx\|^2\,.
\end{equation}
The initial condition is chosen as $\mathbf{x}^{0} = (2,2), \mathbf{v}^{0} =
(3,3)$ and the final time $T = 1$. On the one hand,  we compute  
reference solutions $(\bx_\eps, \bw_\eps, e_\eps)_{\eps > 0}$  and
$(\by, g)$  to the stiff
initial value problem \eqref{eq:ODE_system}  and to the asymptotic
problem \eqref{eq:guiding_center_system} thanks to an explicit
fourth-order Runge-Kutta scheme using  a small time step chosen
according to the size of order $\cO(\eps^2)$ for the initial system. On the
other hand,  for various time steps $\Delta t$, independent of $\eps$, we compute approximate solutions $(\bx_{\eps, \Delta t}, \bw_{\eps,
  \Delta t}, e_{\eps, \Delta t})$ using  the modified Crank-Nicolson
scheme \eqref{scheme:modified_CN} and also compare the results with those obtained using \eqref{scheme:Brackbill} proposed
in \cite{BrFo85, VuBr95}, and
\eqref{scheme:Chacon} described in \cite{RiCh20, ChCh23}.  For completeness, we also
compare our results with those obtained using an IMEX2L for the
augmented system \eqref{eq:augmented_ODE_system} developed
in \cite{FiRo16,FiRo17}. To evaluate the accuracy, the numerical error is measured as
\begin{flalign*}
    \begin{dcases}
        \| \mathbf{x}_{\eps, \Delta t} - \mathbf{x}_{\eps} \| := \dfrac{\Delta t}{T} \sum \limits_{n = 0}^{N_{T}} \| \mathbf{x}_{\eps, \Delta t}^{n} - \mathbf{x}_{\eps}(t^{n}) \|\,, \\
       \| \mathbf{x}_{\eps, \Delta t} - \mathbf{y} \| := \dfrac{\Delta t}{T} \sum \limits_{n = 0}^{N_{T}} \| \mathbf{x}_{\eps, \Delta t}^{n} - \mathbf{y}(t^{n}) \|\,, \\
       \| e_{\eps, \Delta t} - g \| := \dfrac{\Delta t}{T} \sum \limits_{n = 0}^{N_{T}} | e_{\eps, \Delta t}^{n} - g(t^{n}) |\,.
    \end{dcases}
  \end{flalign*}

In Figure \ref{Fig:Fig1}, we present the numerical error on
$\|\mathbf{x}_{\eps, \Delta t} - \mathbf{x}_{\eps} \|$ expressed with respect
to $\eps$ in log-log scale  for various time steps $\Delta t\in\{10^{-5},\cdots,10^{-1}\}$. When
$\eps\geq 10^{-1}$, we observe the expected second order accuracy of the
different schemes. However,  as $\eps$ becomes smaller, the time steps are
too large and the numerical error for both schemes \eqref{scheme:Brackbill}  and
\eqref{scheme:Chacon} increases. In contrast, the numerical error associated to the
modified Crank-Nicolson scheme \eqref{scheme:modified_CN} and the
IMEX2L \cite{FiRo17} decreases with respect to $\eps$. This behavior is typical
of an asymptotic preserving scheme for which the error becomes of order
$\eps$ when the time step $\Delta t$ is sufficiently large. It is worth mentioning that
the behavior of the errors for \eqref{scheme:modified_CN} and the
IMEX2L  differs significantly as $\eps\ll 1$. Indeed, even with a large
time step the  modified Crank-Nicolson scheme
\eqref{scheme:modified_CN} remains so accurate that the error  of order
$\Delta t^2$ is negligible compared to the error with respect to
$\eps$.  This \textcolor{red}{phenomenon} can also be observed  in Figures \ref{Fig:Fig2}
and \ref{Fig:Fig3},  where we report the errors  compared to the
reference solution of the asymptotic model $\| \mathbf{x}_{\eps, \Delta t} - \mathbf{y} \|$ and  $\| e_{\eps,
  \Delta t} - g \|$. Clearly the schemes \eqref{scheme:Brackbill}  and
\eqref{scheme:Chacon}  do not capture a consistent approximation
$(\by,g)$ to  the asymptotic solutions
\eqref{eq:guiding_center_system} as $\eps\rightarrow 0$. These
numerical experiments 
illustrate  the lack of consistency  proven in Propositions \ref{pro:Brackbill_scheme} and
\ref{pro:Chacon_scheme}.  In contrast, both  schemes
\eqref{scheme:modified_CN}  and the IMEX2L, based on the approximation
of the augmented system \eqref{eq:augmented_ODE_system}, successfully capture the limit with the correct convergence rate with
respect to $\eps$ (slope of order one). The advantage of the scheme
\eqref{scheme:modified_CN} is that, when $\eps\ll 1$, the amplitude of
the numerical error is much smaller than the one corresponding to
other schemes. 

\begin{figure}
	\centering	
     	{\includegraphics[width=0.49\linewidth]{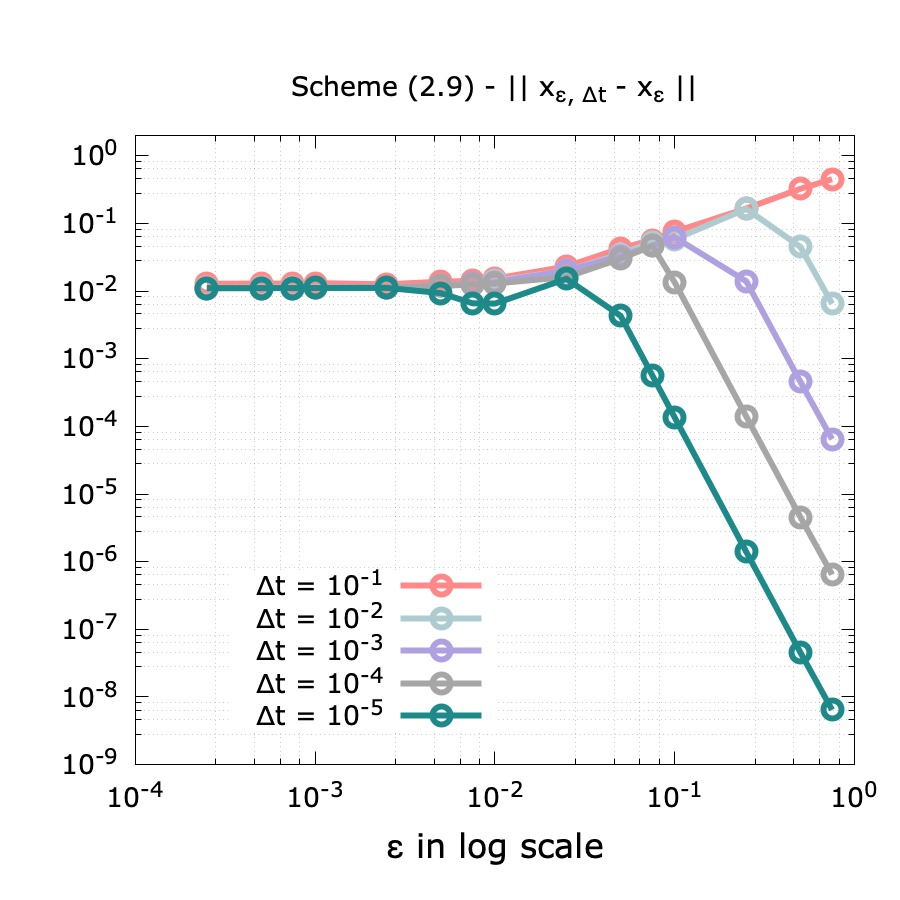}}
	{\includegraphics[width=0.49\linewidth]{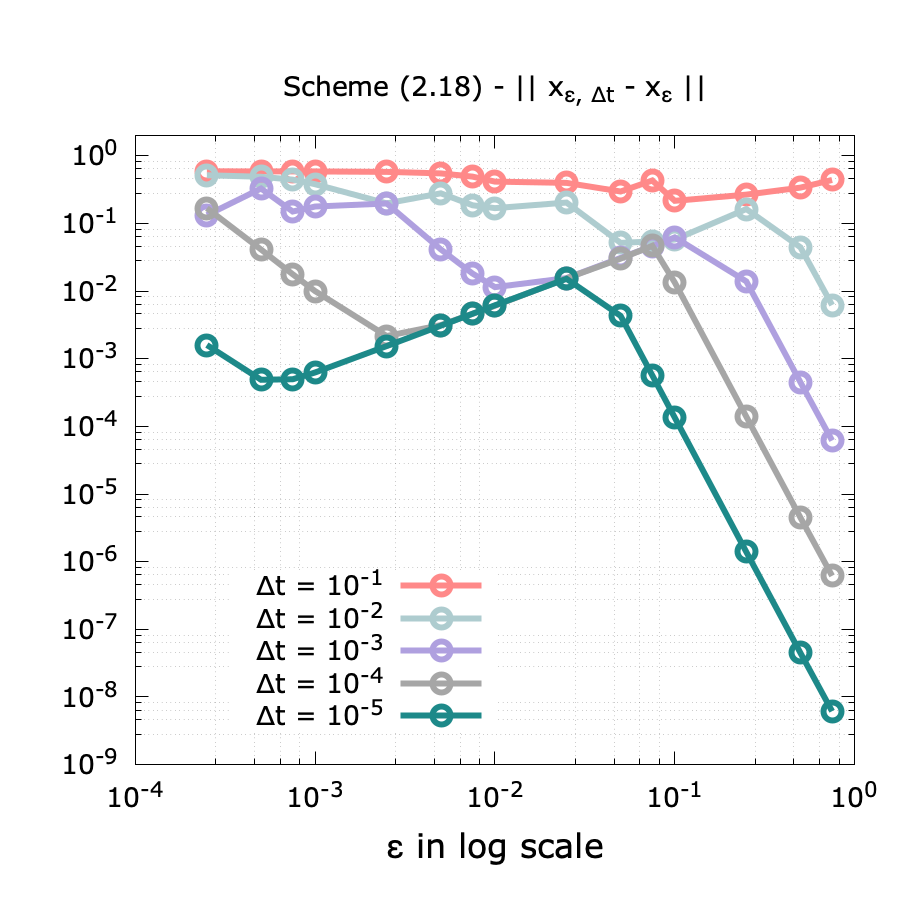}}
	{\includegraphics[width=0.49\linewidth]{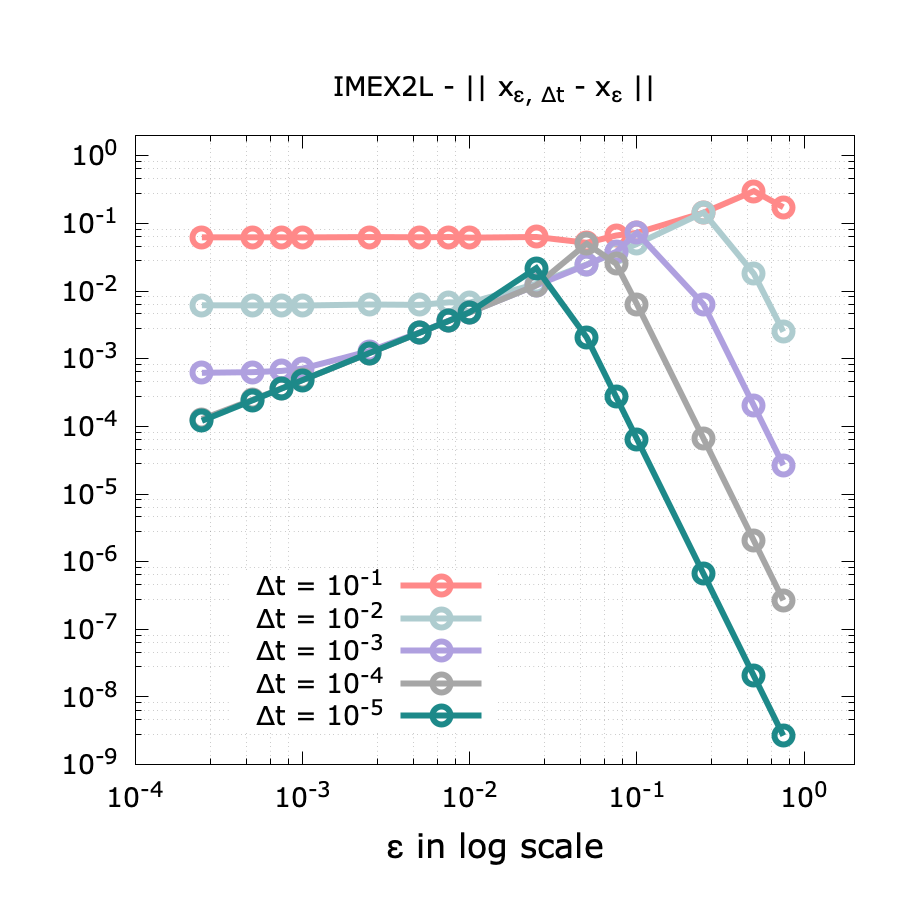}}
       	{\includegraphics[width=0.49\linewidth]{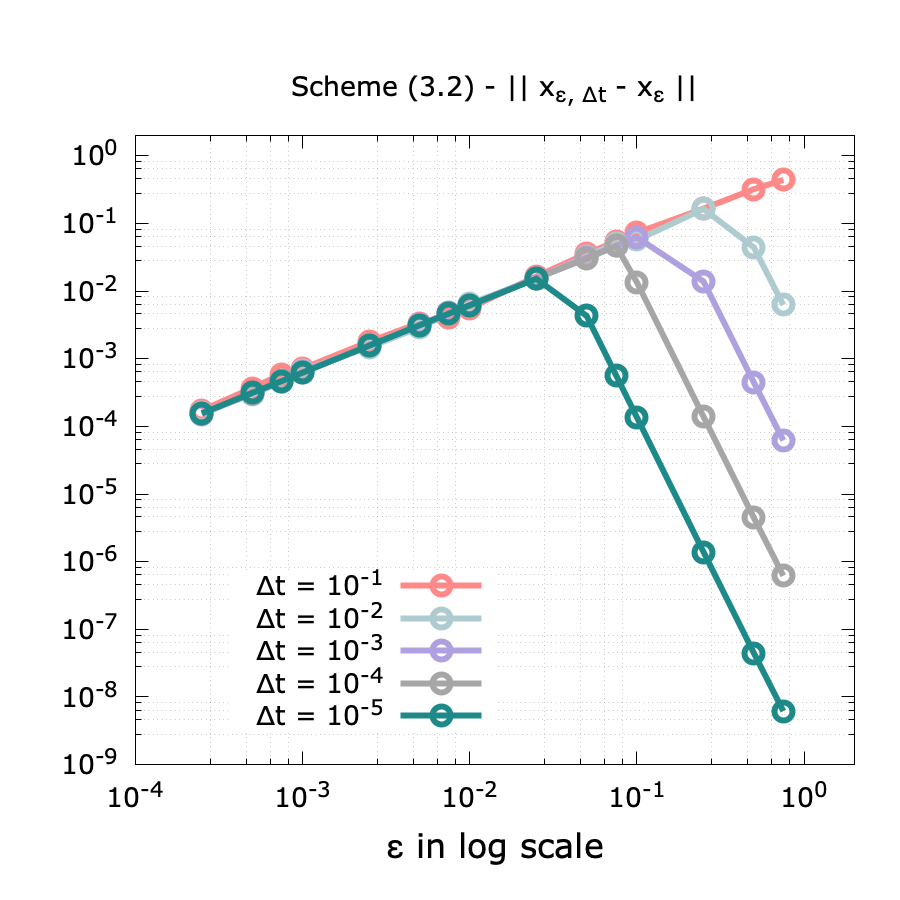}}
	\caption{\textbf{One single particle motion:} Numerical errors of discrete solution $\mathbf{x}_{\eps, \Delta t}$, approximated by several schemes: \eqref{scheme:Brackbill}, \eqref{scheme:Chacon}, IMEX2L and \eqref{scheme:modified_CN}, with reference solution $\mathbf{x}_\eps$ of \eqref{eq:ODE_system} for various $\eps > 0$ and $\Delta t > 0$.}
	\label{Fig:Fig1}
\end{figure}
\begin{figure}
	\centering	
	{\includegraphics[width=0.49\linewidth]{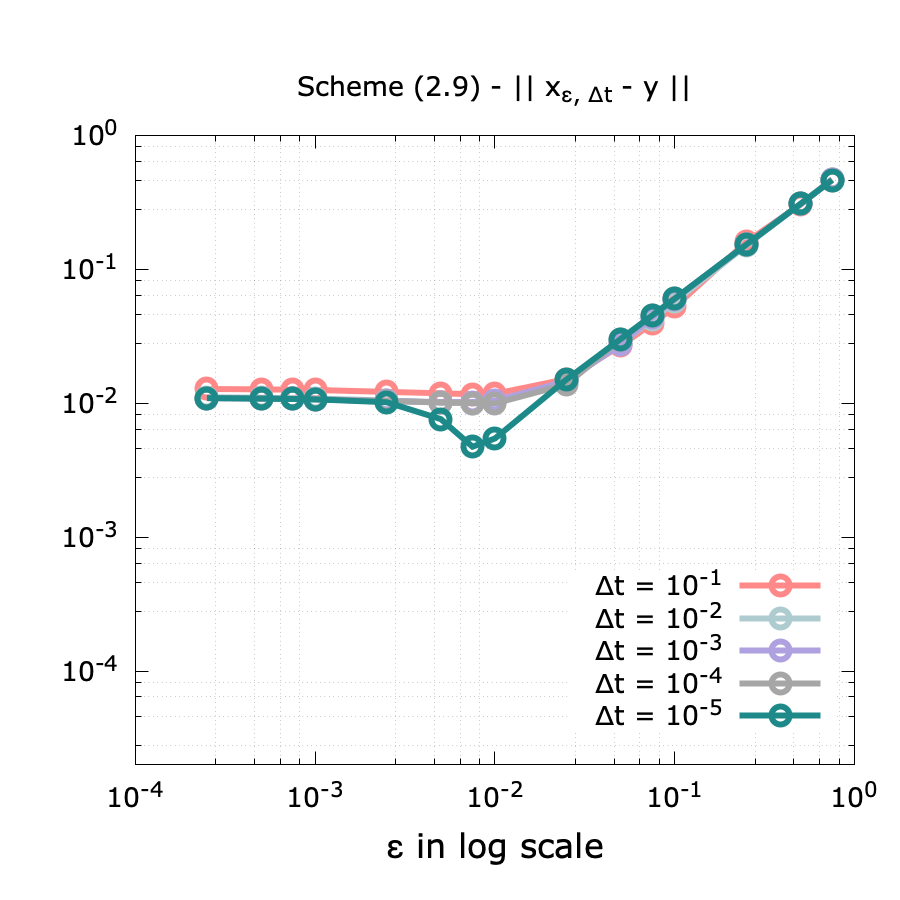}}
	{\includegraphics[width=0.49\linewidth]{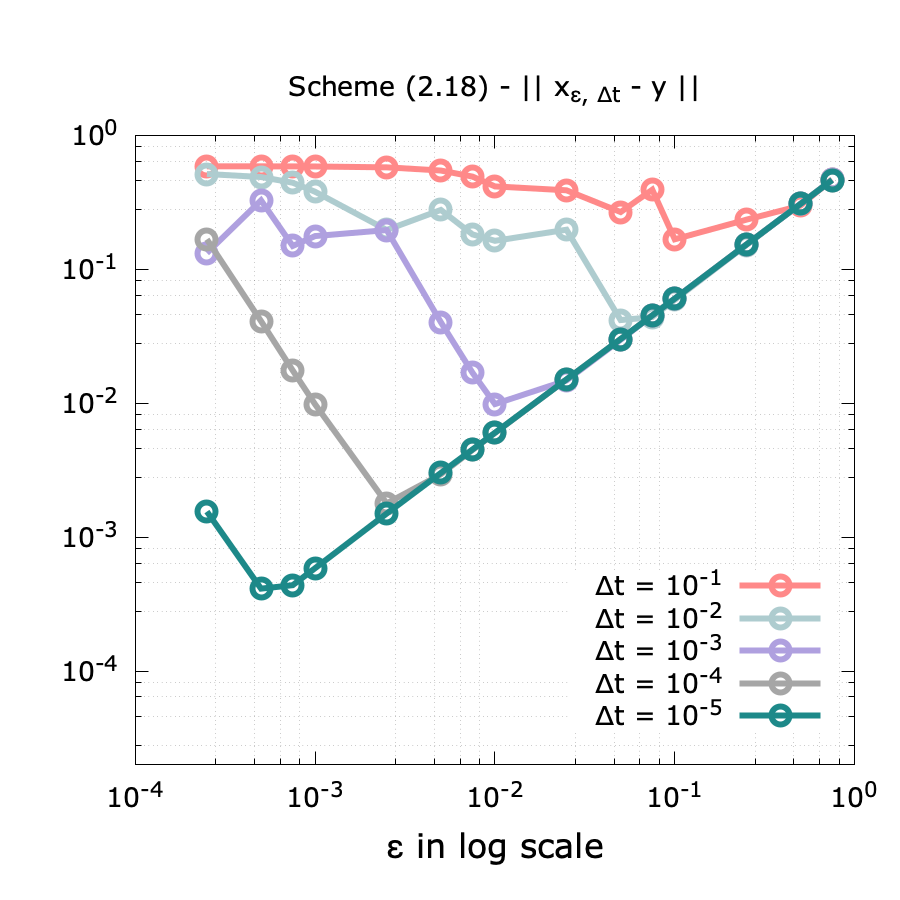}}
	{\includegraphics[width=0.49\linewidth]{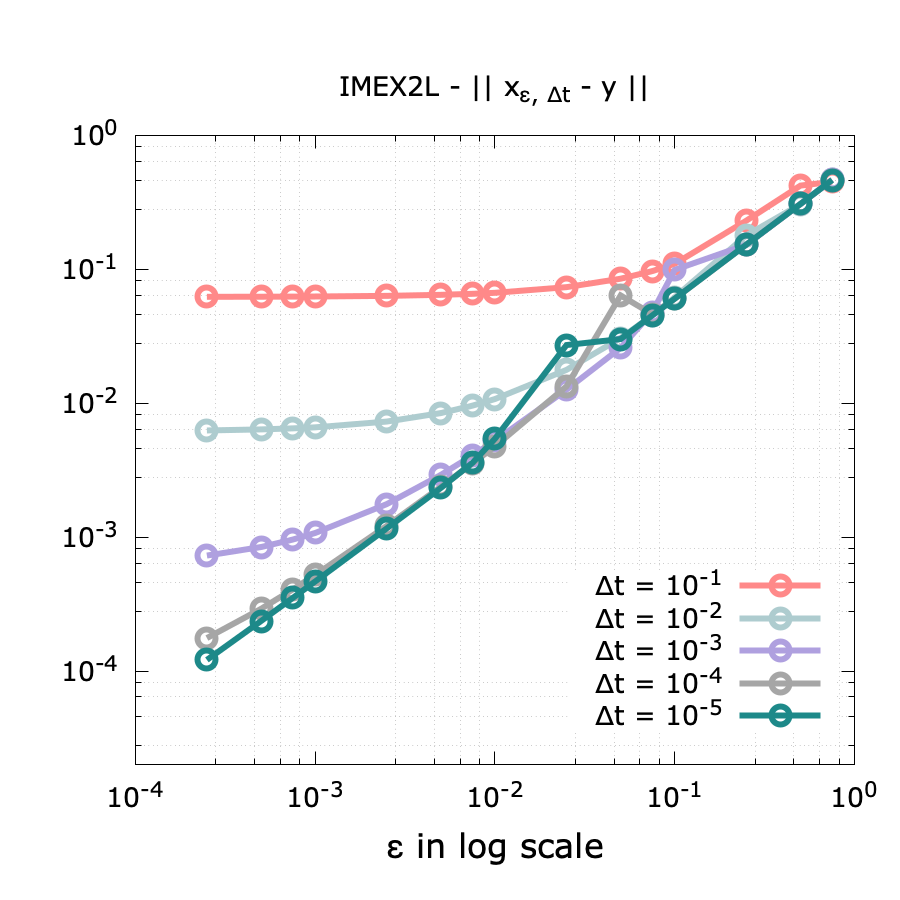}}
       	{\includegraphics[width=0.49\linewidth]{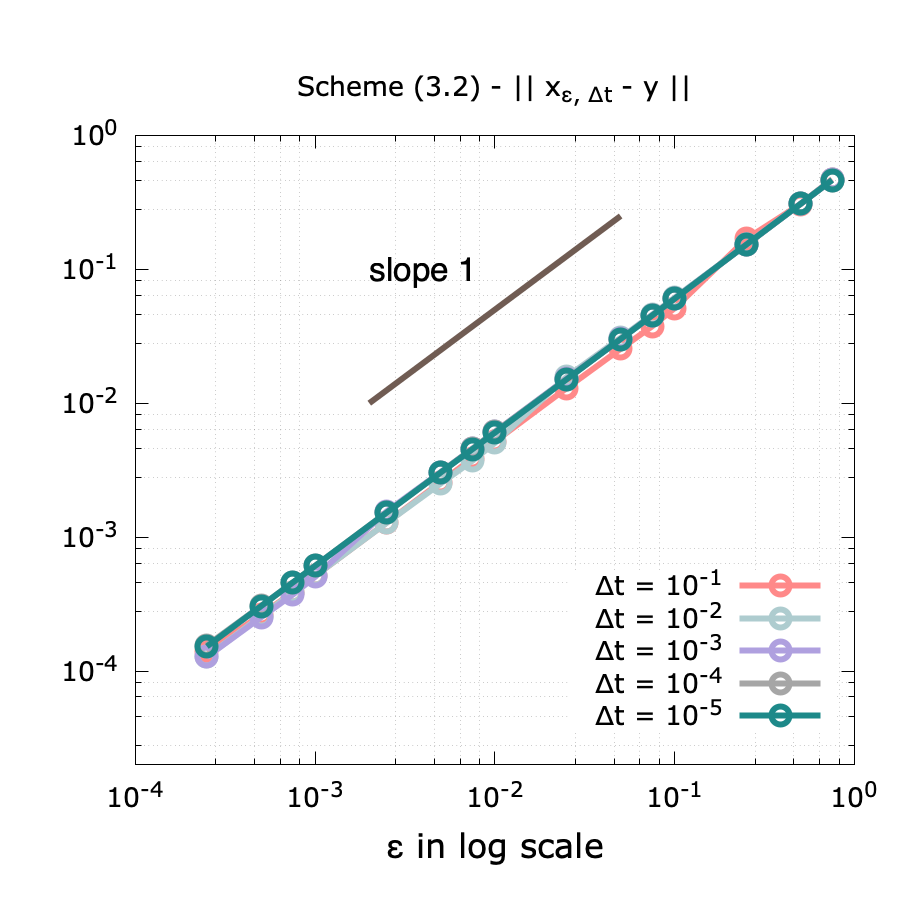}}
	\caption{\textbf{One single particle motion:} Numerical errors of discrete solution $\mathbf{x}_{\eps, \Delta t}$, approximated by several schemes: \eqref{scheme:Brackbill}, \eqref{scheme:Chacon}, IMEX2L and \eqref{scheme:modified_CN}, with guiding center solution $\mathbf{y}$ of \eqref{eq:guiding_center_system} for various $\eps > 0$ and $\Delta t > 0$.}
	\label{Fig:Fig2}
\end{figure}
\begin{figure}
	\centering	
	{\includegraphics[width=0.49\linewidth]{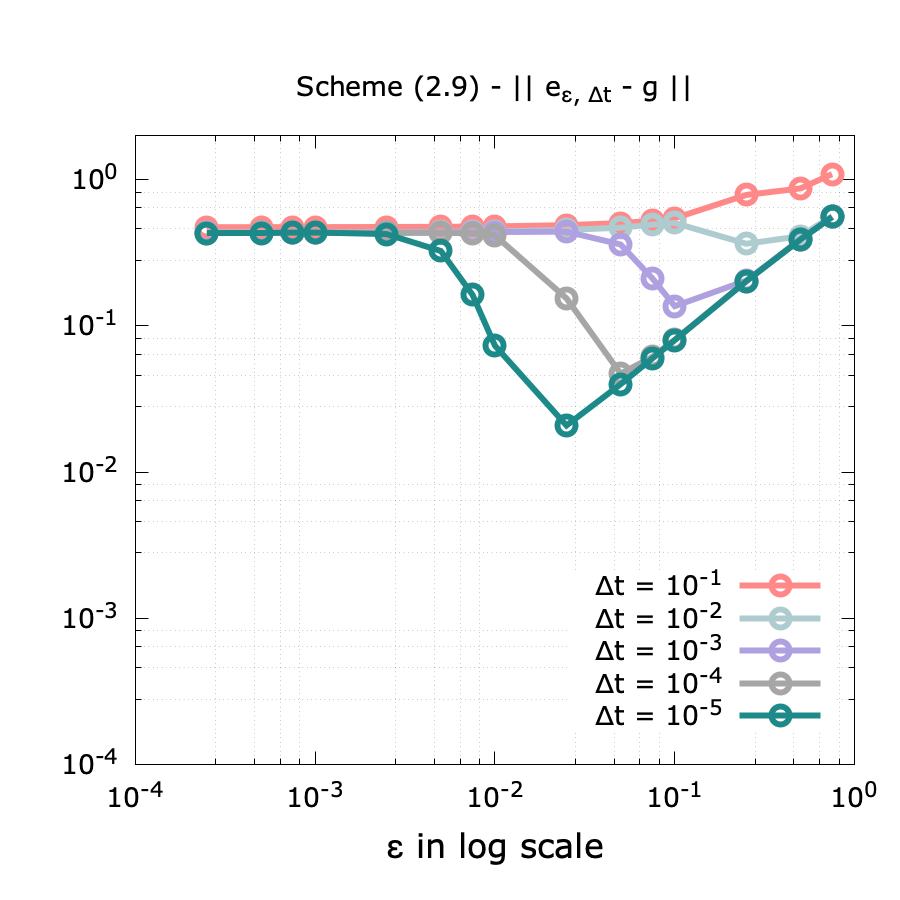}}
	{\includegraphics[width=0.49\linewidth]{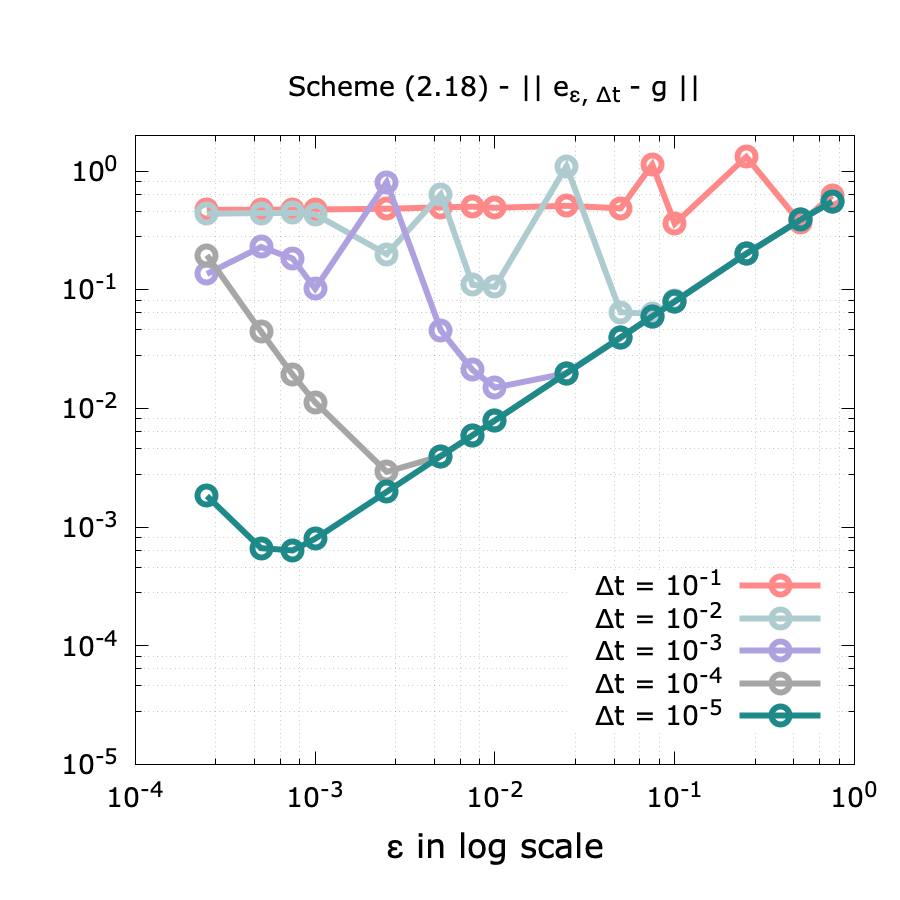}}
	{\includegraphics[width=0.49\linewidth]{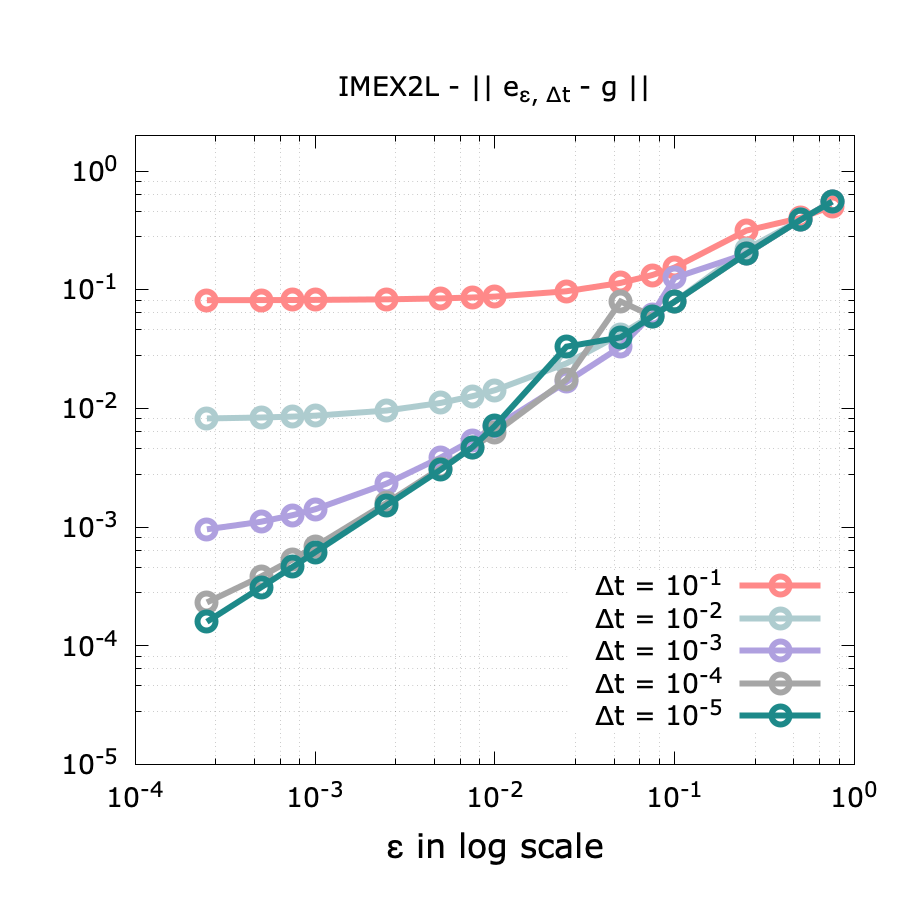}}
       	{\includegraphics[width=0.49\linewidth]{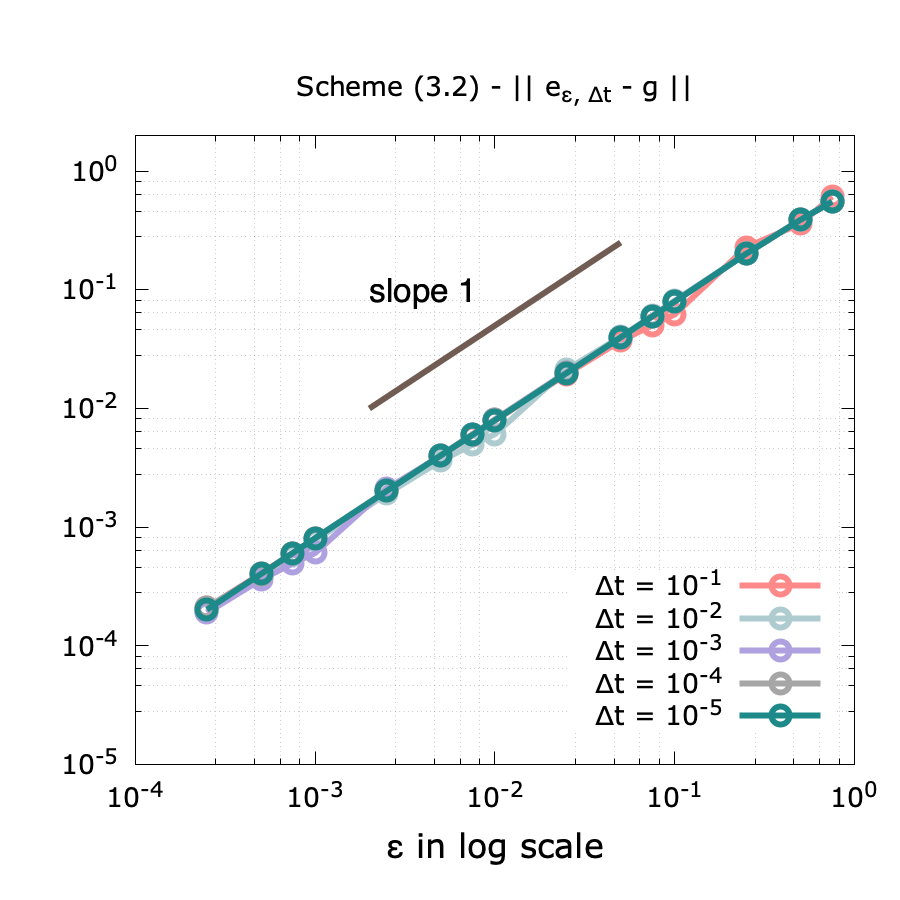}}
	\caption{\textbf{One single particle motion:} Numerical errors of discrete solution $e_{\eps, \Delta t}$, approximated by several schemes: \eqref{scheme:Brackbill}, \eqref{scheme:Chacon}, IMEX2L and \eqref{scheme:modified_CN}, with guiding center solution $g$ of \eqref{eq:guiding_center_system}  for various $\eps > 0$ and $\Delta t > 0$.}
	\label{Fig:Fig3}
\end{figure}


To illustrate this point,  we also present the space trajectories
corresponding to $\eps = 0.01$ and  $\Delta t = 0.1$ for large
time simulations (with $T = 30$). In Figure \ref{Fig:Fig4},  we
observe that the particle
trajectory  forms a circular motion under the effect of  both drifts
$\mathbf{E}^{\perp}/b$ and $\nabla_{\mathbf{x}}^{\perp} b/b^{2}$, which
appear explicitly in the asymptotic model \eqref{eq:guiding_center_system}.  The scheme
\eqref{scheme:modified_CN} provides an approximation close to the
reference solution, while other schemes fail to get  the correct position
of the particle, since they do not capture correctly the drift
$\nabla_{\mathbf{x}}^{\perp} b/b^{2}$ in the limit $\eps \rightarrow 0$.

\begin{figure}
	\centering	
    	{\includegraphics[width=0.49\linewidth]{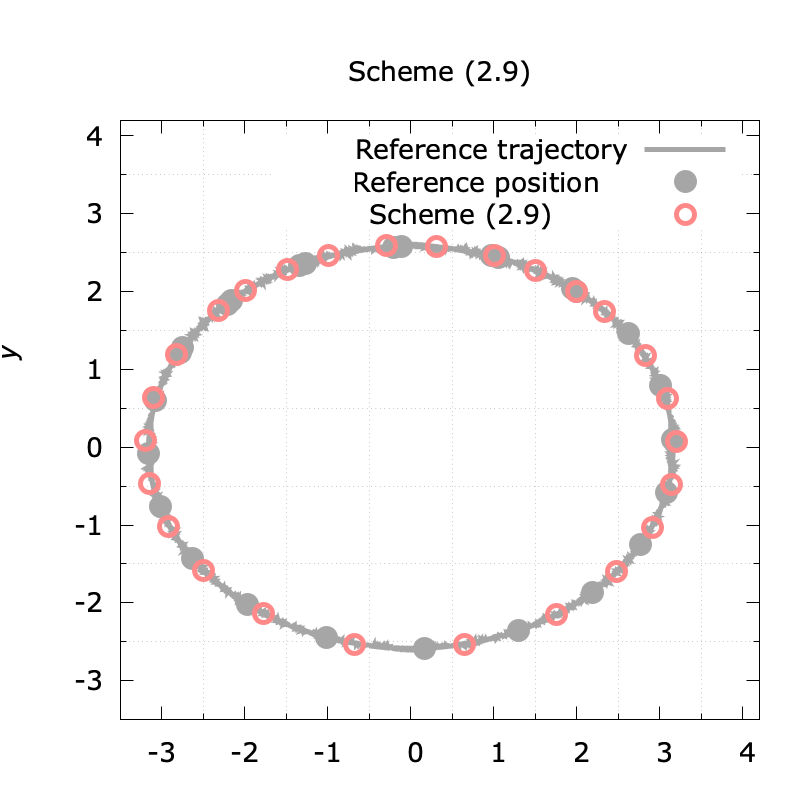}}
	   {\includegraphics[width=0.49\linewidth]{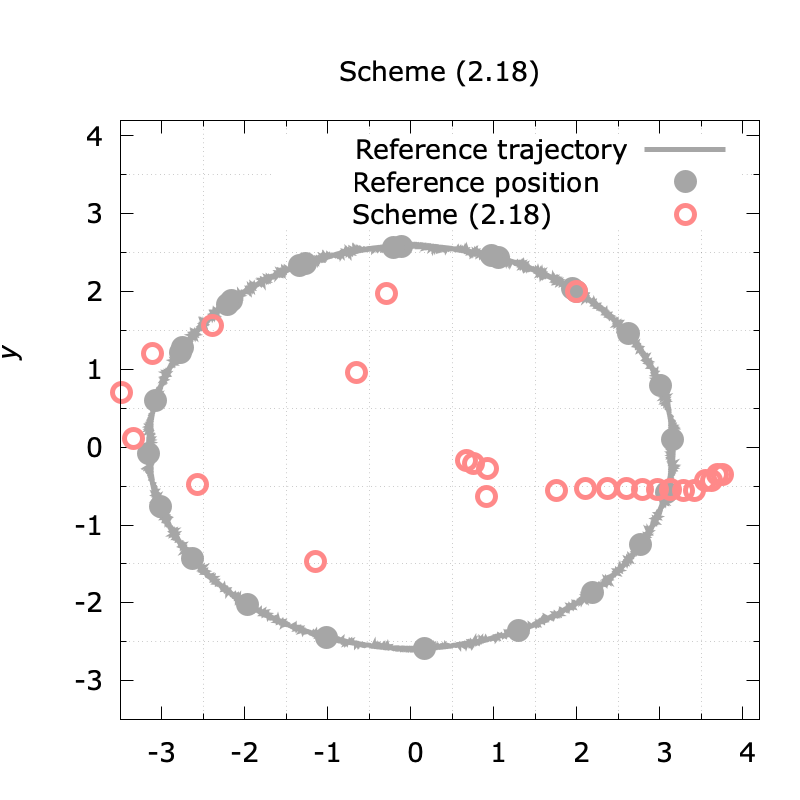}}
	   {\includegraphics[width=0.49\linewidth]{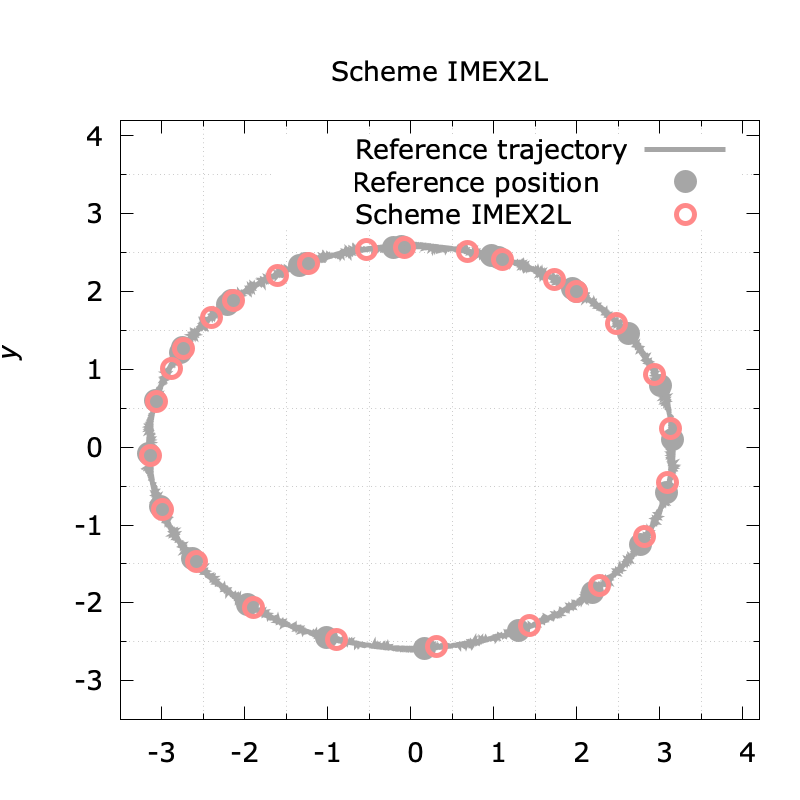}}
       	{\includegraphics[width=0.49\linewidth]{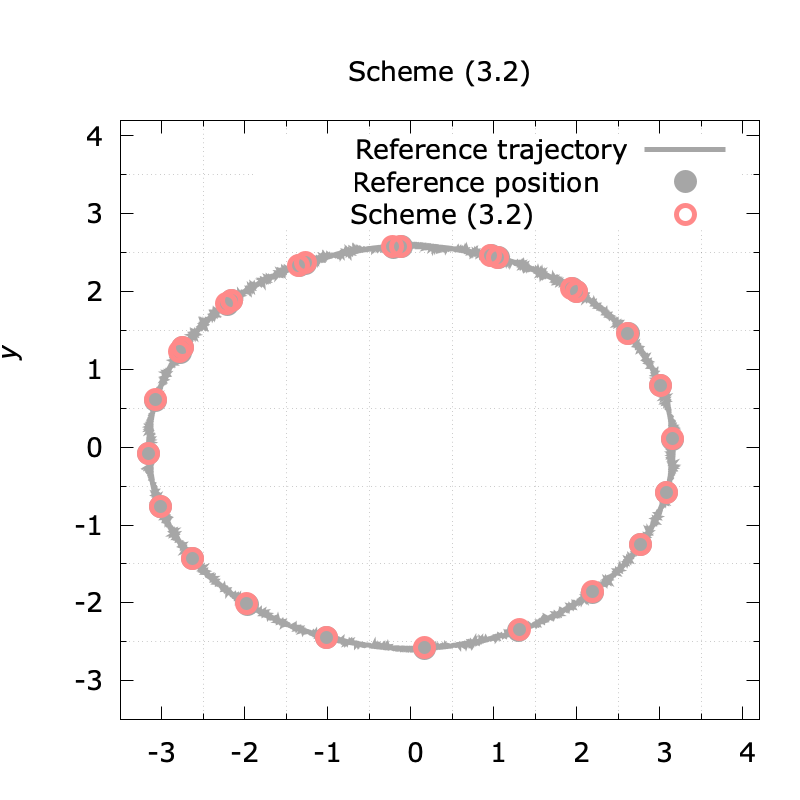}}
	\caption{\textbf{One single particle motion:} Trajectory of particle approximated by several schemes: \eqref{scheme:Brackbill}, \eqref{scheme:Chacon}, IMEX2L and \eqref{scheme:modified_CN} with $\eps = 0.01$,  $\Delta t = 0.1$ and final time $T = 30s$.}
	\label{Fig:Fig4}
\end{figure}

Furthermore, from the results presented in Figure
\ref{Fig:Fig5}, we observe that the IMEX2L and
\eqref{scheme:modified_CN} schemes accurately track the variations in kinetic and potential energy over a long period, unlike
\eqref{scheme:Brackbill} and \eqref{scheme:Chacon}.

\begin{figure}
	\centering	
    	{\includegraphics[width=0.49\linewidth]{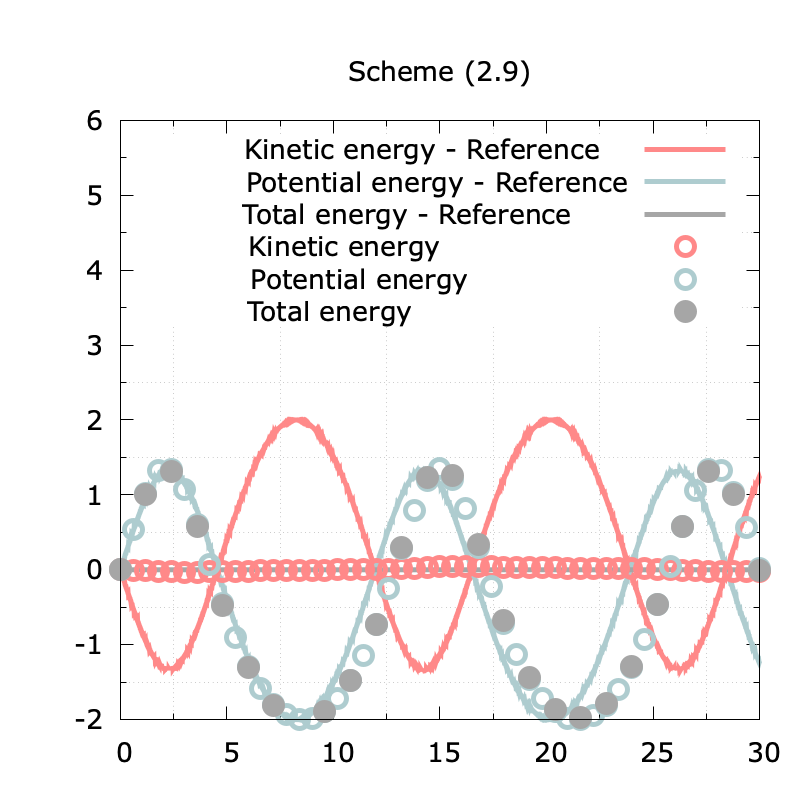}}
	{\includegraphics[width=0.49\linewidth]{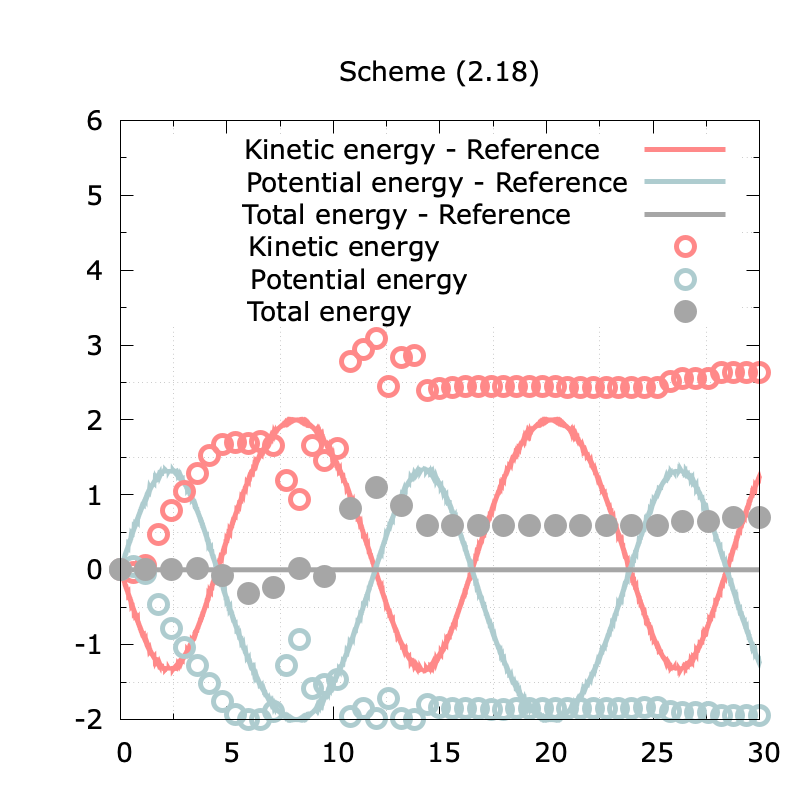}}
	{\includegraphics[width=0.49\linewidth]{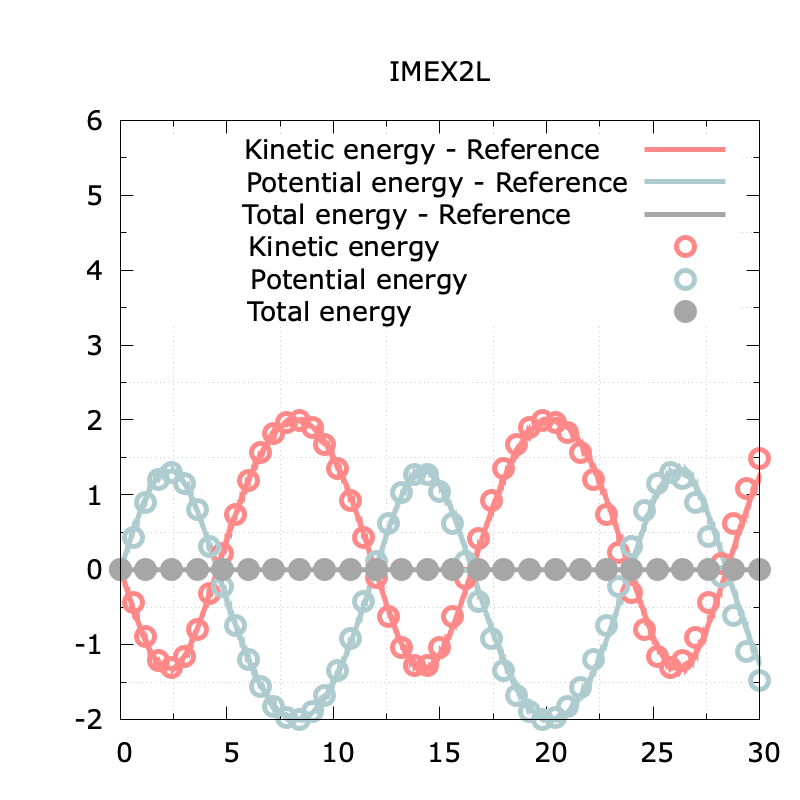}}
       	{\includegraphics[width=0.49\linewidth]{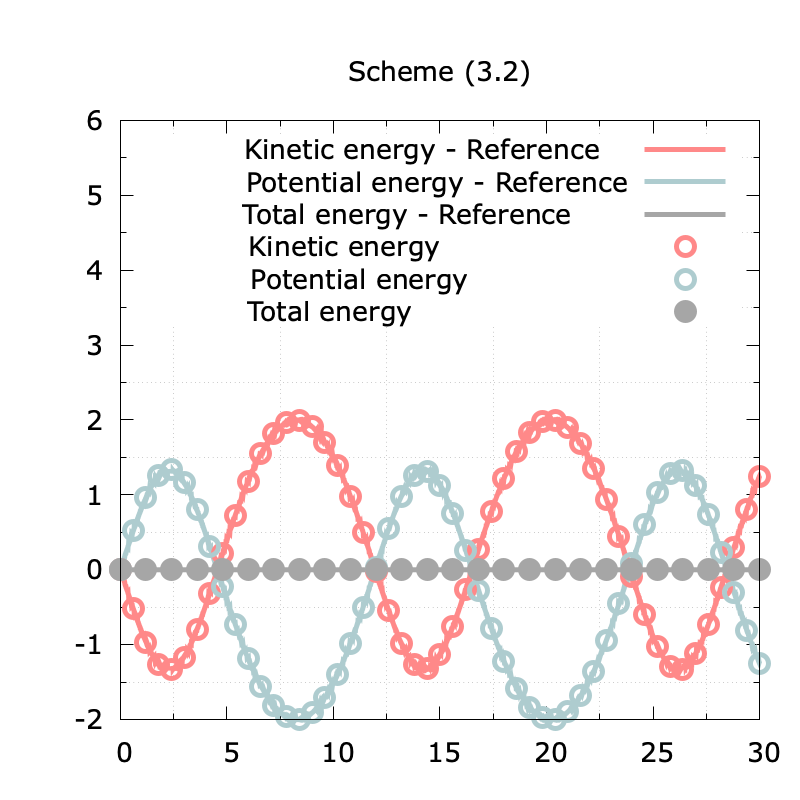}}
	\caption{\textbf{One single particle motion:} Variation of energy approximated by several schemes: \eqref{scheme:Brackbill}, \eqref{scheme:Chacon}, IMEX2L and \eqref{scheme:modified_CN} with $\eps = 0.01$,  $\Delta t = 0.1$ and final time $T = 30s$.}
	\label{Fig:Fig5}
\end{figure}

Finally, the time evolution of the magnetic moment approximation
is presented and compared with a reference solution in Figure \ref{Fig:Fig6}. Let us emphasize that only the
modified scheme \eqref{scheme:modified_CN} accurately describes the
amplitude of the fast oscillations of $\mu_\eps$ compared to the other
schemes. Indeed, even if the time step $\Delta t$ is much larger than
the fastest time scale of order $\eps^2$, the quantity $\mu_{\eps,
  \Delta t}$ oscillates with the correct  amplitude of order
$\eps^2$. 

\begin{figure}
	\centering	
    	{\includegraphics[width=0.49\linewidth]{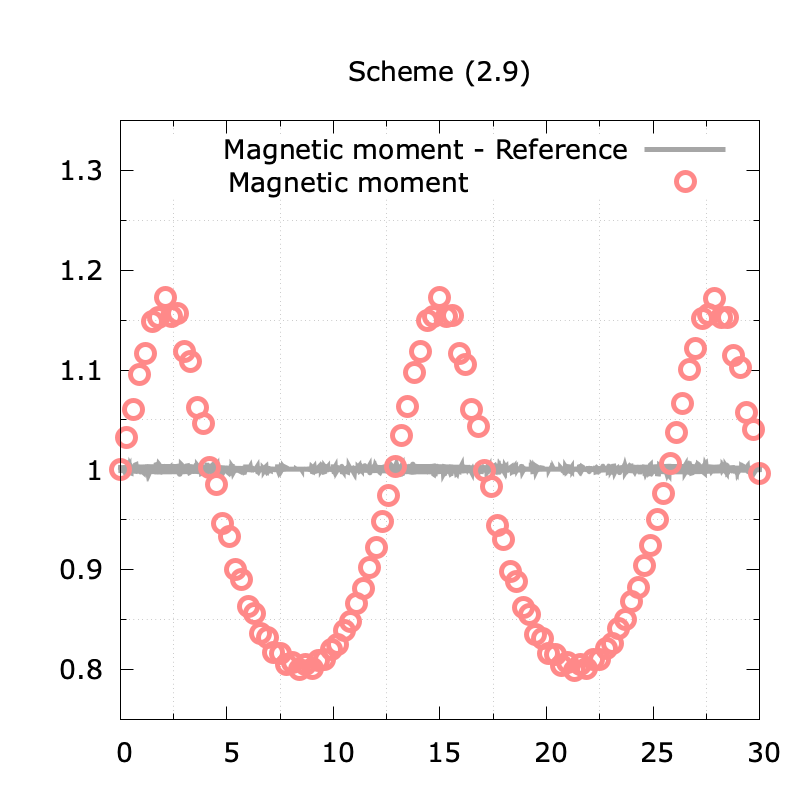}}
	{\includegraphics[width=0.49\linewidth]{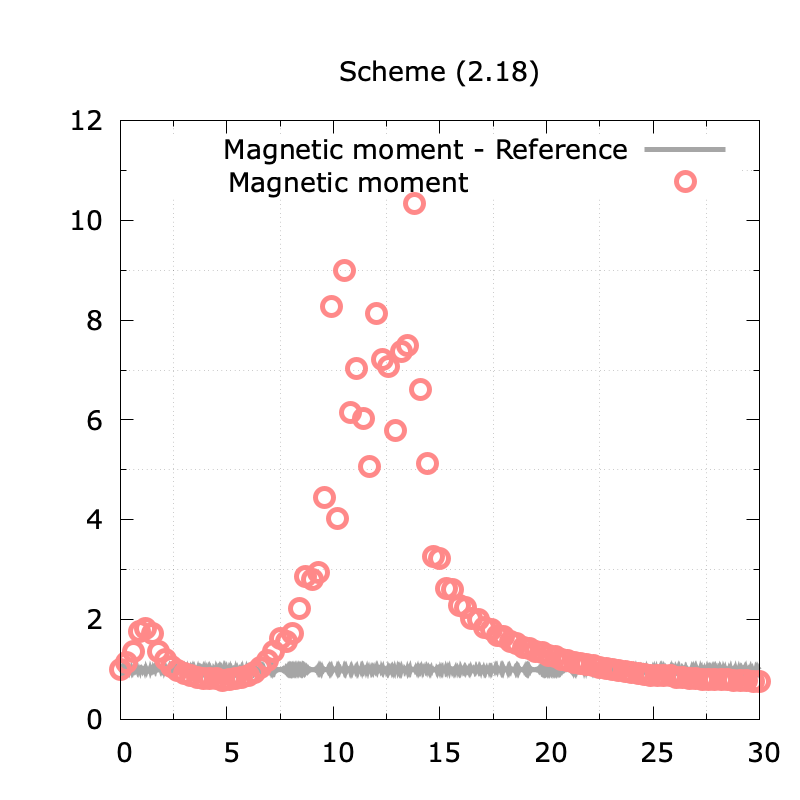}}
        {\includegraphics[width=0.49\linewidth]{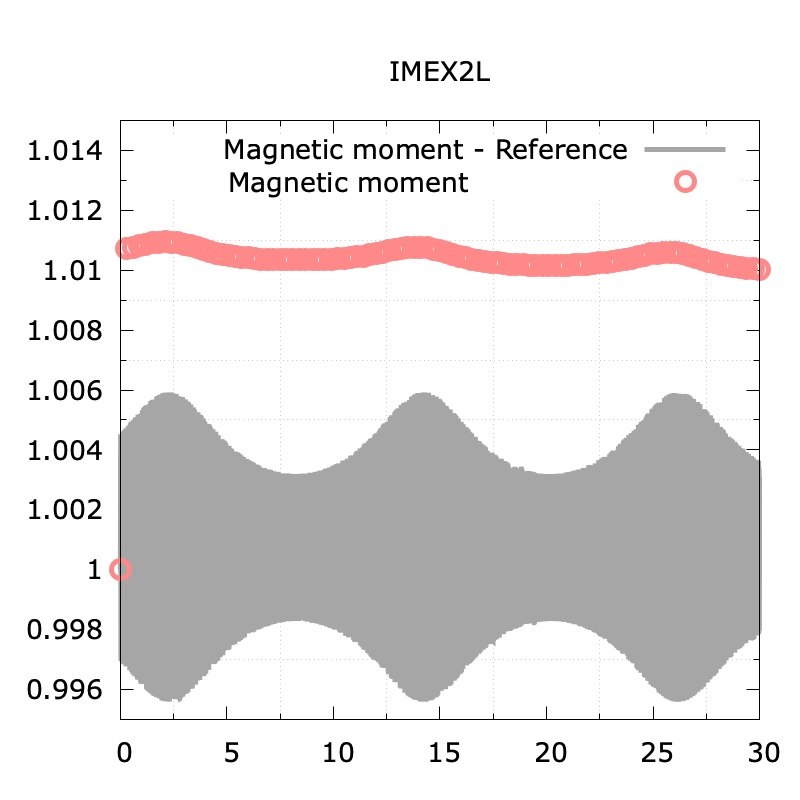}}
       	{\includegraphics[width=0.49\linewidth]{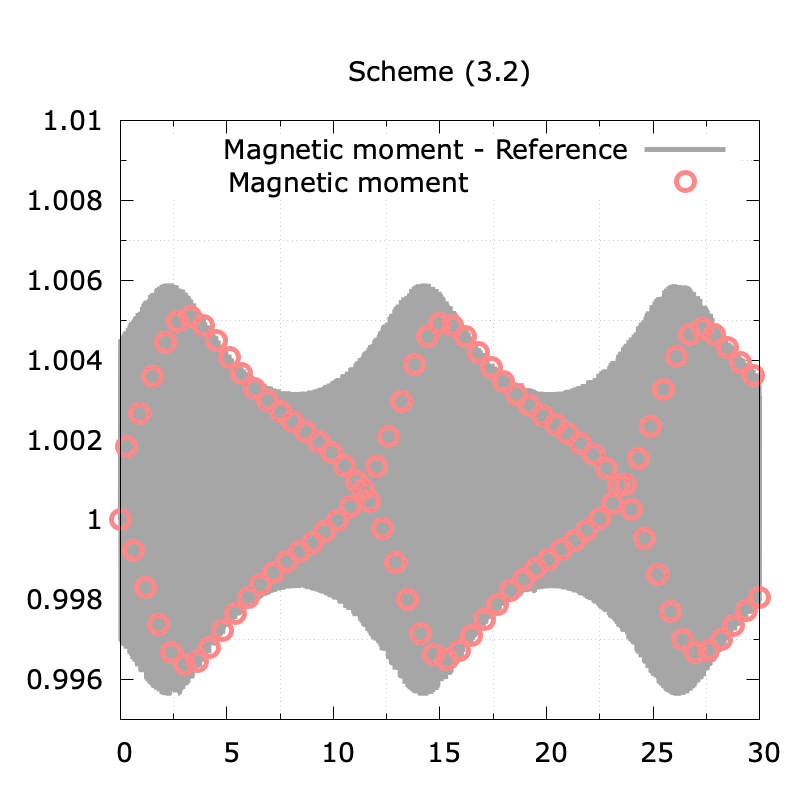}}
	\caption{\textbf{One single particle motion:} Evolution of
          magnetic moment approximated by several schemes: \eqref{scheme:Brackbill}, \eqref{scheme:Chacon}, IMEX2L and \eqref{scheme:modified_CN} with $\eps = 0.01$,  $\Delta t = 0.1$ and final time $T = 30s$.}
	\label{Fig:Fig6}
\end{figure}

These numerical results on the single particle motion, clearly show
that the modified
Crank-Nicolson scheme  \eqref{scheme:modified_CN} is much more accurate than 
the schemes \eqref{scheme:Brackbill}, \eqref{scheme:Chacon} or IMEX2L,
even when  $\eps \ll 1$,  with a \textcolor{red}{fixed} time step $\Delta t$
independent of $\eps$. These elementary
numerical simulations confirm the ability of the modified
Crank-Nicolson scheme \eqref{scheme:modified_CN} to capture the
evolution of the "slow" variables $(\bx_\eps, e_\eps)$ uniformly with
respect to $\eps$ by essentially transitioning automatically to the
guiding center motion \eqref{eq:guiding_center_system} \textcolor{red}{when it is not
able anymore to follow} the fast oscillations of the initial system.

\textcolor{red}{
 Let us remind that the scheme proposed in \cite{RiCh20} uses both an effective force
and an adaptive time-stepping procedure as $\varepsilon$ becomes
small. Our goal here is is to deepen the understanding of the
underlying issues, particularly by distinguishing the effects of the
adaptive time step procedure from those of using an effective force.
In particular, we illustrate the theoretical results obtained in
Sections \ref{sec:Review_schemes} and \ref{sec:Modified_CN}  and demonstrate that the inclusion of a slow variable allows
us to dispense with the adaptive time step procedure, enabling the use
of an arbitrarily large time step independent of
$\varepsilon$. Applying the adaptive time step procedure allows
significantly to decrease the error when $\eps$ goes to zero, but it
requires a small time step depending on $\eps$, which affects the
computational cost. }

\textcolor{red}{
To end this section dedicated to comparing the various numerical
schemes, Table \ref{Tab:Tab2} presents the computational cost of the IMEX2L scheme,
which does not require an iterative method, alongside that of the
modified Crank-Nicolson scheme (3.2). Since the number of iterations
does not exceed 5 at each time step for all $\Delta t$ and
$\varepsilon$, it is observed that the computational time for scheme
(3.2) is more than twice that of the IMEX2L scheme. However, the error
associated with \eqref{scheme:modified_CN}   is significantly smaller, especially when
$\varepsilon \ll 1$.}
\begin{table}
\begin{center}
  \begin{tabular}{| m{9.em} | m{6.em} |m{6.em} | m{6.em} |  m{6.em} |} 
\hline 
    & $\Delta t = 10^{-1} $ &  $\Delta t = 10^{-2} $  &  $\Delta t = 10^{-3} $ &  $\Delta t = 10^{-4} $\\ 
  \hline
  $\eps = 10^{-1}$  & & & & \\
  \hline
  Scheme  \eqref{scheme:modified_CN} &5 &34 &314 &2572 \\ 
  \hline
  IMEX2L &2 &17 &176 &1663 \\ 
  \hline
  \hline
  $\eps = 10^{-2}$ & & & &\\ 
  \hline
  Scheme  \eqref{scheme:modified_CN}  &5 &65 &289 &3399\\ 
  \hline
  IMEX2L &2 &17 &177 &1803 \\ 
  \hline
  \hline
  $\eps = 10^{-3}$ & & & &\\ 
  \hline
  Scheme  \eqref{scheme:modified_CN} &5 &31 &265 &2617 \\ 
  \hline
  IMEX2L &2 &17 &183 &1656 \\ 
\hline
  \end{tabular}
 \caption{Comparison of the computational time (microsecond) of the scheme \eqref{scheme:modified_CN} with IMEX2L, with final time $T = 1$ for various $\eps > 0$ and $\Delta t > 0$.}
 \label{Tab:Tab2}
\end{center}
\end{table}


Let us conclude this section with an important remark about strongly
oscillating fields.
\textcolor{red}{
\begin{remark}
	The scheme \eqref{scheme:modified_CN} and the strategy
        proposed in  \cite{FiRo20}  can also
        be  applied to the case where the electric field is highly
        oscillatory, that is, $\| \partial_{t} \mathbf{E} \| / \|
        \mathbf{E} \| \,=\, \mathcal{O}(1 / \eps)$. In  such a
        situation, the asymptotic limit remains given by
        \eqref{eq:guiding_center_system}, and the scheme
        \eqref{scheme:modified_CN} can be applied directly. Indeed,  we  performed
        numerical simulations  (not presented in this paper) with a
        potential
        $$
\phi(t,\bx) \,=\, \frac{1}{2}\,\cos\left(\frac{t}{\eps}\right) \, \|\bx\|^2               
        $$
        and obtained the same error curves  as those presented in
        this section.
      \end{remark}}

\subsection{Vlasov-Poisson system}
 \textcolor{red}{ We now consider the Vlasov-Poisson system \eqref{eq:VP_system} on a
domain $\Omega \subset \mathbb{R}^{2}$, where $\Omega$ is given either
by a disk or a D shape domain (see Figure \ref{Fig:Fig65}). }

\begin{figure}
	\centering	
    	{\includegraphics[width=0.49\linewidth]{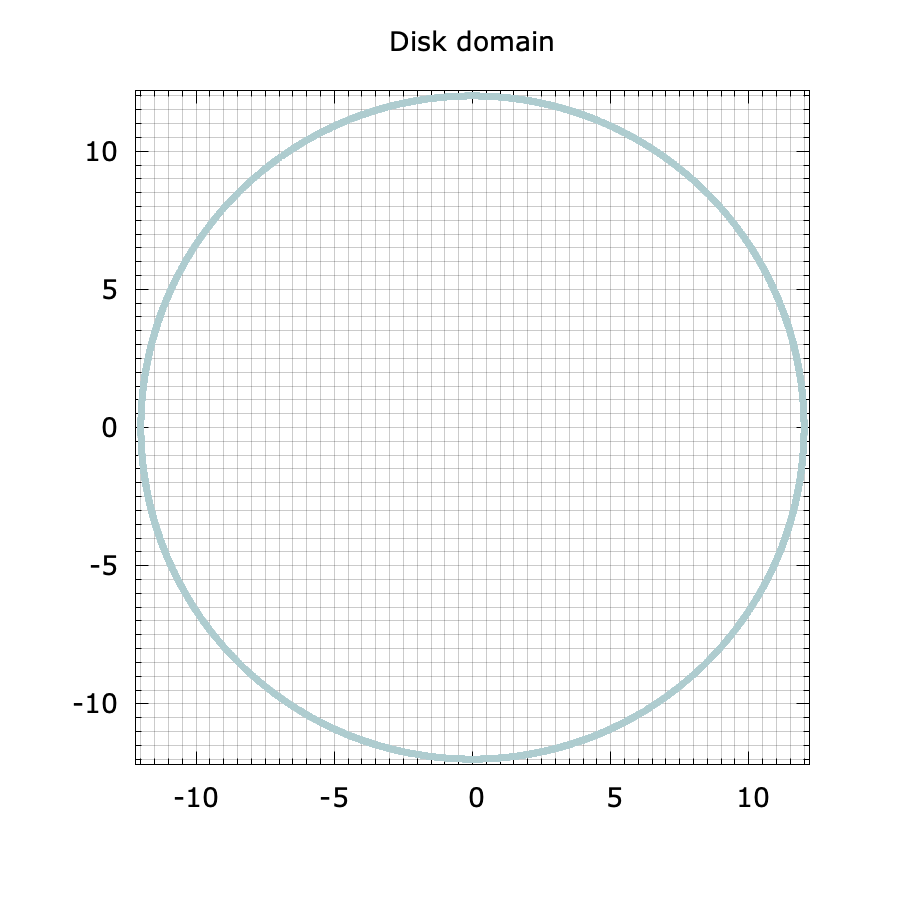}}
	{\includegraphics[width=0.49\linewidth]{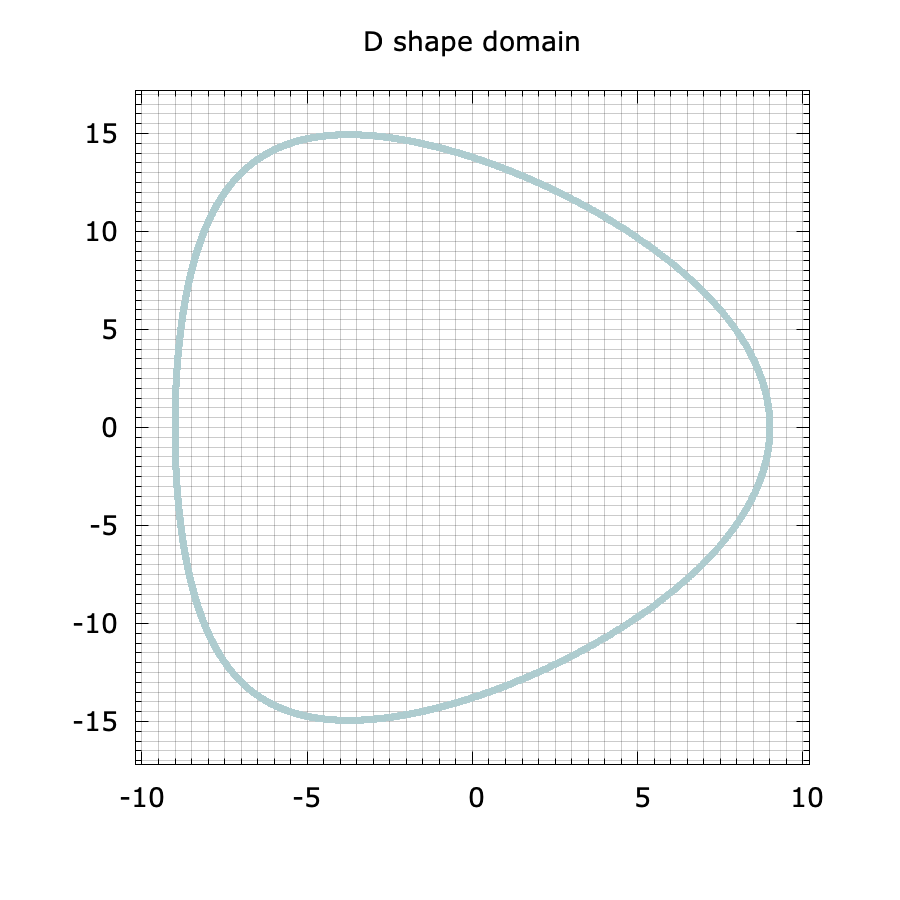}}
	\caption{Disk domain (left) and D-Shape domain (right) embedded in uniform Cartersian grid .}
	\label{Fig:Fig65}
\end{figure}

\textcolor{red}{ Assuming that the density is
concentrated far from the boundary,  we choose to
  remove particles which are located outside the physical domain, this may induce a lack of mass conservation. 
However, thanks to the strong confinement of the magnetic field, we do not observe this situation in
the present simulations. For the potential $\phi_\eps$, the Poisson equation is
solved with homogeneous Dirichlet boundary conditions using a
classical five points second order  finite
difference method with ghost points to take into account the effect of
the boundary conditions \cite{FY14}. }
\\

For each time $t\geq 0$, we may define the total energy $\cE_\eps(t)$ as
\begin{flalign}
	\label{Total_energy_VP}
	\cE_\eps(t)\,\, :=\,\, \bK_\eps(t) \,+\, \bU_\eps(t) , 
\end{flalign}
where  the kinetic energy  $\bK_\eps(t)$ and the
potential energy  $\bU_\eps(t)$  are given by
\[
\left\{
  \begin{array}{l}
    \ds\bK_\eps(t) \, := \, \frac{1}{2}\int \limits_{\Omega} \int \limits_{
    \mathbb{R}^{2}} f_{\eps} (t, \mathbf{x}, \mathbf{v}) \, \|\mathbf{v}
    \|^{2} \, \dD\mathbf{v} \,\dD \mathbf{x}, \\[0.9em]
    \ds\bU_\eps(t) \, := \,  \frac{1}{2} \int \limits_{\Omega} \|
    \mathbf{E}_{\eps}(t, \mathbf{x}) \|^{2} \, \dD \mathbf{x}\,.
    \end{array}\right.
\]
Assuming that the distribution function is compactly supported in the
open set $\Omega$,  the total energy  $\cE_\eps(t)$ is conserved
for all  time $t\geq 0$.

We also define the magnetic moment for the Vlasov-Poisson system \eqref{eq:VP_system}  given by
\begin{flalign} 
	\label{mu_VP}
	\mu_\eps(t) \,:=\, \frac{1}{2}\,\int \limits_{\Omega} \int
        \limits_{ \mathbb{R}^{2}} f_{\eps} (t, \mathbf{x}, \mathbf{v})
        \dfrac{\| \mathbf{v} \|^{2}}{ b(\mathbf{x})} \,\dD \mathbf{v} \,\dD
        \mathbf{x}, \qquad t\geq 0
\end{flalign}
and expect that $\mu_\eps(t)$ is an invariant in time in the asymptotic limit
$\eps \rightarrow 0$ for the limit model \eqref{eq:guiding_center_system}.

For this section,  we performed numerical experiments using the
modified Crank-Nicolson scheme \eqref{scheme:modified_CN} to approximate the particles
trajectory corresponding to the Vlasov equation.  Despite the fact that the modified Crank-Nicolson scheme
\eqref{scheme:modified_CN} does not conserves exactly the total energy, we
expect that its variations are of order $\cO(\Delta t^2)$ even when $\eps$ tends to zero.  Furthermore,  the modified Crank-Nicolson scheme
\eqref{scheme:modified_CN}  should capture correctly  the asymptotic
limit $\eps \rightarrow 0$, as it has been  shown for the single
particle motion. 

\subsubsection{Diocotron instability} \label{Dio_Disk}
We first consider Vlasov-Poisson system \eqref{eq:VP_system} set in a
disk $\Omega = D(0,12)$ centered at the origin with a  radius
$R=12$ \textcolor{red}{ as in Figure \ref{Fig:Fig65}}. \textcolor{red}{Here, the Particle-In-Cell method is
  implemented with approximatively  100 particles per cell  on a
uniform grid of  the square $ (-12,12)^2$  with  $\Delta \bx = 0.1$.} The
simulation starts with a Maxwellian distribution in velocity, whose
macroscopic density is a perturbed  uniform distribution in an annulus. More precisely, we choose
\begin{equation*}
    f(0,\mathbf{x}, \mathbf{v}) \,=\, \dfrac{\rho_{0}(\mathbf{x})}{2 \pi} \,\exp \left(-\frac{\| \mathbf{v} \|^{2}}{2} \right),
\end{equation*}
where $\rho_{0}$ is given by
\begin{flalign*}
	\rho_{0}(\mathbf{x}) = 
	\left\{\begin{array}{l}
 	n_{0}(1 \,+\, \alpha \,\cos(7 \theta)),	\,\, {\rm for }\,\,6 \leq \| \mathbf{x} \| \leq 7\,, \\[0.9em]
 	0,	\,\, {\rm else,} 
	\end{array}\right. 
\end{flalign*}
in which $n_{0} = 0.25$, $\alpha = 0.001$, and the angle $\theta$ is
defined as $\theta = \arctan(x_{2}/x_{1})$ with $\mathbf{x} = (x_{1},
x_{2})$. Moreover, we consider strong external magnetic field which is given by
\begin{equation*} 
	b(t, \mathbf{x}) = \dfrac{20}{\sqrt{400 - \|\bx\|^{2}}}.
\end{equation*}

Since $b$ is not homogeneous, even in the asymptotic
regime the kinetic and potential parts of the total energy are not
preserved separately, but the total energy corresponding to the
Vlasov–Poisson system is still preserved. Figure \ref{Fig:Fig7} shows
that all these features are captured satisfactorily by the modified
Crank-Nicolson scheme \eqref{scheme:modified_CN}  even on long time
evolutions with a large time step $\Delta t = 0.1$ and small $\eps=
10^{-2}$.  On the one hand, the variations of the total energy have an amplitude of
order  $10^{-4}$, which is satisfying compared to the physical variations of
the potential and kinetic energy of order $2\times \,10^{-2}$. On the other
hand, the quantity $\mu_{\eps}$ also varies around $10^{-4}$, which
corresponds to the scale of $\eps^2=10^{-4}$. This \textcolor{red}{phenomenon} has already been
observed for the single particle motion and will be discussed below.

\begin{figure}
	\centering	
    	{\includegraphics[width=0.49\linewidth]{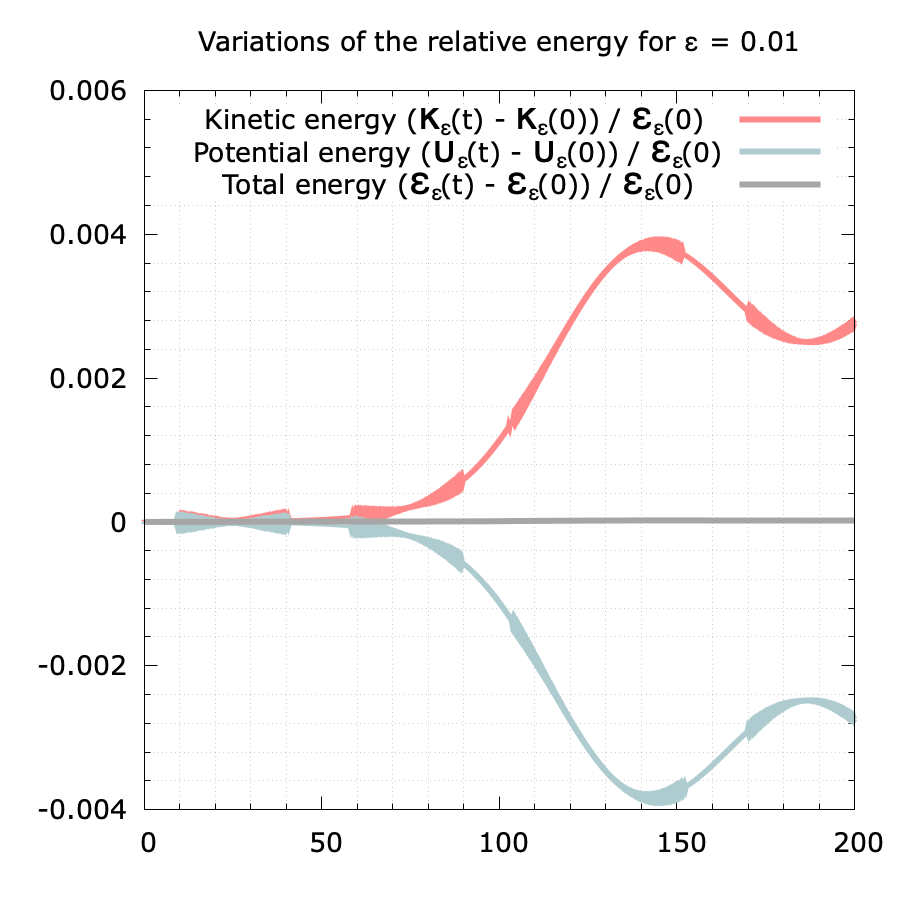}}
	{\includegraphics[width=0.49\linewidth]{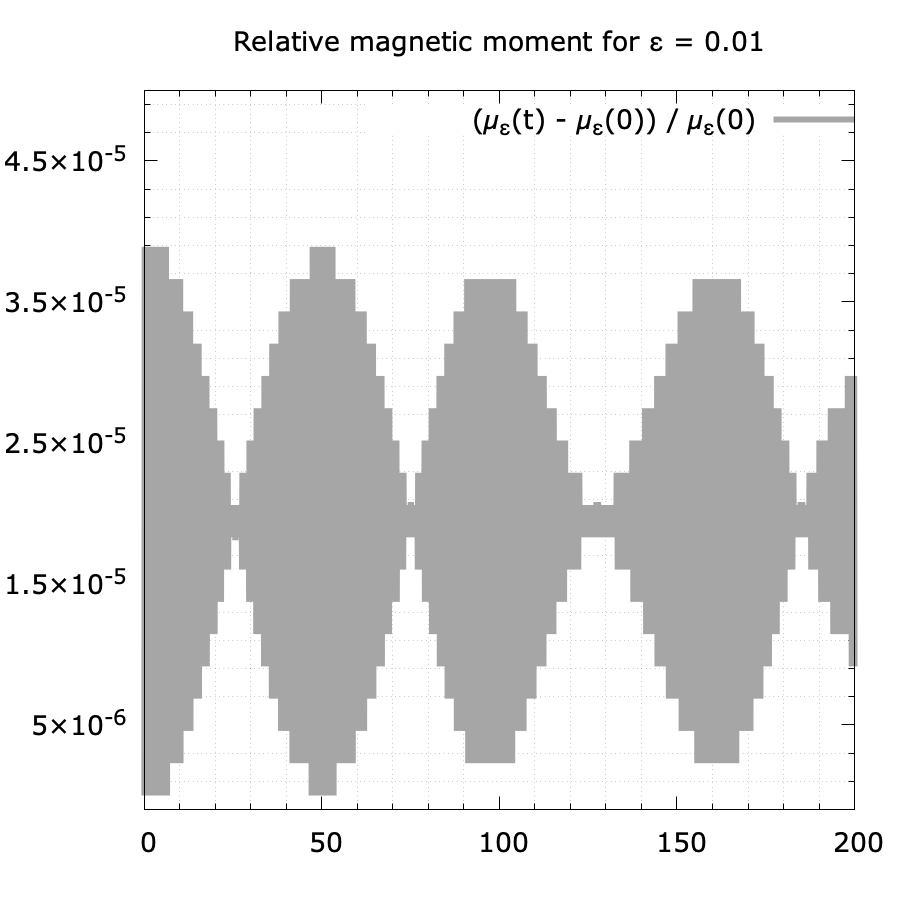}}
	\caption{\textbf{Diocotron instability:} Time evolution of the
          variations of the relative potential $\bU_\eps$ and kinetic energy $\bK_\eps$ (left)
          and magnetic moment $\mu_\eps$ (right) with
          $\eps=10^{-2}$ and $(\Delta
          t,\Delta \bx) = (0.1,0.1)$, using  the modified Crank-Nicolson
          scheme \eqref{scheme:modified_CN}.}
	\label{Fig:Fig7}
\end{figure}

In Figure \ref{Fig:Fig8}, we visualize the corresponding dynamics by
presenting several snapshots of the macroscopic density at some specific
times $t = 0$, $50$, $100$ and $150$. The numerical results obtained
with our PIC methods are in good agreements with those obtained with a
finite difference scheme \cite{FiYa18}  for the asymptotic model \eqref{eq:gc}.

\begin{figure}
	\centering	
    	{\includegraphics[width=0.49\linewidth]{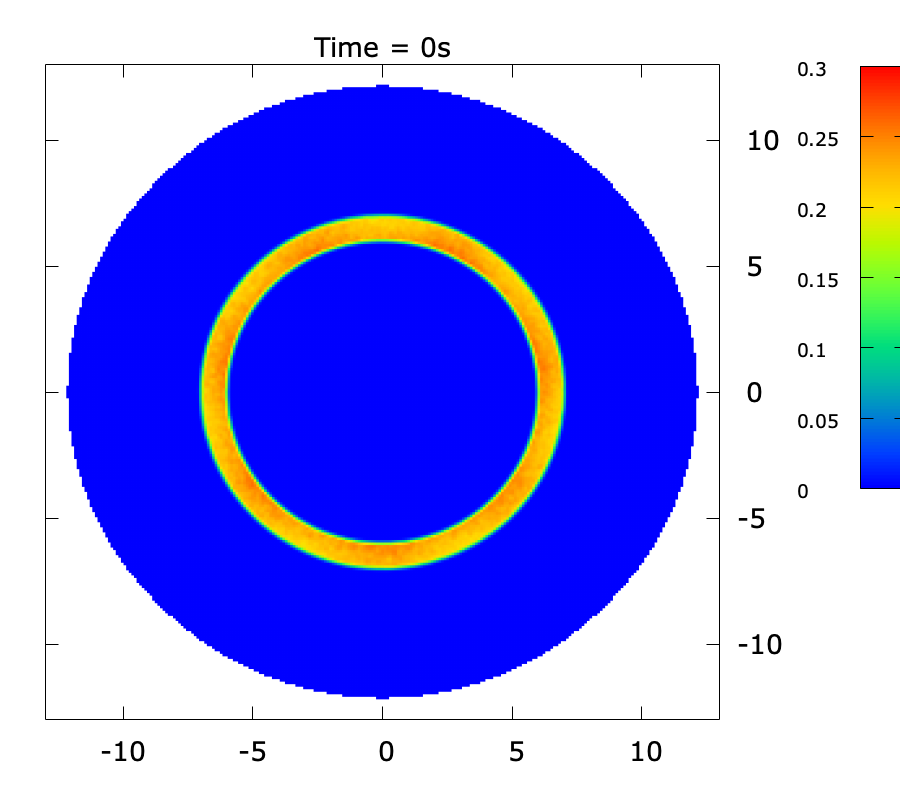}}
	{\includegraphics[width=0.49\linewidth]{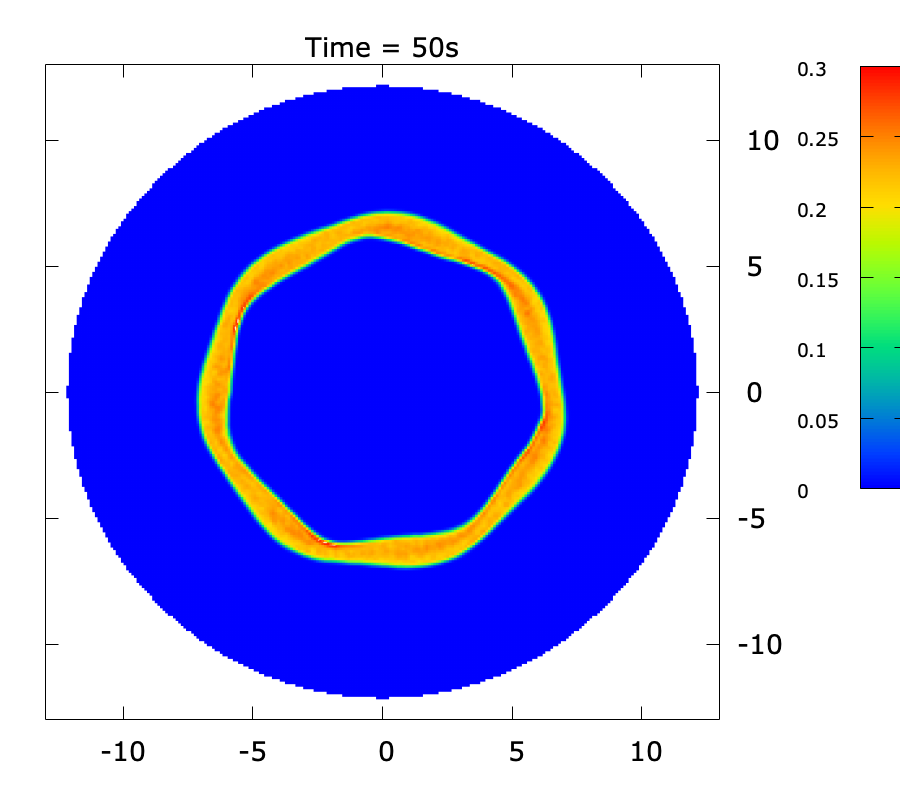}}
	{\includegraphics[width=0.49\linewidth]{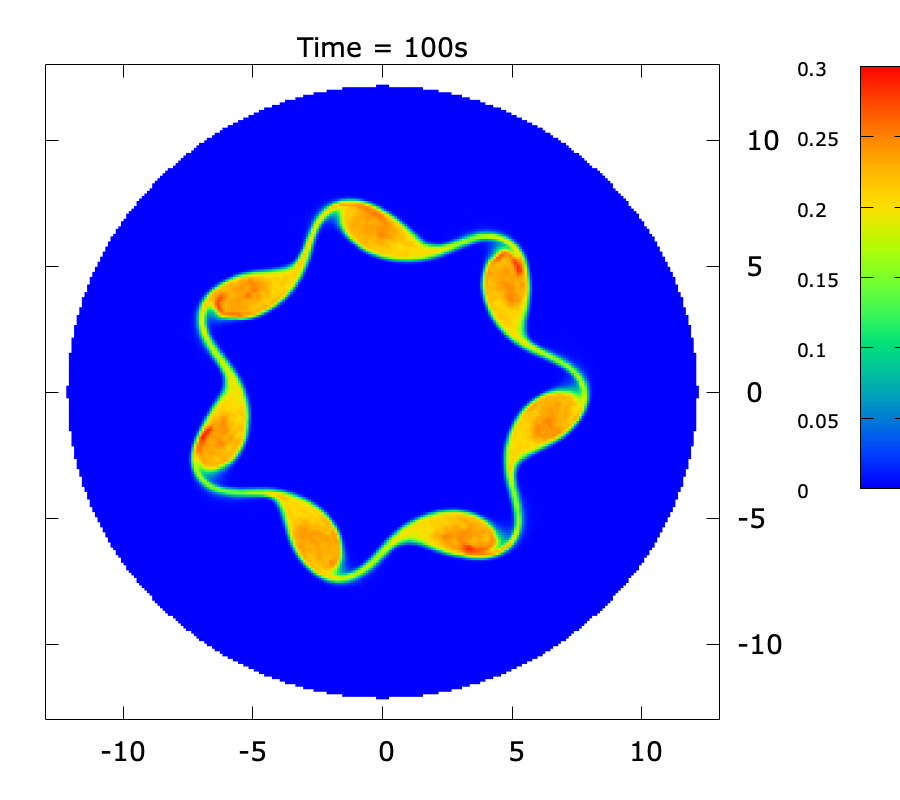}}
       	{\includegraphics[width=0.49\linewidth]{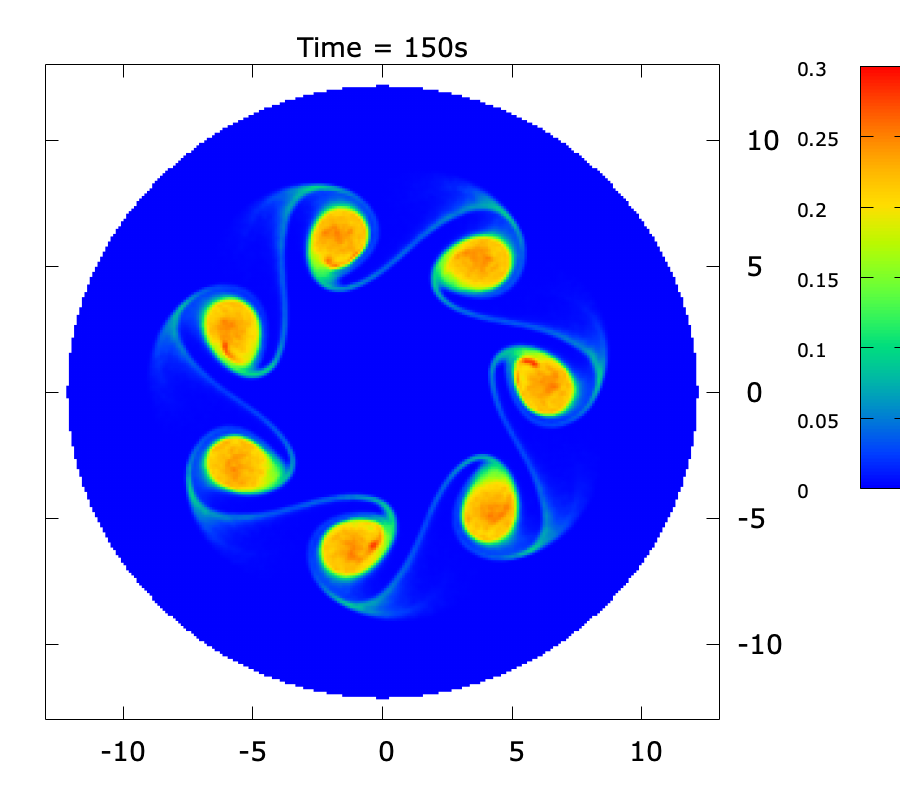}}
	\caption{\textbf{Diocotron instability:} Macroscopic density evolution at some specific time $T = 0, 50, 100, 150$ with
          $\eps=10^{-2}$ and $(\Delta
          t,\Delta \bx) = (0.1,0.1)$, using  the modified Crank-Nicolson
          scheme \eqref{scheme:modified_CN}.}
	\label{Fig:Fig8}
\end{figure}

Furthermore, we performed simulations for various $\eps\in\{10^{-2},\,
5.\,10^{-2}, \,10^{-1}\}$, shown in
Figure \ref{Fig:Fig9}. On the left hand side, we report the total
energy variations, which are theoretically of order $\Delta
t^2$. However,  when $\eps$ is small these variations decrease so
that the total energy is well preserved in the limit $\eps\to 0$. On
the right hand side, we present  the variations of  the adiabatic
invariant $\mu_\eps$, which is not  preserved by the solution to the
Vlasov-Poisson system but only by the asymptotic model \eqref{eq:gc}. Here, we notice that
this quantity oscillates with an amplitude of order
$\eps^2$.  Surprisingly,  even with a large time step $\Delta t$, the
numerical scheme \eqref{scheme:modified_CN}  is able to capture the correct amplitude. This can
also observed on the variations  of  both the potential and the
kinetic energy  in Figures \ref{Fig:Fig7}. The
slow variations definitively correspond to the effect of the  drifts
$\mathbf{E}^{\perp}/b$ and $\nabla_{\mathbf{x}}^{\perp} b/b^{2}$, but the fast oscillations and their
amplitudes are more intricate. In order to verify that these small
oscillations are not a numerical artefact, we compute a reference
solution with a small time step for $\eps=10^{-1}$ and $10^{-2}$ and
compare these  results with the those  obtained from
\eqref{scheme:modified_CN} with $\Delta t=0.1$. Figure \ref{Fig:Fig10}
clearly indicates that with such a time step, the scheme
\eqref{scheme:modified_CN} described accurately the amplitude of these
fast oscillations. However, since $\Delta t$ is much larger than the
oscillation period, the modified scheme can not describe the physical frequency.

\begin{figure}
	\centering	
    	{\includegraphics[width=0.49\linewidth]{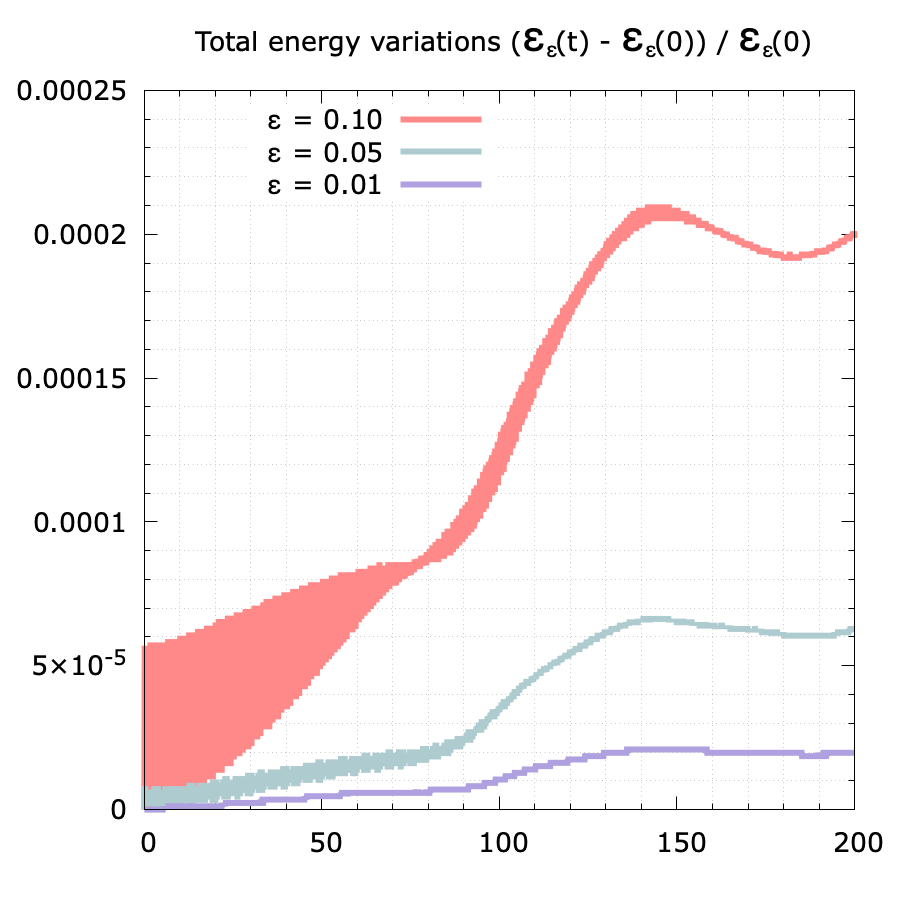}}
	{\includegraphics[width=0.49\linewidth]{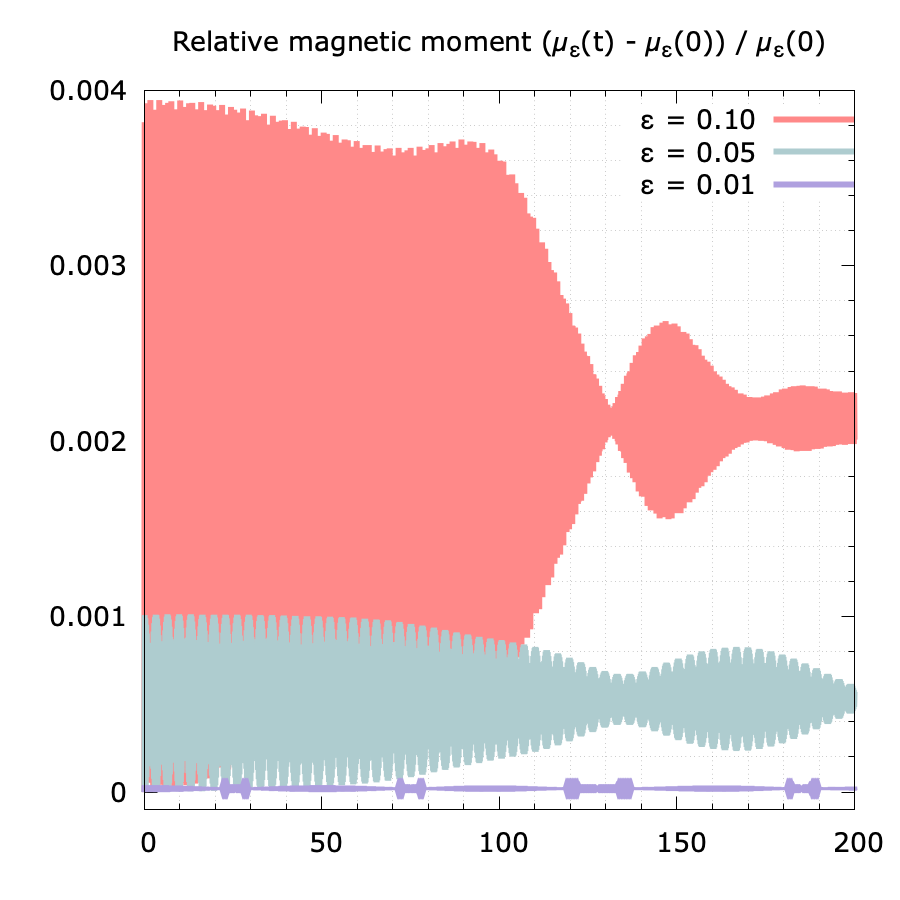}}
	\caption{\textbf{Diocotron instability:} Time evolution of the
          variations of the relative total energy $\cE_\eps$ (left)
          and magnetic moment $\mu_\eps$ (right) with different
          $\eps=10^{-1}$, $5.\,10^{-2}$ and $10^{-2}$ with   $(\Delta
          t,\Delta \bx) = (0.1,0.1)$, using  the modified Crank-Nicolson
          scheme \eqref{scheme:modified_CN}.}
	\label{Fig:Fig9}
\end{figure}

\begin{figure}
	\centering	
    	   {\includegraphics[width=0.49\linewidth]{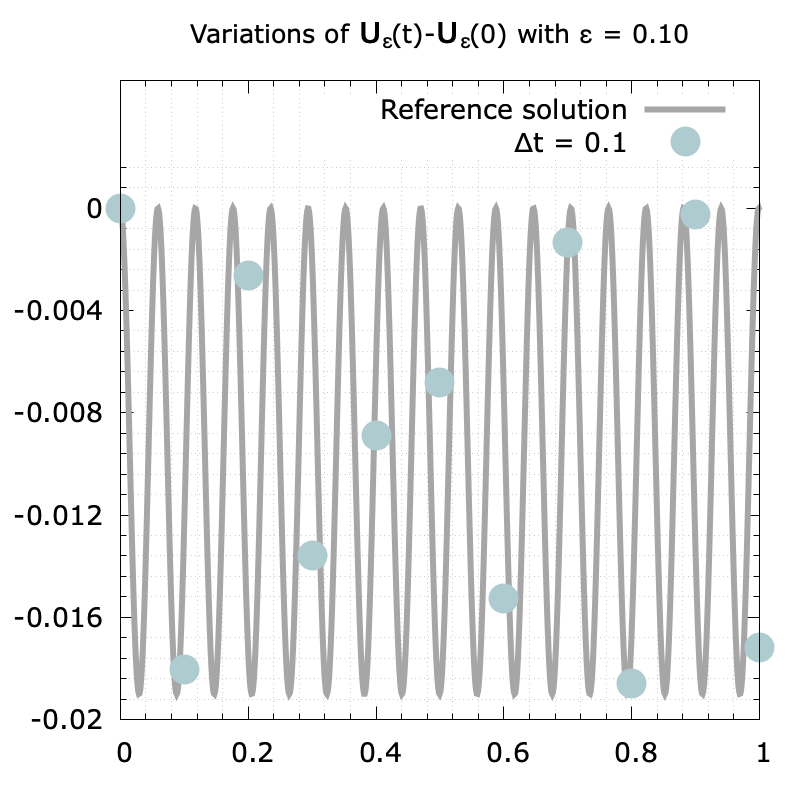}}
	   {\includegraphics[width=0.49\linewidth]{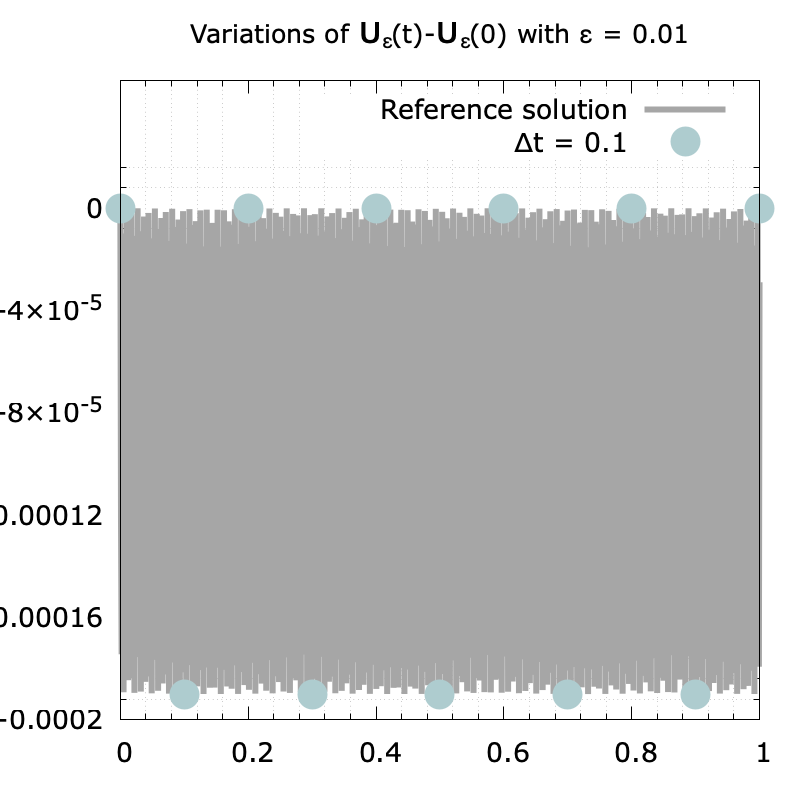}}
	\caption{\textbf{Diocotron instability:} Time evolution of the
          variation of the potential energy $\bU_\eps$ in a  short time interval
          with $\eps = 0.1$ (left) and $\eps = 0.01$ (right) using
          the modified Crank-Nicolson scheme \eqref{scheme:modified_CN}.}
	\label{Fig:Fig10}
\end{figure}

\subsubsection{Vortex interaction}
\label{Vor_DShape}
Finally we consider Vlasov-Poisson system \eqref{eq:VP_system} in a
D-Shaped domain $\Omega \subset \mathbb{R}^{2}$, described by a
mapping from polar coordinates $(r, \theta)$ to Cartesian coordinates
$\mathbf{x} = (x_{1}, x_{2})$  given by
\begin{flalign*}
	\begin{dcases}
		\ds x_{1} \,=\, a \,+\, r \,\cos \left(\theta + \arcsin(0.416)\,\sin(\theta)\right)\,, \\
		\ds x_{2} \,=\, b \,+\, 1.66 \,r\, \sin(\theta)\,,
	\end{dcases}
\end{flalign*}
centered at the origin $(a,b) = (0,0)$ where $0 < r \leq 10$ and $0
\leq \theta \leq 2 \pi$ \textcolor{red}{ as in Figure \ref{Fig:Fig65}}. \textcolor{red}{Here, we consider the
  Particle-In-Cell method with approximatively 100 particles per cell on a uniform grid} of the
rectangle $(-11,11)\times (-17,17)$ with a space discretization
$\Delta \bx = 0.1$. We choose the  initial distribution function such that 
\begin{equation}
    f(0,\mathbf{x}, \mathbf{v}) \,=\, \dfrac{5}{8 \pi^{2}} \left[ \exp\left( -\dfrac{\| \mathbf{x} - \mathbf{x}_{0} \|^{2}}{2} \right) \,+\, \exp\left( -\dfrac{\| \mathbf{x} + \mathbf{x}_{0} \|^{2}}{2} \right) \right]\,\exp \left(-\dfrac{\| \mathbf{v} \|^{2}}{2} \right),
\end{equation}
with $\mathbf{x}_{0} = (1.5, -1.5)$. Moreover, we consider a non homogeneous external magnetic field such as
\begin{equation*}
	b(t, \mathbf{x}) \,=\, \dfrac{20}{\sqrt{400 - x_{1}^{2} - x_{2}^{2}}}\,.
\end{equation*}

As expected for such a configuration, since $b$ is not
homogeneous, even in the asymptotic regime the kinetic and
potential parts of the total energy are not preserved
separately, but the total energy corresponding to the
Vlasov–Poisson system is still preserved. In addition, the
quantity $\mu_\eps$ is an invariant for the guiding center model
but oscillates with a high frequency  for the Vlasov-Poisson system with a
small amplitude. Indeed, Figure \ref{Fig:Fig11}
shows that all these features are again captured  by the  scheme
\eqref{scheme:modified_CN}  even on long time evolutions with a
large time step.   Furthermore, in Figure  \ref{Fig:Fig12},  we visualize the corresponding dynamics by
presenting several snapshots of the time evolution of the
macroscopic charge density for $\eps=10^{-2}$ at time $t = 0$,
$80$, $160$, $240$, $320$ and $400$. Since $\eps\ll
1$,  the  conservation of $e/b(\bx)$ offers coercivity
jointly in $(\bx,e)$ allowing to confine the density in the
D-shape domain. Such a confinement is indeed observed, jointly with the expected eventual merging of two initial
vortices in a relatively short time. We also observe small
filaments at low density, which  generate a "halo" propagating
into the domain as already observed in  \cite{FiYa18,FX22}. 
Finally, from Figure \ref{Fig:Fig13}, similarly to diocotron
instability experiment in Section \ref{Dio_Disk}, the relative
variations of the total energy $\cE_{\eps}$ and the adiabatic
invariant $\mu_{\eps}$ show the ability of preserving these parameters
$(\cE_{\eps}, \mu_{\eps})$ in the limit $\eps \to 0$.

\begin{figure}
	\centering	
    	{\includegraphics[width=0.49\linewidth]{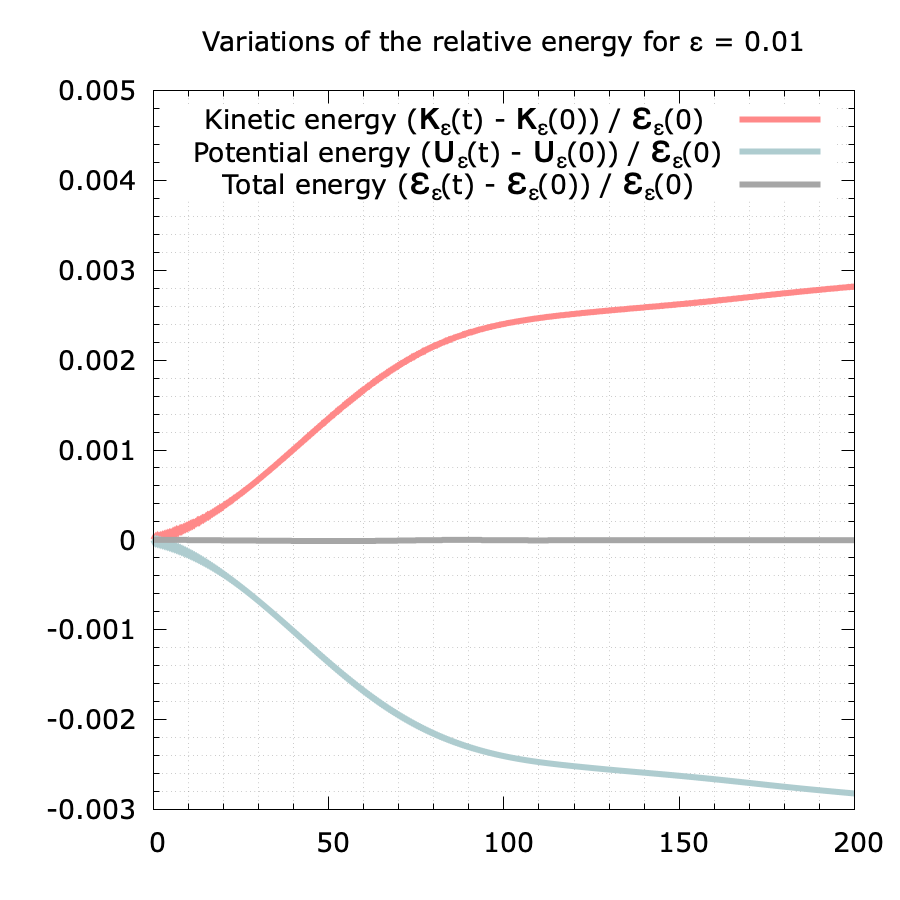}}
        {\includegraphics[width=0.49\linewidth]{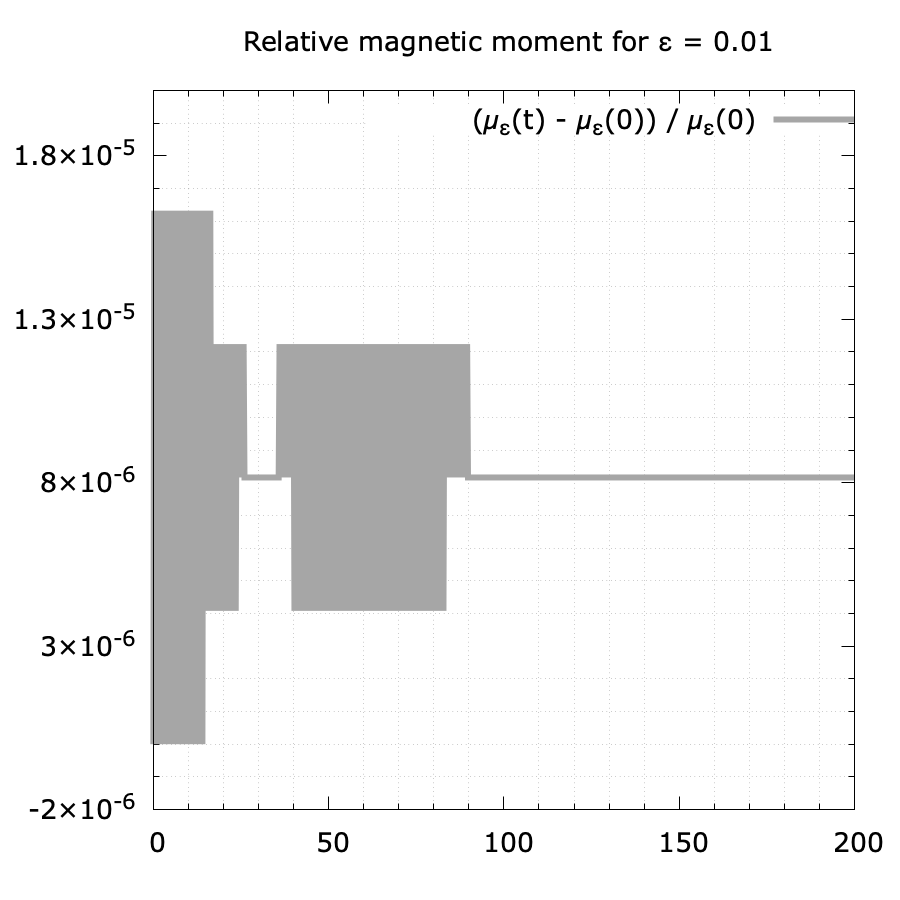}}
	\caption{\textbf{Vortex interaction:}  Time evolution of the
          variations of the relative potential $\bU_\eps$ and kinetic energy $\bK_\eps$ (left)
          and magnetic moment $\mu_\eps$ (right) with
          $\eps=10^{-2}$ with   $(\Delta
          t,\Delta \bx) = (0.1,0.1)$, using  the modified Crank-Nicolson
          scheme \eqref{scheme:modified_CN}.}
	\label{Fig:Fig11}
\end{figure}

\begin{figure}
	\centering	
    	{\includegraphics[width=0.44\linewidth]{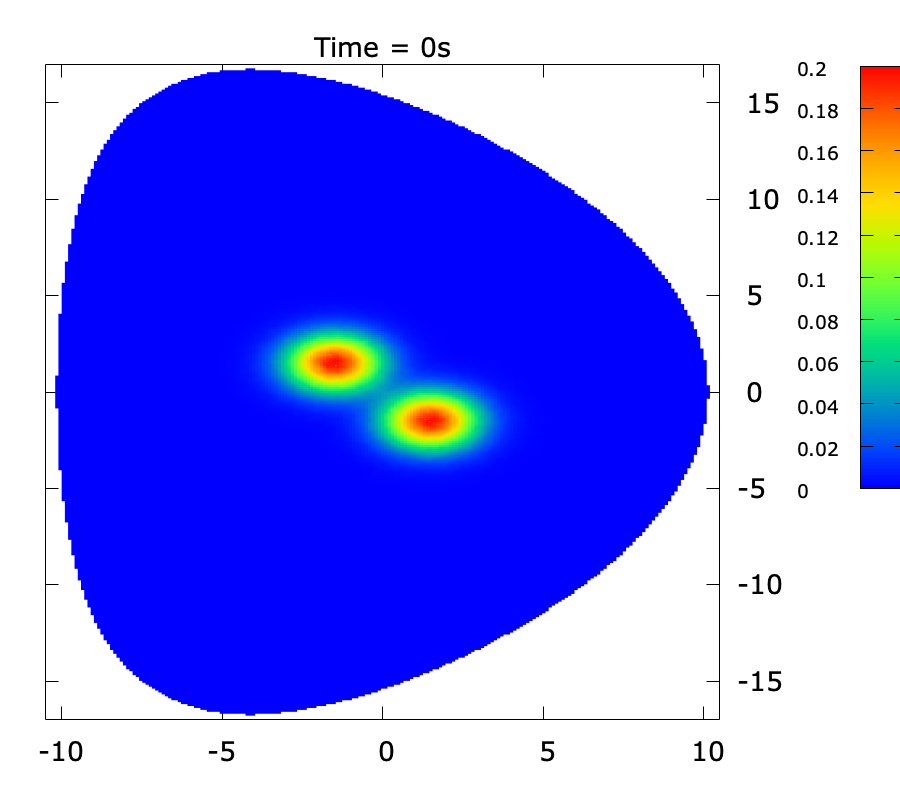}}
	{\includegraphics[width=0.44\linewidth]{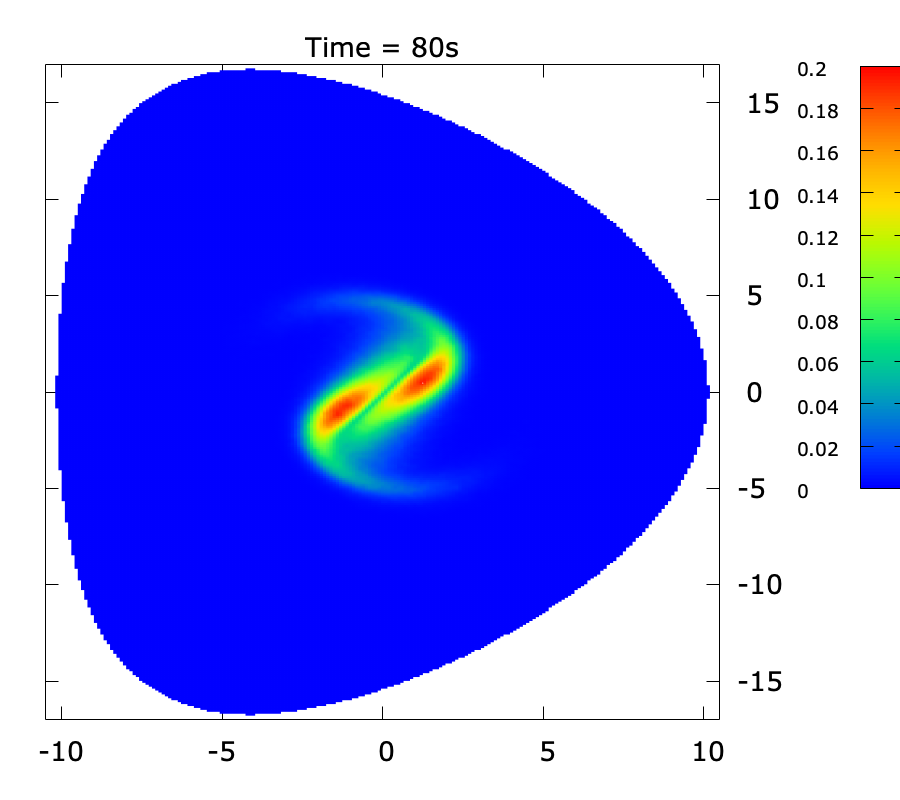}}
	 {\includegraphics[width=0.44\linewidth]{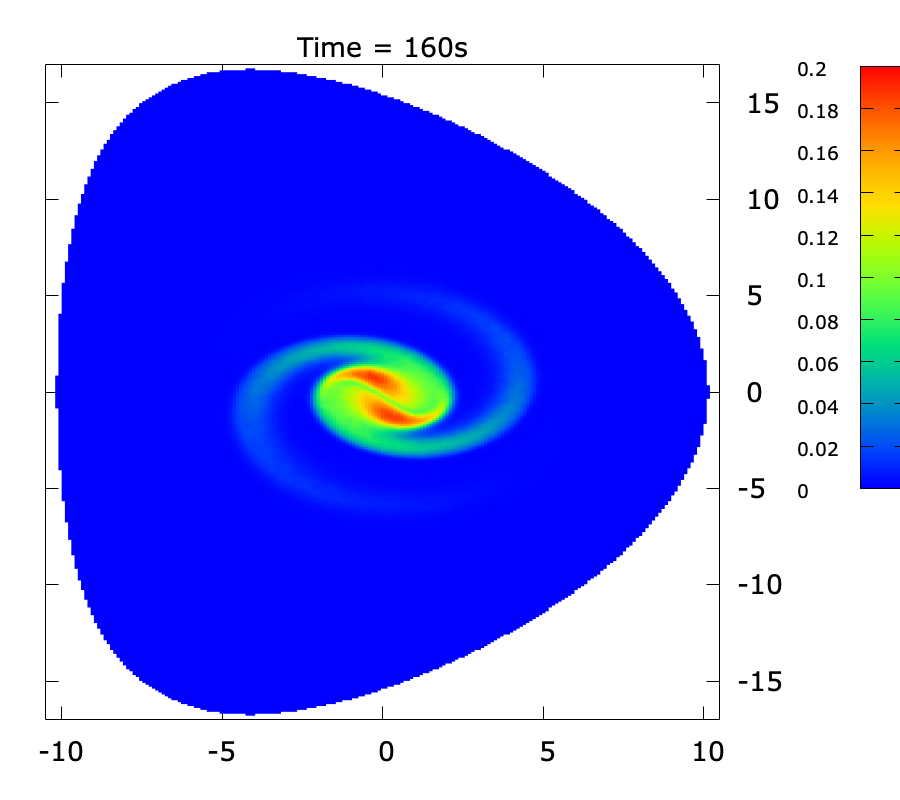}}
       	{\includegraphics[width=0.44\linewidth]{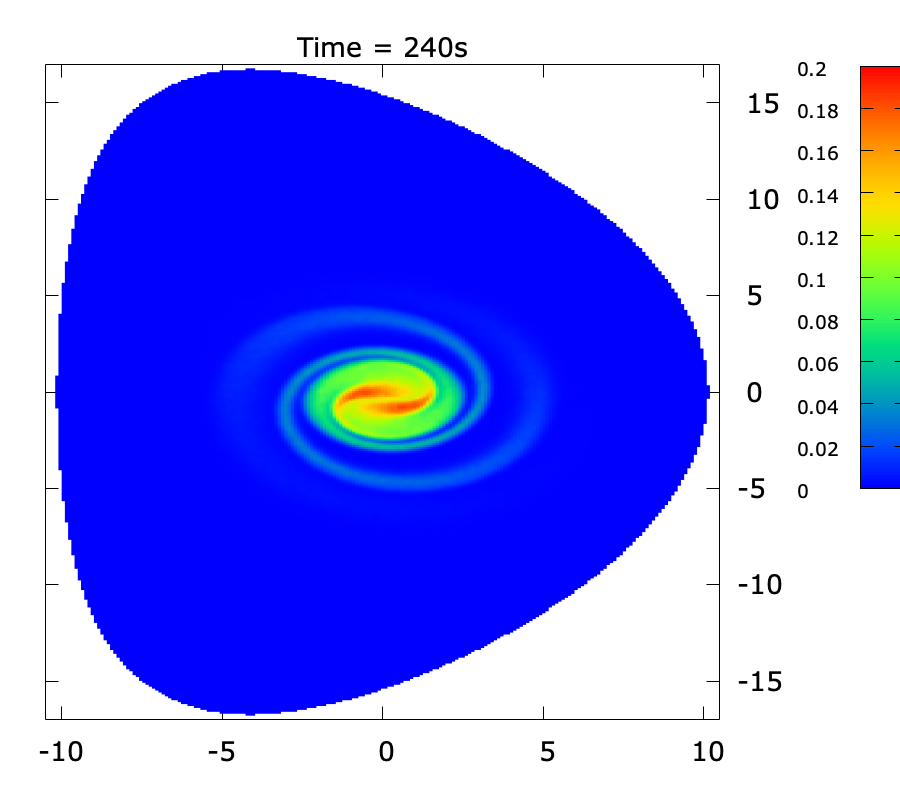}}
	 {\includegraphics[width=0.44\linewidth]{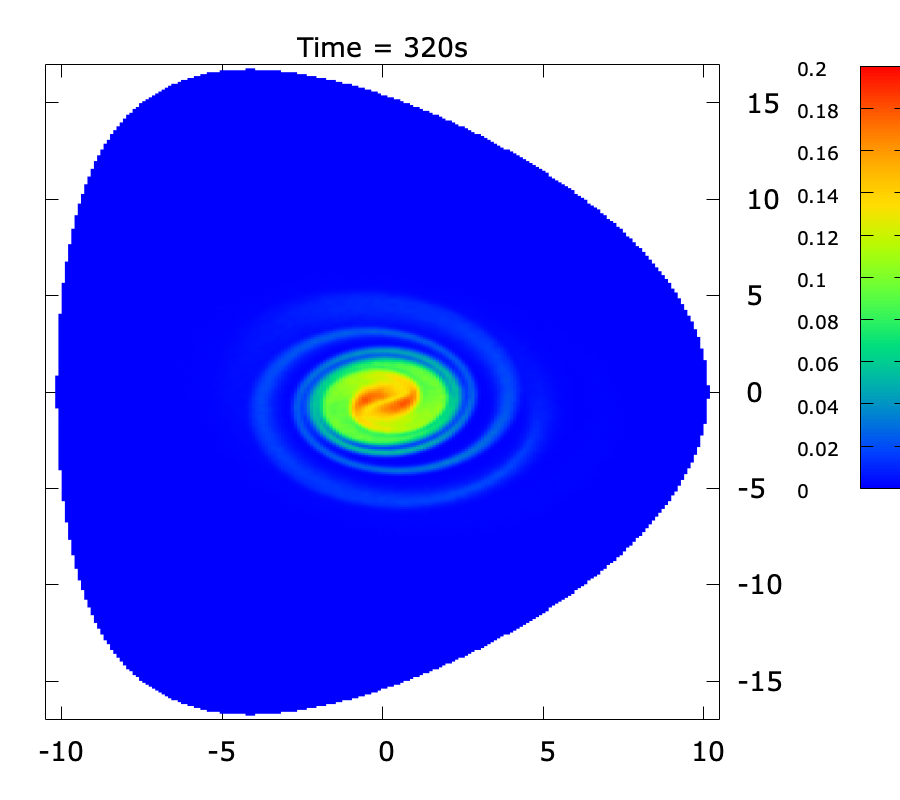}}
       	{\includegraphics[width=0.44\linewidth]{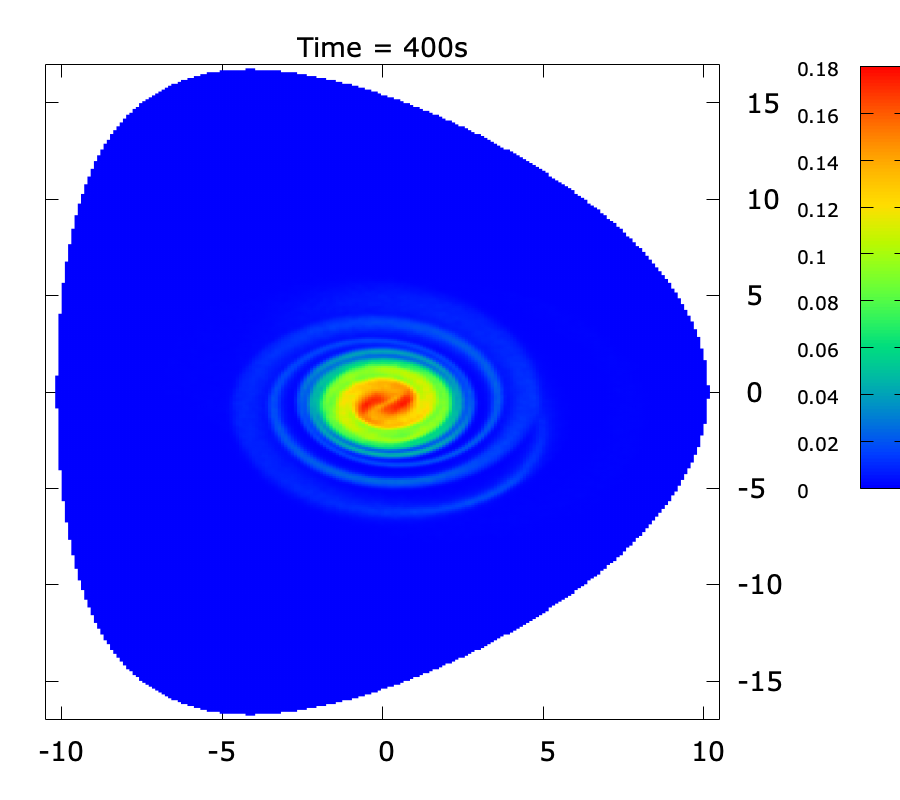}}
	\caption{\textbf{Vortex interaction:} Macroscopic density evolution at some specific time $T$ with
          $\eps=10^{-2}$ with  $(\Delta
          t,\Delta \bx) = (0.1,0.1)$, using  the modified Crank-Nicolson
          scheme \eqref{scheme:modified_CN}.}
	\label{Fig:Fig12}
\end{figure}

\begin{figure}
	\centering	
 	{\includegraphics[width=0.49\linewidth]{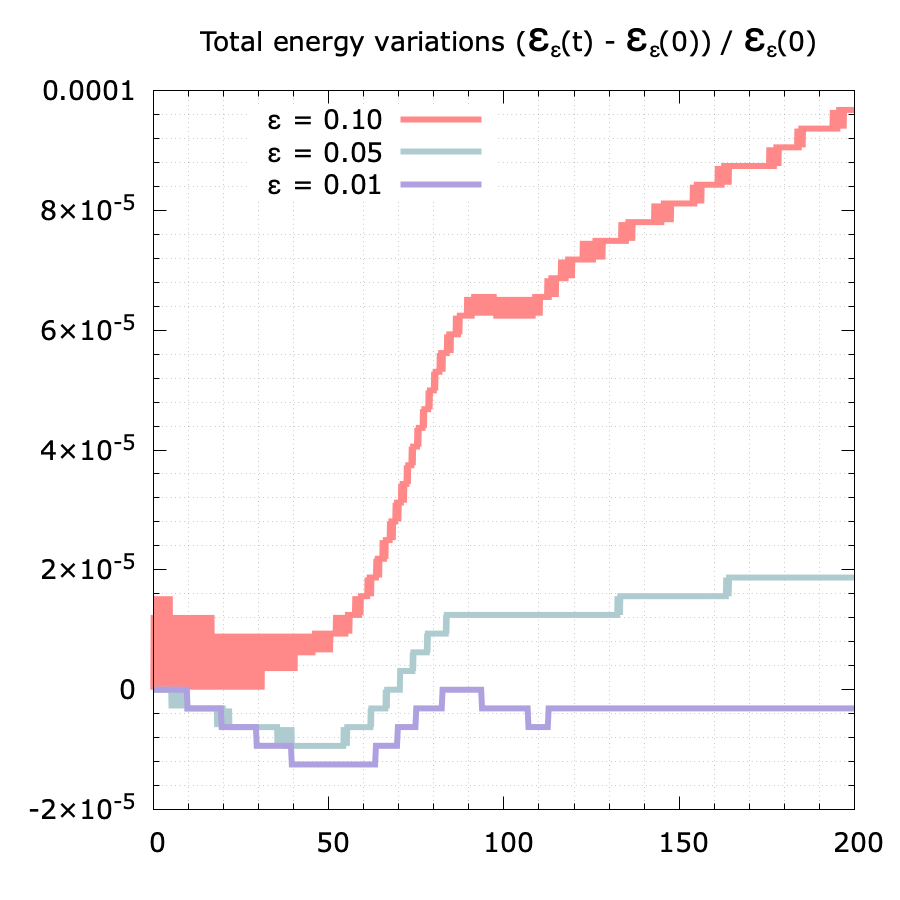}}
 	{\includegraphics[width=0.49\linewidth]{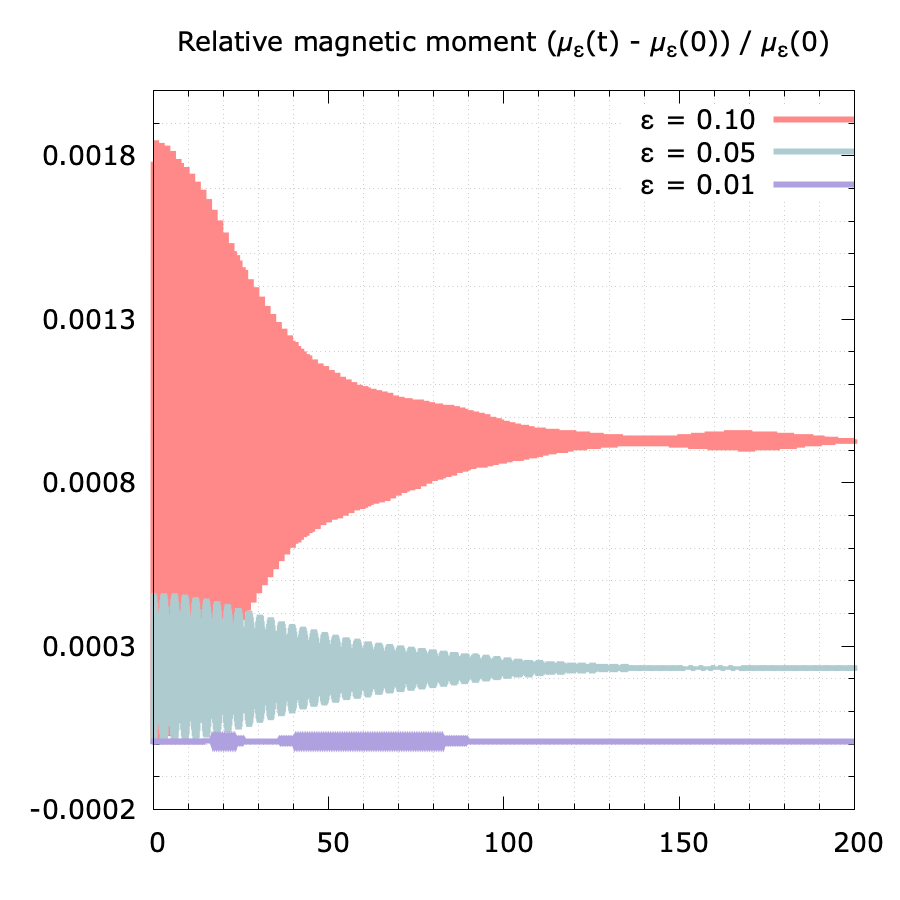}}
	\caption{\textbf{Vortex interaction:} Time evolution of the
          variations of the relative total  energy $\cE_\eps$ (left)
          and magnetic moment $\mu_\eps$ (right) with different
          $\eps=10^{-1}$, $5.\,10^{-2}$ and $10^{-2}$ and  $(\Delta
          t,\Delta \bx) = (0.1,0.1)$, using  the modified Crank-Nicolson
          scheme \eqref{scheme:modified_CN}.}
	\label{Fig:Fig13} 
\end{figure}

In summary, the modified Crank-Nicolson scheme \eqref{scheme:modified_CN} integrated into Particle-In-Cell  method shows consistency and stability of long time simulations with a coarse time step, even for small $\eps \ll 1$. Moreover, the solutions $(\cE_{\eps}, \mu_{\eps})$ of these numerical tests preserve the structure of the limit system $\eps \to 0$ and provide accurately the amplitude of these variations.

%
%

\section{Conclusion}

In this paper, we propose a modified Crank-Nicolson time
discretization technique for Particle-In-Cell simulations. Our
approach guarantees the accuracy and stability of small-scale
variables, even when the magnetic field amplitude becomes large, thus
correctly capturing their long-term behavior, including in cases of
inhomogeneous magnetic fields and coarse time grids. Comparison with
previous contributions on Crank-Nicolson and semi-implicit schemes
demonstrates the effectiveness of our approach, with accuracy
improving by several orders of magnitude. Even for large-time
simulations, the resulting numerical schemes provide acceptable
accuracy on physical invariants (total energy for all  $\eps$,
magnetic moment when $\eps\ll 1$), while fast scales are automatically filtered when the time step is large compared to $\varepsilon^2$.

As a theoretical validation, we have proven that, under certain
stability assumptions on the numerical approximations, the slow part
of the approximation converges, when $\varepsilon \to 0$, to the
solution of a limit scheme of asymptotic evolution, preserving the
initial order of precision.
\textcolor{red}{
Let us notice that here we chose to use a Crank-Nicolson scheme, which is very popular in computational physics;
however, our strategy can be easily applied to higher-order schemes \cite{FiRo16}
(such as IMEX or multi-step methods).}

\textcolor{red}{The next step involves extending this strategy to $3D$
  particle simulations, incorporating curvature effects and addressing
  both the parallel and orthogonal directions relative to the magnetic
  field. These aspects give rise to more complex phenomena; however,
  by considering the slow variables and the conservation of energy, we
  are able to apply the same approach effectively \cite{FiRo23}. A complete analytical study (including proofs of stability), in the spirit of \cite{FiRo21}, is currently under investigation \cite{FRT25}.}


\appendix 

\section{Formal asymptotic behavior for a given electromagnetic field}\label{s:app}

For the convenience of the reader, we provide here the main algebraic manipulations of equations \eqref{eq:ODE_system} leading to the guiding-center system \eqref{eq:guiding_center_system}. The system is well-known and no details of the derivation is needed in the rest of the paper but consistency with \eqref{eq:guiding_center_system} is crucially used in our evaluations of numerical schemes.

To begin with, in order to eliminate $\eps^{-1}\bv_\eps$ from \eqref{eq:augmented_ODE_system}, we observe that equation \eqref{eq:v} may be replaced with 
\begin{equation}
  \label{eq:v2}
	\eps\, \dfrac{\dD }{\dD t} \left(
          \dfrac{\mathbf{v}_{\eps}^{\perp}}{b(\mathbf{x}_{\eps})} \right)
        \,=\,
        \dfrac{\mathbf{E}^{\perp}(\mathbf{x}_{\eps})}{b(\mathbf{x}_{\eps})}
        \,-\, \left(\dfrac{\nabla_{\mathbf{x}} b
            (\mathbf{x}_{\eps})}{b^{2}(\mathbf{x}_{\eps})} \cdot
          \mathbf{v}_{\eps} \right) \,\mathbf{v}_{\eps}^{\perp} \,+\, \dfrac{\mathbf{v}_{\eps}}{\eps}\,.
\end{equation}
In order to characterize the asymptotic dynamics of the slow variables
$(\mathbf{x}_{\eps}, e_{\eps})$ when $\eps\rightarrow 0$, we notice that their equations depend
linearly on $\bv_\eps$, hence as in \cite{FiRo20} we write for all $t\in\R^+$ and any linear operator $\bL(t)$
\[
\eps \,\dfrac{\dD}{\dD t}\left( \bL(t)
  \frac{\mathbf{v}_{\eps}^\perp}{b(\bx_\eps)} \right)  \,=\,
\eps\dfrac{\dD\bL(t)}{\dD t} \,
\frac{\mathbf{v}_{\eps}^\perp}{b(\bx_\eps)}
\,+\, \bL(t)\left(
\frac{\mathbf{E}^\perp(\mathbf{x}_{\eps})}{b(\mathbf{x}_{\eps})}  \,-\, \left(\frac{\nabla_\bx
  b(\mathbf{x}_{\eps})}{b^2(\mathbf{x}_{\eps})} \cdot \bv_\eps\right)\,\bv_\eps^\perp  \,+\,  \frac{\mathbf{v}_{\eps}}{\eps}\right)\,.
\]
Therefore, applying the latter to $\mathbf{L}_{\mathbf{x}}(t) : \mathbf{u} \in \mathbb{R}^{2} \mapsto  \mathbf{u}  \in \mathbb{R}^{2}$ and $\mathbf{L}_{e}(t) : \mathbf{u} \in \mathbb{R}^{2} \mapsto \mathbf{E} (\mathbf{x}_{\eps}) \cdot \mathbf{u} \in \mathbb{R}$ and inserting the outcome into the system of slow variables \eqref{eq:augmented_ODE_system}, we obtain
\begin{equation}
  \label{eq:augmented_ODE_system2}
	\left\{\begin{array}{l}
 	\ds\dfrac{\dD}{\dD t} \left( \mathbf{x}_{\eps} - \eps\,
          \dfrac{\mathbf{v}_{\eps}^{\perp}}{b(\mathbf{x}_{\eps})} \right)
        \,=\, -
        \dfrac{\mathbf{E}^{\perp}(\mathbf{x}_{\eps})}{b(\mathbf{x}_{\eps})}
        \,+\, \left(\dfrac{\nabla_{\mathbf{x}}
            b(\mathbf{x}_{\eps})}{b^{2}(\mathbf{x}_{\eps})} \cdot
          \mathbf{v}_{\eps}\right) \, \mathbf{v}_{\eps}^{\perp},\\[1.2em]
	\ds\dfrac{\dD}{\dD t} \left( e_{\eps} - \eps\, \mathbf{E}(\mathbf{x}_{\eps}) \cdot \dfrac{\mathbf{v}_{\eps}^{\perp}}{b(\mathbf{x}_{\eps})} \right) \,=\, -\left(\mathbf{v}_{\eps}\,\nabla_{\mathbf{x}}\right) \left(\dfrac{\mathbf{E}}{b}\right)(\mathbf{x}_{\eps})\cdot\,\mathbf{v}^\perp_{\eps}\,.
	\end{array}\right.
    \end{equation}
    This latter system may replace 
    \eqref{eq:augmented_ODE_system} and is 
    not  stiff with respect to $\eps\ll 1$, but it also suggests the introduction of new variables, as
    the guiding center variable $\mathbf{x}_{\eps} - \eps\,
    {\mathbf{v}_{\eps}^{\perp}}/{b(\mathbf{x}_{\eps})} $. However,  these new equations still involve at leading order quadratic terms in $\mathbf{v}_{\eps}$ on the right hand side. Hence, following \cite{FiRo20}, we turn our attention to bilinear operators
    $\mathbf{A}(t)$ and derive
    \begin{eqnarray}
    \label{eq:bilinear}
\eps^2 \dfrac{\dD}{\dD t}\left( \mathbf{A}(t)\left(\mathbf{v}_{\eps},
          \dfrac{\mathbf{v}_{\eps}^{\perp}}{b(\bx_\eps)}\right) \right)&=& 
        \eps^2 \dfrac{\dD \mathbf{A}(t)}{\dD t} \left(\mathbf{v}_{\eps}\,,
          \dfrac{\mathbf{v}_{\eps}^{\perp}}{b(\bx_\eps)}\right) +
        \eps\mathbf{A}(t)\left(\frac{\mathbf{E}(\mathbf{x}_{\eps})}{b(\bx_\eps)}\,,
          \mathbf{v}_{\eps}^{\perp}\right)
        \\[0.9em]
      	\nonumber
      	 &+&\eps\mathbf{A}(t)\left(\mathbf{v}_{\eps}\,,
          \frac{\mathbf{E}^{\perp}(\mathbf{x}_{\eps})}{b(\bx_\eps)}\right)
                                                             -\eps\,\left(\frac{\nabla_\bx
             b(\mathbf{x}_{\eps})}{b^2(\mathbf{x}_{\eps})} \cdot
           \bv_\eps\right)\, \mathbf{A}(t)(\mathbf{v}_{\eps},
         \mathbf{v}_{\eps}^{\perp})  
         \\[0.9em]
      	\nonumber
         &-&\mathbf{A}(t)(\mathbf{v}_{\eps}^{\perp}, \mathbf{v}_{\eps}^{\perp}) +\mathbf{A}(t)(\mathbf{v}_{\eps}, \mathbf{v}_{\eps}).
      \end{eqnarray}
 Since $(\mathbf{v}_{\eps}, \mathbf{v}_{\eps}^{\perp})$ is an orthogonal basis
of $\mathbb{R}^{2}$ (when $\bv_\eps$ is non zero), introducing the operator ${\rm Tr}$ one observes that
\begin{flalign*}
\| \mathbf{v}_{\eps} \|^{2}\, {\rm Tr}(\mathbf{A}(t)) \,=\, \mathbf{A}(t)(\mathbf{v}_{\eps}^{\perp}, \mathbf{v}_{\eps}^{\perp}) \,+\, \mathbf{A}(t)(\mathbf{v}_{\eps}, \mathbf{v}_{\eps})\,.
\end{flalign*}
Therefore, one may reformulate \eqref{eq:bilinear} as 
 \begin{equation*}
        \mathbf{A}(t)(\mathbf{v}_{\eps}, \mathbf{v}_{\eps}) \,=\, \dfrac{1}{2}\, \|
        \mathbf{v}_{\eps} \|^{2} \, {\rm Tr}(\mathbf{A}(t)) \,+\, \eps^{2} \,\dfrac{\dD \kappa_{\mathbf{A}}}{\dD t}(t,\mathbf{x}_{\eps}, \mathbf{v}_{\eps}) \,+\, \eps\, \eta_{\mathbf{A}}(t,\mathbf{x}_{\eps}, \mathbf{v}_{\eps})\,,
    \end{equation*}
where $\kappa_{\mathbf{A}}$ and $\eta_{\mathbf{A}}$ are given by
\[
\left\{\begin{array}{l}
	\ds\kappa_{\mathbf{A}}(t, \mathbf{x}_{\eps}, \mathbf{v}_{\eps})  \,=\, \dfrac{1}{2}\,\mathbf{A}(t)\left( \mathbf{v}_{\eps}, \dfrac{\mathbf{v}_{\eps}^{\perp}}{b(\mathbf{x}_{\eps})} \right),\\[0.9em]
	\ds\eta_{\mathbf{A}}(t, \mathbf{x}_{\eps}, \mathbf{v}_{\eps}) \,=\,
         \dfrac{\nabla_{\mathbf{x}} b(\bx_\eps)}{2\, b^2(\bx_\eps)} \cdot \mathbf{v}_{\eps}  \,\mathbf{A}(t) \left( \mathbf{v}_{\eps},
         \mathbf{v}_{\eps}^{\perp} \right)  \,-\,
         \frac{\eps}{2}\,\frac{\dD\mathbf{A}(t)}{\dD t}\left( \mathbf{v}_{\eps}, \dfrac{\mathbf{v}_{\eps}^{\perp}}{b(\mathbf{x}_{\eps})} \right)\\[0.9em]
	\qquad\qquad\qquad\ds- \frac{1}{2}\left( \mathbf{A}(t) \left( \frac{\mathbf{E}(\mathbf{x}_{\eps})}{b(\bx_\eps)}, \mathbf{v}_{\eps}^{\perp}\right) \,+\, \mathbf{A}(t) \left(\mathbf{v}_{\eps}, \dfrac{\mathbf{E}^{\perp}(\mathbf{x}_{\eps})}{b(\mathbf{x}_{\eps})} \right)\right)\,.
\end{array}\right.
\]
To apply the latter to the bilinear maps $\mathbf{A}_{\mathbf{x}}(t) : (\mathbf{u}_{1}, \mathbf{u}_{2}) \in
\mathbb{R}^{2} \times  \mathbb{R}^{2} \mapsto  \mathbf{u}_{2}^{\perp}
\dfrac{\nabla_{\mathbf{x}} b(\mathbf{x}_{\eps})}{b^{2}(\mathbf{x}_{\eps})}
\cdot \mathbf{u}_{1}  \in \mathbb{R}^{2}$ and $\mathbf{A}_{e}(t) :
(\mathbf{u}_{1}, \mathbf{u}_{2}) \in \mathbb{R}^{2} \times  \mathbb{R}^{2}
\mapsto  (\mathbf{u}_{1}\cdot\nabla_{\mathbf{x}})\left(
  \dfrac{\mathbf{E}(\mathbf{x}_{\eps})}{b(\mathbf{x}_{\eps})} \right)\cdot
\mathbf{u}_{2}^\perp  \in \mathbb{R}$, we compute
\[
	{\rm Tr}\left( \mathbf{A}_{\mathbf{x}}(t) \right) \,=\,
        \dfrac{\nabla_{\mathbf{x}}^{\perp}
          b(\mathbf{x}_{\eps})}{b^{2}(\mathbf{x}_{\eps})} \quad{\rm and
        }\quad {\rm Tr}\left( \mathbf{A}_{e}(t) \right) = {\rm div}_{\mathbf{x}} \left( - \dfrac{\mathbf{E}^{\perp}}{b}(\mathbf{x}_{\eps})  \right)\,.
\]
Then we may replace \eqref{eq:augmented_ODE_system2} with the new
\begin{flalign} \label{eq:leading_order}
	\left\{\begin{matrix}
    & \ds\dfrac{\dD}{\dD t} \left( \mathbf{x}_{\eps} - \eps
      \,\dfrac{\mathbf{v}_{\eps}^{\perp}}{b(\mathbf{x}_{\eps})}
      - \eps^2\,\kappa_{\bx_\eps}(t, \mathbf{x}_{\eps}, \mathbf{v}_{\eps}) \right) \,=\, - \dfrac{\mathbf{E}^{\perp}(\mathbf{x}_{\eps})}{b(\mathbf{x}_{\eps})} + e_{\eps} \dfrac{\nabla_{\mathbf{x}}^{\perp} b}{b^{2}} (\mathbf{x}_{\eps}) \;+\, \eps\,  \eta_{\bx_\eps}(t, \mathbf{x}_{\eps}, \mathbf{v}_{\eps})\,, \\[0.9em]
    & \ds\dfrac{\dD }{\dD t} \left( e_{\eps} - \eps\,
      \mathbf{E}(\mathbf{x}_{\eps}) \cdot
      \dfrac{\mathbf{v}_{\eps}^{\perp}}{b(\mathbf{x}_{\eps})} +
      \eps^2\,\kappa_{e_\eps}(t, \mathbf{x}_{\eps}, \mathbf{v}_{\eps})
    \right) \,=\, e_{\eps} {\rm div}_{\mathbf{x}} \left( \dfrac{ \,\mathbf{E}^{\perp}}{b} \right)(\mathbf{x}_{\eps})- \eps \,\eta_{e_\eps}(t, \mathbf{x}_{\eps}, \mathbf{v}_{\eps})\,,
      \end{matrix}\right.
  \end{flalign}
  coupled with \eqref{eq:v} for $\bv_\eps$ and $\kappa_\alpha$ and $\eta_\alpha$, for $\alpha\in\{\bx_\eps,
\,e_\eps\}$ are short-hand for $\kappa_{\bA_\alpha}$ and $\kappa_{\bA_\alpha}$.

This last formulation easily allows to characterize  the asymptotic
limit as $\eps \rightarrow 0$ of the  slow variables
$(\mathbf{x}_{\eps}, e_{\eps})$ provided one already knows that the fast variable $\mathbf{v}_{\eps}$ remains
bounded. Note in particular that $\bE$ is curl-free thus $\mathbf{E}^{\perp}$ is divergence-free.

\bibliographystyle{abbrv}
\bibliography{refer}

\end{document}